\documentclass[12pt]{amsart}
\usepackage{adjustbox}
\allowdisplaybreaks
\usepackage{subcaption}
\usepackage{xspace}
\usepackage[T1]{fontenc}
\usepackage{graphicx}
\usepackage{graphics}
\usepackage{fancyvrb}
\usepackage{dsfont}
\usepackage{float}
\usepackage{flafter} 
\usepackage[sort,nocompress]{cite}
\usepackage[pagebackref]{hyperref}
\usepackage{xparse,aliascnt,bookmark}
\usepackage{tikz}
\usetikzlibrary{calc}
\usepackage{pgfplots}
\pgfplotsset{width=10cm,compat=newest}
\usepgfplotslibrary{external}
\usepackage{algorithm}
\usepackage{xcolor}
\usepackage{amsfonts}
\usepackage{amsmath,amssymb,enumerate,epsfig,bbm,calc,color,ifthen,capt-of}
\usepackage{mathtools}
\usepackage{hyperref}
\usepackage[T1]{fontenc}
\NewDocumentCommand{\xnewtheorem}{m o m}
 {%
  \IfNoValueTF{#2}
   {\newtheorem{#1}{#3}}
   {%
    \newaliascnt{#1}{#2}%
    \newtheorem{#1}[#1]{#3}%
    \aliascntresetthe{#1}%
    \expandafter\newcommand\csname #1autorefname\endcsname{#3}%
   }%
 
}
\usepackage{nicefrac}
\usepackage{bm}

\newtheorem{theorem}{Theorem}[section]
\xnewtheorem{corollary}[theorem]{Corollary}
\xnewtheorem{lemma}[theorem]{Lemma}
\xnewtheorem{dfn}[theorem]{Definition}
\xnewtheorem{prpstn}[theorem]{Proposition}
\xnewtheorem{note}[theorem]{Note}
\xnewtheorem{notation}[theorem]{Notation}
\xnewtheorem{remark}[theorem]{Remark}
\xnewtheorem{example}[theorem]{Example}
\xnewtheorem{conjecstate}[theorem]{Conjecture}
\xnewtheorem{conclusion}[theorem]{Conclusion}
\xnewtheorem{explanation}[theorem]{Explanation}
\xnewtheorem{definition}[theorem]{Definition}
\xnewtheorem{comment}[theorem]{Comment}
\xnewtheorem{proposition}[theorem]{Proposition}
\xnewtheorem{conference}[theorem]{Conference}
\xnewtheorem{warn}[theorem]{Warning}
\xnewtheorem{assumption}[theorem]{Assumption}
\xnewtheorem{project}{Project}
\xnewtheorem{algo}[theorem]{Algorithm}
\xnewtheorem{result}[theorem]{Result}
\xnewtheorem{patch}[theorem]{Patch}
\xnewtheorem{quesaug}[theorem]{Question}
\usepackage[english]{babel}
\usepackage{hyperref}
\usepackage[margin=1in]{geometry}
\newcommand\numberthis{\addtocounter{equation}{1}\tag{\theequation}}
\usepackage{pgfplots}
\pgfplotsset{width=15cm,compat=newest}
\usepgfplotslibrary{external}
\pgfplotsset{colormap={CM}{rgb(-500)=(.9,.45,.1) color(0)=(red) rgb255(1700)=(15,18,238)}}
\email[Current]{hsingh@uttyler.edu}
\title[A novel rkhs for data-driven discovery with limited data Acquisition]{Data-driven discovery with limited data Acquisition for fluid flow across cylinder}
\author[Singh, Himanshu]{Himanshu Singh}
\address[Visiting Assistant Professor for Academic Year Aug 2023-May 2024]{Department of Mathematics, The University of Texas at Tyler, TX-75799}
\usepackage{tabularx,booktabs}
\newcolumntype{C}{>{\centering\arraybackslash}X} 
\usepackage{longtable}
\usepackage{booktabs, siunitx}
\usepackage[tableposition=above]{caption}
\newcommand{\ess}{\operatorname{e}}
\newcommand{\Cn}{\mathbb{C}^n}

\newcommand{\bmz}{\bm{z}}
\newcommand{\koop}{\mathcal{K}_{\varphi}}
\begin{document}
\maketitle
\begin{abstract}
One of the central challenge for extracting governing principles of dynamical system via Dynamic Mode Decomposition (DMD) is about the limit data availability or formally called as \emph{limited data acquisition} in the present paper. In the interest of discovering the governing principles for a dynamical system with limited data acquisition, we provide a variant of Kernelized Extended DMD (KeDMD) based on the Koopman operator which employ the notion of Gaussian random matrix to recover the dominant Koopman modes for the standard fluid flow across cylinder experiment. It turns out that the traditional kernel function, Gaussian Radial Basis Function Kernel, unfortunately, is not able to generate the desired Koopman modes in the scenario of executing KeDMD with limited data acquisition. However, the Laplacian Kernel Function successfully generates the desired Koopman modes when limited data is provided in terms of data-set snapshot for the aforementioned experiment and this manuscripts serves the purpose of reporting these exciting experimental insights. 
This paper also explores the functionality of the Koopman operator when it interacts with the reproducing kernel Hilbert space (RKHS) that arises from the normalized probability Lebesgue measure $d\mu_{\sigma,1,\Cn}(\bmz)\coloneqq(2\pi\sigma^2)^{-n}\exp\left(-\nicefrac{\|\bmz\|_2}{\sigma}\right)dV(\bmz)$ when it is embedded in $L^2-$sense for the holomorphic functions over $\Cn$, in the aim of determining the Koopman modes for fluid flow across cylinder experiment. 
We explore the operator-theoretic characterizations of the Koopman operator on the RKHS generated by the normalized Laplacian measure $d\mu_{\sigma,1,\Cn}(\bmz)$ in the $L^2-$sense. In doing so, we provide the \emph{compactification \& closable} characterization of Koopman operator over the RKHS generated by the normalized Laplacian measure in the $L^2-$sense. 
\end{abstract}
\section{Introduction}
\subsection{Dynamical Systems \& Hilbert space}
Dynamical systems provides the mathematical framework for the understanding of the physical reality in which we are preoccupied. Since about the conceptualization of the dynamical systems, the ability to characterize or achieving the best-possible prediction for the future state-variables is the key challenge faced by various scientific practitioners across the field of engineering and significantly others. However, modern mathematical framework which includes machine learning algorithms such as reduced-ordered modeling has now indeed, grown at such a fast pace, especially in the past few decades, that these techniques are now regarded as the cornerstone to develop the data-driven features of the dynamical systems.
\begin{align}\label{eq_1}
    \underbrace{\frac{d}{dt}\mathbf{x}(t)=\mathbf{f}\left(\mathbf{x}(t)\right)}_{\textsc{Dynamical System}}~\overset{\textsc{Discretization}}{\implies}~\mathbf{F}_t\left(\mathbf{x}(t_0)\right)=\mathbf{x}\left(t_0\right)+\int_{t_0}^{t_0+t}\mathbf{f}\left(\mathbf{x}(\mathfrak{t})\right)d\mathfrak{t}.
\end{align}
Machine learning architectures in artificial intelligence, in particular, \emph{deep learning and neural networks (NN)} usually offer a competitive platform to simulate and forecast complex, chaotic and non-linear dynamical systems as demonstrated in \autoref{fig:NNxyz} for the Lorenz dynamical system (L63, \cite{lorenz1963deterministic}) \autoref{fig:Lorentz}.
\begin{figure}[H]
    \centering
    \includegraphics[scale=.5]{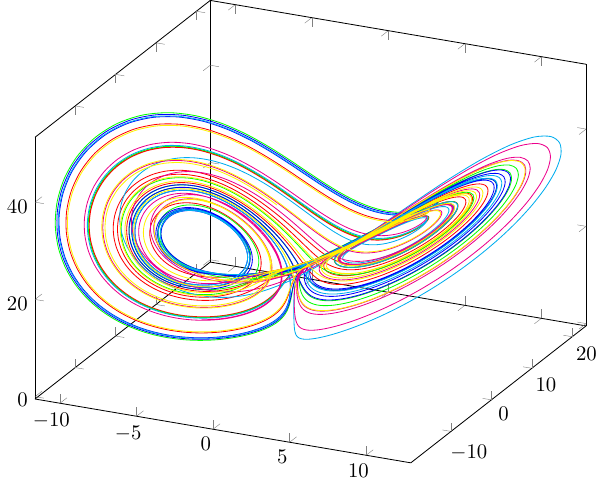}
    \caption{L63 dynamical system}
    \label{fig:Lorentz}
\end{figure}
\begin{figure}[H]
    \centering
    \includegraphics[scale=.5]{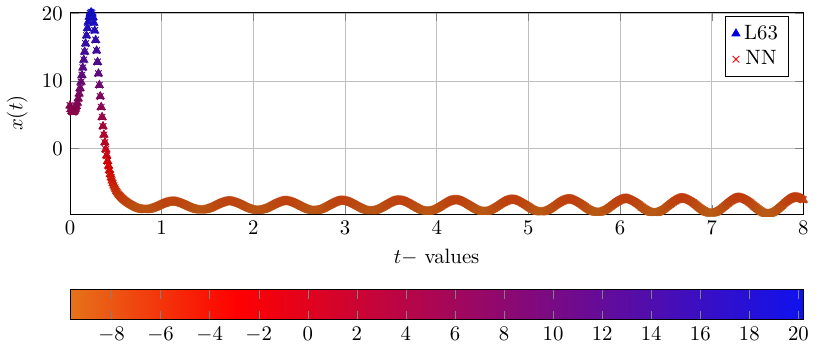}
    \includegraphics[scale=.5]{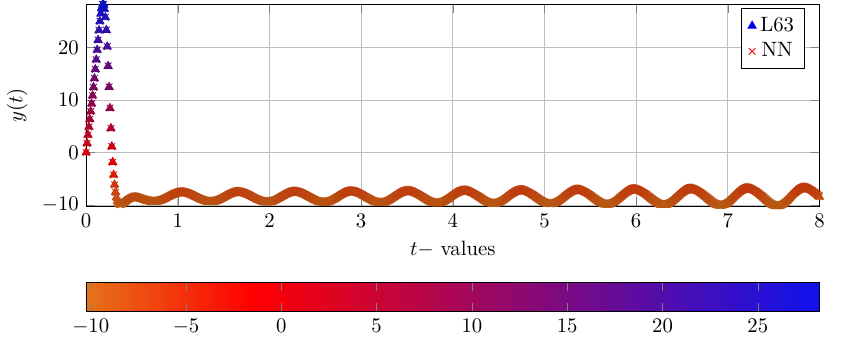}
    \includegraphics[scale=.5]{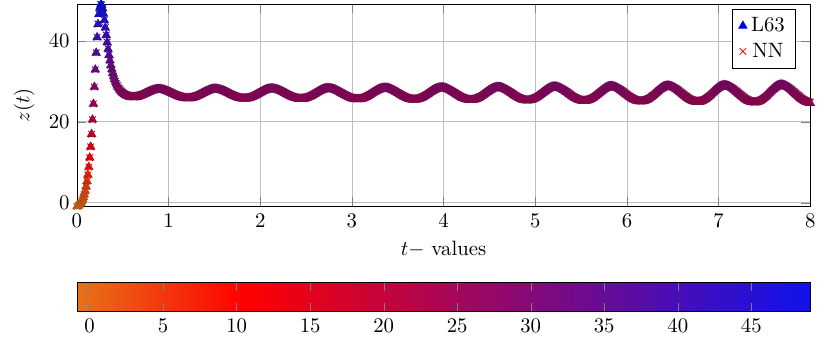}
    \caption{NN Simulation of L63 in $x(t),~y(t)$ and $z(t)-$coordinates}
    \label{fig:NNxyz}
\end{figure}
However, deep learning architecture faces serious issues for instant dynamics changing context as they lack the tendency to operate or evolve with respect to \emph{time} \cite{mezic2021koopman,sarker2021deep}. The scope of discovering intrinsic information from the dynamical systems evolving with respect to time perfectly aligns with the reduced-ordered modeling techniques via the operator-theoretic approaches when intertwined with the theory of Hilbert spaces generated by the reproducing kernels \cite{aronszajn1950theory} RKHS. We define RKHS followed by the basic definition of \emph{kernel function}.
\begin{dfn}[Kernel Function]
    Let $X=\emptyset$, then a function $k:X\times X\to\mathbf{K}$ is called the kernel on $X$ if there exists a $\mathbf{K}-$Hilbert space $\left(H,\langle\cdot,\cdot\rangle_H\right)$ accompanied by a map $\Phi:X\to H$ such that $\forall x,x'\in X$, we have
    \begin{align}
        k(x,x')=\langle\Phi(x'),\Phi(x)\rangle_H.\label{eq_kernel_H}
    \end{align}
    We regard $\Phi$ as the feature map and $H$ as the feature space of $k$.
\end{dfn}
\begin{example}
    The most common kernel function that arises from the class of radial basis function\footnote{
        A function $\Phi:\mathbb{R}^n\to\mathbb{R}$ is called as \emph{radial} if there exists a \emph{univariate} function $\phi:\{0\}\cup\mathbb{R}_+\to\mathbb{R}$ such that 
        \begin{align*}
            \Phi\left(\bm{x}\right)=\phi(r),~\text{where $r=\|\bm{x}\|$}.
        \end{align*}Here, $\|\cdot\|$ is some norm defined on $\mathbb{R}^n$, usually $\|\cdot\|_2$ which is \emph{Euclidean norm}.
    } (cf. \cite{fasshauer2007meshfree}) used in modern machine learning and artificial intelligence routine such as speech enhancement is the class of \emph{exponential power kernels} \cite{giraud2005positive,hui2018kernel} given as 
\begin{align}\label{eq_expkernel}
    K_{\text{exp}}^{\gamma,\sigma}(\bm{x},\bm{z})\coloneqq\exp\left(-\frac{\|\bm{x}-\bm{z}\|_2^\gamma}{\sigma}\right);~\text{where $\gamma,\sigma>0$ and $\bm{x},\bm{z}\in\mathbb{C}^n$}.
\end{align}
The constants present in \eqref{eq_expkernel}, that is, $\sigma$ is referred as \emph{kernel bandwidth} and $\gamma$ is often called as \emph{shape parameter}. If $\gamma=1$ then we get $K_{\text{exp}}^{1,\sigma}(\bm{x},\bm{z})\coloneqq\exp\left(-\frac{\|\bm{x}-\bm{z}\|_2}{\sigma}\right)$, which is referred as \emph{Laplace Kernel}. If $\gamma=2$ in \eqref{eq_expkernel}, we get $K_{\text{exp}}^{2,\sigma}(\bm{x},\bm{z})\coloneqq\exp\left(-\frac{\|\bm{x}-\bm{z}\|_2^2}{\sigma}\right)$, which is commonly referred as the \emph{Gaussian Radial Basis Function (GRBF) Kernel}. We direct interested readers \cite{rasmussen2006gaussian} to learn more about the feature map of GRBF Kernel. It should be noted that the choice of kernel functions can dramatically change the performance of the (supervised) machine learning routine \cite{geifman2020similarity,NEURIPS2020_1006ff12}, specifically in those situation when shorter training time or limited information is available.
\end{example}
We are now ready to define what RKHS is along with its necessary details and examples.
\begin{dfn}[Reproducing Kernel Hilbert Space]\label{dfn_RKHS}
    Let $X=\emptyset$ and $\left(H,\langle\cdot,\cdot\rangle_H\right)$ be the Hilbert function space over $X$.
    \begin{enumerate}
        \item The space $H$ is called as the \emph{\textbf{reproducing kernel Hilbert space}} (RKHS) if $\forall x\in X$, the evaluation functional $\mathcal{E}_x:H\to\mathbf{K}$ defined as $\mathcal{E}_x(f)\coloneqq f(x),~f\in H$ is continuous.
        \item A function $k:X\times X\to\mathbf{K}$ is called \emph{reproducing kernel} of $H$ if we have:
        \begin{enumerate}
            \item $k(\cdot,x)\in H~\forall x\in X$, that is $\|k(\cdot,x)\|_H<\infty$, and
            \item $k(\cdot,\cdot)$ has the reproducing property; that is
          \begin{align*}
                f(x)=\langle f,k(\cdot,x)\rangle_H~\forall f\in H\text{~and~} x\in X.
            \end{align*}
        \end{enumerate}
        \end{enumerate}
\end{dfn}
It is worth-full to mention that the norm convergence yields the point-wise convergence inside RKHS. This fact can be readily learned due to the continuity of evaluation functional. This is demonstrated as follows for an arbitrary $f\in H$ and $\left\{f_n\right\}_n\in H$ with $\|f-f_n\|_H\to0$ as $n\to\infty$, then
\begin{align*}
\lim_{n\to\infty}f_n(x)=\lim_{n\to\infty}\mathcal{E}_x\left(f_n\right)
    =_{\text{(continuity of $\mathcal{E}_x$)}}\mathcal{E}_x\left(f\right)
    =f(x).
\end{align*} 
\begin{theorem}[\textsc{Moore-Aronszajn} Theorem \cite{aronszajn1950theory}]\label{theorem_aronsjan}
    Let $H$ be an RKHS over an nonempty set $X$, Then $k:X\times X\to\mathbf{K}$ defined as $k(x,x')\coloneqq \langle\mathcal{E}_x,\mathcal{E}_{x'}\rangle_H$ for $x,x'\in X$ is the only reproducing kernel of $H$. Additionally, for some index set $\mathcal{I}$, if we have $\left\{\mathbf{e}_i\right\}_{i\in\mathcal{I}}$ as an orthonormal basis then for all $x,x'\in X$, we have
    \begin{align}\label{eq_5theorem1.4}
        k(x,x')=\sum_{i\in\mathcal{I}}\mathbf{e}_i(x)\overline{\mathbf{e}_i(x')},
    \end{align}
    with an absolute convergence.
    \end{theorem}
\begin{example}
In this example we present the RKHS for the GRBF Kernel by considering a holomorphic function $f:\mathbb{C}^n\to\mathbb{C}$. Then, we define the norm as
\begin{align}\label{eq_6GRBFfunctionspace}
    \|f\|_{\sigma}^2\coloneqq\frac{2^n\sigma^{2n}}{\pi^n}\int_{\mathbb{C}^n}|f(\bm{z})|^2e^{\sigma^2\sum_{i=1}^n\left(z_i-\overline{z}_i\right)^2}dV(\bm{z}),
\end{align}
where $dV(\bm{z})$ is the usual Lebesgue volume measure on $\mathbb{C}^{n}\equiv\mathbb{R}^{2n}$. The RKHS for $K_{\text{exp}}^{2,\sigma}(\bm{x},\bm{z})$ is given as follows:
\begin{align}\label{eq_7GRBFHilbertspace}
    H_{\sigma}\coloneqq\left\{f:\mathbb{C}^n\to\mathbb{C}:\text{$f$ is holomorphic and $\|f\|_{\sigma}<\infty$}\right\}.
\end{align}
We appeal interested readers to follow \cite{steinwart2006explicit,steinwart2008support} more for the RKHS on GRBF Kernel.
\end{example}
If, for instance, let $k_{x'}(x)=k(x,x')$ in \eqref{eq_5theorem1.4}, then we have an important result for the RKHS in terms of a \emph{weakly converging sequence}. As a matter of fact, the sequence is weakly convergent if and only if it is bounded in norm and it converges point-wise. Upon the use of this, we can formulate following lemma.
\begin{lemma}[Lemma 2.4 in \cite{le2017composition}]\label{lemma_weaklyconverging}
    Let $H$ be the RKHS as defined in \autoref{dfn_RKHS} over $\mathbb{C}^n$ and let $k_x$ be its reproducing kernel. Then, following holds for the RKHS $H$:
    \begin{enumerate}
        \item $\lim_{\|x\|\to\infty}\nicefrac{k_x}{\|k_x\|}=0$;
        \item let $\nu_M$ be a sequence converging weakly to $0$ in $\mathbb{C}^n$ (that is, in particular, $\nu_M$ is bounded). For each $M$, put $f_M(\bmz)=\langle\bmz,\nu_M\rangle$ for $\bm{z}\in\mathbb{C}^n$, then $\lim_{M\to\infty}f_M=0$ weakly in $H$.
    \end{enumerate}
\end{lemma}
It should be noted that for the RKHS $H$ with reproducing kernel $k_x$ and its reproducing property (\textsc{rp}), we have following result:
\begin{align}\label{eq_4563}
    \|k_x\|^2=\langle k_x,k_x\rangle_H\overset{\textsc{rp}}{=}k_x(x)=k(x,x).
\end{align}
\subsection{Interface between Operator theory \& Dynamical System}
Koopman operators or \emph{Composition operators} theory by \textsc{Bernard Koopman} in 1930's provides an alternative mathematical framework at the operator theory level to understand the complex high-dimensional systems. At-least at the historical perspective, this operator theoretic framework was majorly borrowed by the traditional \emph{composition operators} approaches when it interacts with Hilbert spaces. It comes to the very strange understanding of the fact that the desire of determining the eigen-observables of the composition operators is as old as early 1840's due to \textsc{Ernst Schr\"{o}der}'s work in \cite{schroder1870ueber} but it was \textsc{Koopman}'s contribution in \cite{koopman1931hamiltonian} which made \emph{composition operators} synonymous to \emph{Koopman operators} that we now know of today.

To this end, we define the Koopman operator as follows followed immediately by the following necessary assumption.
\begin{assumption}[Sampling-flow assumption] Let $\mathcal{M}$ be a metric space and $\mathbf{F}_t:\mathcal{M}\to \mathcal{M}$ be the flow as defined in \eqref{eq_1} along with the Borel-probability measure $\mu$ whose support $\operatorname{supp}\mu=X$. Let the system be sampled at a fixed-time-instant, say $\Delta t~(>0)$ such that $\mathbf{F}_{n\Delta t}:\mathcal{M}\to \mathcal{M}$.
\label{assumption_1}\end{assumption}

\begin{dfn}[Koopman Operators]\label{def_KoopmanOperators}
In the light of action of \autoref{assumption_1}, the obvious choice of observables of dynamical system is the Lebesgue-square integrable functions with respect to the measure $\mu$, that is $L^2\left(\mu\right)$. $L^2\left(\mu\right)$ is the Hilbert space equipped by the inner product as $\langle f,g\rangle_{L^2\left(\mu\right)}=\int_Xf(\cdot)g^*(\cdot)d\mu(\cdot)$. The other natural selection for the Hilbert space can be the RKHS. The dynamical flow $\mathbf{F}_t$ induces a linear map $\mathcal{K}_{\mathbf{F}_t}$ on the vector space of complex-valued functions on $\mathcal{M}$ and on $X$ defined as 
\begin{align}\label{eq_9r}
    \mathcal{K}_{\mathbf{F}_t}:L^2\left(\mu\right)\to L^2\left(\mu\right)\overset{\textsc{via}}{\implies}\mathcal{K}_{\mathbf{F}_t}g\coloneqq\underbrace{ g\circ\mathbf{F}_t}_{\textsc{Composition}}.
\end{align}
We can also provide the continuous-time infinitesimal generator of the Koopman operator family for the given dynamical system \eqref{eq_1} as 
\begin{align}\label{eq_6}
A_{\mathbf{f}}g\coloneqq\lim_{t\to0}\frac{\mathcal{K}_{\mathbf{F}_t}g-g}{t}=\lim_{t\to0}\frac{g\circ\mathbf{F}_t-g}{t}.
\end{align}
However, the time derivative of $g$ in the direction of trajectories $\mathbf{x}(t)$ of dynamical system in $\eqref{eq_1}$ yields 
\begin{align}\label{eq_7}
    \frac{d}{dt}g\left(\mathbf{x}(t)\right)=\nabla g\cdot\dot{\mathbf{x}}(t)=\nabla g\cdot\mathbf{f}\left(\mathbf{x}(t)\right),
\end{align}
which if we equate in the following way
\begin{align*}
    \overbrace{\nabla g\left(\mathbf{x}(t)\right)\cdot\mathbf{f}\left(\mathbf{x}(t)\right)=\frac{d}{dt}g\left(\mathbf{x}(t)\right)}^{\text{\eqref{eq_7}}}&=\underbrace{\lim_{\mathfrak{t}\to0}\frac{g\left(\mathbf{x}(t+\mathfrak{t})\right)-g\left(\mathbf{x}(t)\right)}{\mathfrak{t}}=A_{\mathbf{f}}g\left(\mathbf{x}(t)\right)}_{\text{\eqref{eq_6}}}\\
    \therefore A_{\mathbf{f}}g&\overset{\text{def}}{\coloneqq}\nabla g\cdot\mathbf{f}.\numberthis\label{eq_liouvilleoperator}
\end{align*}
\end{dfn}
The operator `$A_{\mathbf{f}}$' in \eqref{eq_liouvilleoperator}
 is called as the \emph{Liouville operator} and is formally defined as follows over the underlying Hilbert space, in particular RKHS $H$.
 \begin{dfn}[Liouville Operator]
     Consider the dynamical system in \eqref{eq_1} where $\mathbf{f}:\mathbb{R}^n\to\mathbb{R}^n$ is Lipschitz continuous. Let $H$ be a RKHS over the non-empty compact set $X\subset\mathbb{R}^n$. Then, the Lipschitz continuous dynamics $\mathbf{f}$ induces a linear map $A_{\mathbf{f}}$ with its natural symbol as $\mathbf{f}$ defined as
     \begin{align*}
         A_{\mathbf{f}}:\mathcal{D}\left(A_{\mathbf{f}}\right)\to H&\overset{\textsc{via}}{\implies}A_{\mathbf{f}}g=\nabla g\cdot\mathbf{f},\quad{\text{where}}\\
         \mathcal{D}\left(A_{\mathbf{f}}\right)&\coloneqq\left\{g\in H:\nabla g\cdot\mathbf{f}\in H\right\}.
     \end{align*}
 \end{dfn}
The Liouville operator was introduced by \textsc{Rosenfeld} and his collaborators in the year 2019 in \cite{rosenfeld2019occupation} to understand the complex theory of system identification. The main difference between the Liouville operator and the Koopman operator is how the dynamics get encapsulated by these operators. As such, it is worthwhile to mention that the Liouville operator $A_{\mathbf{f}}$ directly encapsulate the dynamics, on the other hand, the Koopman operator $\mathcal{K}_{\mathbf{F}}$ encapsulate the flow of the dynamics which needs the dynamics to be discretizable.

In what follows, a formal relationship between the eigen-observables of Koopman operator with the possibility of model reduction for high-dimensional dynamical systems was devised by \textsc{Mezic} in \cite{mezic2004comparison,mezic2005spectral}. In the interest of performing reduced ordered modelling via either \emph{Proper Orthogonal Decomposition} or \emph{Principal Component Analysis} for a high-dimensional dynamical systems such as turbulent flow, extracting the spatio-temporal modes of the system remains a central challenge since either of these two techniques are unable to preserve these modes as the system evolves with respect to time. Hence, realizing this critical issue, it led to the development of physical and mathematical framework of \emph{Dynamic Mode Decomposition} (DMD). 

DMD was formulated by \textsc{Schmid} in \cite{schmid2010dynamic} along with its other parallel variants in \cite{schmid2011application,schmid2022dynamic} etc. to address the issue for the identification of spatio-temporal coherent structures for high-dimensional time series data in the fluid community area, in particular cavity flow and jet flow. Contemporary to the work of \textsc{Schmid} in \cite{schmid2010dynamic} in 2009, around the same period,  \textsc{Rowley} and \textsc{Mezic} together with their collaborators in \cite{rowley2009spectral} carved out indispensable connections between DMD and the spectral-observables of Koopman operators and demonstrated their results on a jet in cross-flow. DMD with other of its variants such as \emph{DMD with control} (DMDc) \cite{proctor2016dynamic} or \emph{multi-resolution DMD} (mrDMD) \cite{kutz2016multiresolution} finds extensive applications in characterizing epidemiological systems and fluid turbulence respectively. 

Collectively, the notion of casting the non-linear dynamical systems into the Koopman operator, indeed allows DMD to accurately characterize periodic/quasi-periodic behavior given that \emph{enough acquisition of data is ensured beforehand}. When the dynamical system is initiated by the large set of state variables, then Kernelized \emph{extended DMD} (KeDMD)  \cite{MatthewO.Williams2015JournalofComputationalDynamics} methods are invoked which are motivated by the \emph{kernel-trick}  \cite{scholkopf2000kernel} of some appropriatly chosen RKHS, which intrinsically approximates the DMD measurements via the evolution operator $\bm{A_Y}$ in $\bm{y}_{k+1}=\bm{A_Y}\bm{y}_{k}$ (cf. \cite[Page 327]{brunton2022data}). Looking ahead to the spectral convergence of the Koopman operators to identify the governing features of the complex dynamical systems via DMD, the convergence exhibits in the \emph{strong operator topology} (SOT) \cite{korda2018convergence} which is significantly similar as the \emph{point-wise convergence} \cite[Chapter 13]{halmos2012hilbert}. However, in the pursuit of achieving better convergence then SOT, one can identify the spectral information of the Liouville operators which provide the \emph{norm convergence} \cite{rosenfeld2022dynamic,rosenfeld2023singular,rosenfeld2021theoretical,morrison2023dynamic,gonzalez2023modeling,rosenfeld2021occupationACC,kamalapurkar2021occupation,russo2022liouville}. 
\subsection{What this paper offers?}Following is the practical utility of the present paper:
\subsubsection{Motivation for this paper}
Modern data science is unarguably going through a phase which can be easily called as a data revolution, where scientific machine learning algorithms such as DMD and its related variants are helping in creating data-driven models to comprehend the understanding of complex dynamical system. As such, we have already learned the importance of DMD in various scientific and engineering field as presented in before in this paper with obvious important references included, however, we also simultaneously see that we lack a \emph{proper and robust framework} which can execute the same if enough data-set or \emph{enough number of snapshot} is not present. Lack of investigation in this direction motivates the need of this paper, which provide the solution to this problem by taking the advantage of Kernelized eDMD and random matrix theory.
\subsubsection{Offerings of this paper}
We immediately provide the offerings of the current work which is in the direction of developing a novel methodological setup to execute the Kernelized eDMD but with limited data acquisition. Here we describe the chief contributions of this paper:
\begin{enumerate}
    \item This paper studies the operator theoretic interactions of the Koopman operators over the RKHS of holomorphic functions which is generated by the normalized Laplacian measure $d\mu_{\sigma,1,\mathbb{C}^n}(\bm{z})$ in the traditional $L^2-$sense  given in \eqref{eq_13Laplacemeasure}. In particular, we are interested in determining the \textbf{\emph{compactification}} behaviour of the Koopman operators over the newly developed RKHS associated with the Laplacian measure. In doing so, we demonstrated meticulously that how the Koopman operators are essentially going to act boundedly over the RKHS of the normalized Laplacian measure. Then we also investigates the \emph{essential norm estimates}
for the Koopman operators. Upon the optimization of the essential norm estimates for the Koopman operators over RKHS of the normalized Laplacian measure, we eventually prove that the desired compactification of the Koopman operators. 
\item Being the ability of the Koopman operators that they can be compact over the RKHS of the normalized Laplacian measure, it immediately allow us to extract the \emph{finite rank representation} of the Koopman operators over the RKHS of the Laplacian measure. Thus, once we are aware of this structure, we can proceed further to perform the eDMD with limited data acquisition by using the Koopman operators and RKHS of the normalized Laplacian measure.
\item It should be noted that at the time of this research investigation, we do not have any reference in which the both the reproducing kernel and the corresponding RKHS of the normalized Laplacian measure are discussed. Therefore, this manuscript immediately takes the action of developing the theory on grounds for the Laplacian measure acting in the $L^2-$sense over the holomorphic functions in multivariate case over the complex plane. In doing so, this manuscript provides following Hilbert function space theory details for the Laplacian measure:
\begin{enumerate}
    \item Inner product formulation for the Hilbert space arising from the Laplacian measure, 
    \item Establishment of \emph{point-evaluation inequality} which leads to the formation of the RKHS from the Laplacian measure,
    \item Determination of the orthonormal basis of the Hilbert space generated by the normalized Laplacian measure,
    \item Formulation of the closed-form-expression of the reproducing kernel for the RKHS generated by the normalized Laplacian measure,
    \item Determination of the weakly converging sequence in the RKHS of the normalized Laplacian measure followed by the bounds (both lower and upper) for the norm of reproducing kernel of the RKHS generated by the normalized Laplacian measure.
\end{enumerate}
\item We also investigated the property of the Koopman operator to be \textbf{closable} over the newly generated RKHS from the normalized Laplacian measure, which helps making this RKHS as novel and non-trivial choice for the data-science practice in the light of limited data acquisition. It was first pointed out by \cite{ikeda2022koopman} to exhibit the closable nature of the Koopman operators as a desirable operator property of it over the underlying Hilbert spaces which later on noted by \cite[Page 43]{colbrook2023multiverse} as a tough task to circumvent through.
\end{enumerate}
\subsubsection{Detail plan of the paper}
Needlessly to say that the paper is quite lengthy and this is perhaps due to the fact that many of mathematical details and technicalities are not present in the ready-to-use perspective. After that we have described the key offerings of this paper, we present now the detailed plan of the paper which addresses these offerings. 

The rest of the paper immediately opens up with reviewing the Kernelized Extended Dynamic Mode Decomposition in \autoref{section_DMDLDA} and also we understands the notion of limited data acquisition and Gaussian random matrix theory can help constructing the data matrix. This is where we provide the algorithm to perform the Laplacian Kernel variant of eDMD coupling with Gaussian random vectors along with certain theoretical justifications in \autoref{prpstn_anstoQues_1} and \autoref{prpstn_ansQues2}. The algorithm provided in this section is can be readily taken in hand to perform the experimental study for fluid flow across cylinder experiment, which is performed in \autoref{section_experimentsffac}. However, to get to that section, one has to perform the detail analysis of the RKHS generated by the Laplacian Kernel Function when it is viewed as in $L^2-$measure sense and then how the Koopman operator acts on this RKHS. Both of the study is performed simultaneously in the respective sections of \autoref{section_RKHSLaplaceMeasure} and \autoref{section_Koopman}. Lastly, we show why the Laplacian Kernel and its corresponding RKHS is novel in \autoref{section_novelLap} by determining the sequence of functions which makes Koopman operators to act closable in contrast to any other kernel function and its RKHS such as Gaussian Radial Basis Kernel Function.

\section{DMD with Limited Data Acquisition}\label{section_DMDLDA}
We provided various versions of DMD in the introduction, however, the present research investigation smooth aligns with one of the variant of DMD which relies on the choice of the reproducing kernel and its corresponding function space, which often, as already mentioned, is referred RKHS. The kernel variant of eDMD has one very strong benefit, that it has the strong ability to handle the \emph{curse of dimensionality}. Thanks to the \emph{kernel-trick}, which is a natural nature offered by the reproducing kernel. The ideology for extracting the dominant pieces of information from the time-changing physical phenomena is the approximation of Koopman operators by embedding it in the RKHS via its corresponding reproducing kernel functions and this can be immediately observed in \cite{baddoo2022kernel,philipp2023error,khosravi2023representer,klus2018kernel,klus2020kernel,giannakis2020extraction,das2020koopman,das2021reproducing,fujii2019dynamic,alexander2020operator,burov2021kernel,zhao2016analog}. 
\subsection{Extended Dynamic Mode Decomposition}
Now, that we have already provided the exposure of the dynamical systems in the RKHS setting via the action of the Koopman operator, we recall one of the variant of DMD which exploits the direct use of the Koopman operators over the choice of the RKHS, i.e. the extended-Dynamic Mode Decomposition (eDMD) \cite{kevrekidis2016kernel,williams2015data}.
\subsubsection{Review of eDMD} Our algorithm for the data-driven discovery heavily relies on the setting of the eDMD and this strongly motivates us to review the eDMD which is followed by the definition of the data-set of snapshots.
\begin{dfn}\label{dfn_4.1}
    Consider $(n,\mathcal{M},\mathbf{F})$ be the discrete dynamical system where $n\in\mathbb{Z}$ is time $\mathcal{M}\subseteq\mathbb{R}^n$ is the state space and $\mathbf{x}\mapsto\mathbf{F}(\mathbf{x})$ is the dynamics. Then the data-set of snapshots of pairs corresponding to the discrete dynamical system $(n,\mathcal{M},\mathbf{F})$ is given as following:
    \begin{align}\label{eq_data-setsnaps}
    \begin{bmatrix}
\bm{|} & \bm{|} & \bm{|}&\bm{|}\\
\mathbf{x}_1 & \mathbf{x}_2 &\cdots& \mathbf{x}_m\\
\bm{|} & \bm{|} &\bm{|}&\bm{|}
\end{bmatrix}\overset{\mathbf{x}\mapsto\mathbf{F}(\mathbf{x})}{\mapsto}\begin{bmatrix}
\bm{|} & \bm{|} & \bm{|}&\bm{|}\\
\mathbf{F}(\mathbf{x}_1) & \mathbf{F}(\mathbf{x}_2) &\cdots& \mathbf{F}(\mathbf{x}_m)\\
\bm{|} & \bm{|} &\bm{|}&\bm{|}
\end{bmatrix}\overset{\mathbf{y}=\mathbf{F}(\mathbf{x})}{\coloneqq}\begin{bmatrix}
\bm{|} & \bm{|} & \bm{|}&\bm{|}\\
\mathbf{y}_1 & \mathbf{y}_2 &\cdots& \mathbf{y}_m\\
\bm{|} & \bm{|} &\bm{|}&\bm{|}
\end{bmatrix}.
\end{align}
\end{dfn}
With a slight abuse of notation to the Koopman operator as $\mathcal{K}$ (instead of $\mathcal{K}_{\mathbf{F}_t}$ in \eqref{eq_9r}), we understand that the Koopman operator acting on the observable $\phi:\mathcal{M}\to\mathbb{C}$ as
\begin{align*}
    \mathcal{K}\phi(\mathbf{x})=\phi\circ\mathbf{F}(\mathbf{x})=\phi(\mathbf{F}(\mathbf{x})),
\end{align*}
yields a brand new scalar valued function that gives the value of $\phi$ \emph{one-step ahead in the future} against the discrete dynamical system $\left(n,\mathcal{M},\mathbf{F}\right)$, where $n\in\mathbb{Z}$, $\mathcal{M}\subseteq\mathbb{R}^N$ and $\mathbf{x}\mapsto\mathbf{F}(\mathbf{x})$. In the natural interest of determining the Koopman spectra-observables i.e. Koopman eigen-values $(\mu_k)$ and Koopman eigen-functionals $(\varphi_k)$, they are also accompanied by the Koopman modes $(\bm{\xi}_k)$ of a certain vector valued observable $\bm{g}:\mathcal{M}\to\mathbb{R}^{N_o},$ (${N_o}\in\mathbb{N}$), which is refer as the \emph{full state observable} given as $\bm{g}\left(\mathbf{x}\right)=\mathbf{x}$. Further, one can have a following decomposition in terms of the aforementioned the triple \emph{eigen-values, eigen-functionals \& modes} of the Koopman operator corresponding to the (unknown) dynamics $\mathbf{x}\mapsto\mathbf{F}(\mathbf{x})$: 
\begin{align*}
    \mathbf{x}=\sum_{k=1}^{N_k}\bm{\zeta}_k\varphi_k(\mathbf{x}),\quad\mathbf{F}(\mathbf{x})=\sum_{k=1}^{N_k}\mu_k\bm{\zeta}_k\varphi_k(\mathbf{x}),
\end{align*}
where, supposing that $N_k$ is the number of tuples required for the re-construction of the system from the data of the dynamical system.

In the light of choosing scalar observables for the eDMD process, we will define the important notion of \emph{feature map} and \emph{feature space} as follows:
\begin{dfn}
The eDMD is provided the choice of scalar observables and for that let $\mathcal{F}$ be the appropriate choice of scalar observables (such as RKHS). To do this, let $\psi_k:\mathcal{M}\to\mathbb{R}$ for $k=1,\ldots,N_k$ under the assumption that $\operatorname{span}(\mathcal{F}_{N_k})\subset\mathcal{F}$. In particular,  the space of scalar observables is approximated using $\left\{\psi_k\right\}_{k=1}^{N_k}$ functions  then feature space is $\mathbb{R}^{N_k}$. Additionally, the \emph{feature map} $\bm{\psi}$ will be the `stacked' column vector of entries $\left\{\psi_{k}\right\}$ formally given as follows:
\begin{align*}
    \bm{\psi}\left(\mathbf{x}\right)=\begin{bmatrix}
        \psi_{1}\left(\mathbf{x}\right)\\
        \psi_{1}\left(\mathbf{x}\right)\\
        \vdots\\
        \psi_{N_k}\left(\mathbf{x}\right)
    \end{bmatrix}.
\end{align*}
\end{dfn}
With the feature space $\mathbb{R}^{N_k}$ and considering the value of any functions $\phi,\Tilde{\phi}\in\mathcal{F}_{N_k}$, where (again) $\operatorname{span}\mathcal{F}_{N_k}\subset\mathbb{R}^{N_k}$, one can define the evaluation of both $\phi$ and $\Tilde{\phi}$ against the inner product with certain  coefficient vector $\bm{a}$ and $\Tilde{\bm{a}}$ in $\mathbb{C}^{N_k}$:
\begin{align*}
    \phi\left(\mathbf{x}\right)&=\langle\bm{a},\bm{\psi}(\mathbf{x})\rangle_{\mathbb{R}^{N_k}}=\bm{\psi}(\mathbf{x})^\top\bm{a}=\sum_{k=1}^{N_k}a_k\psi_{k}(\mathbf{x})\\
     \Tilde{\phi}\left(\mathbf{x}\right)&=\langle\Tilde{\bm{a}},\bm{\psi}(\mathbf{x})\rangle_{\mathbb{R}^{N_k}}=\bm{\psi}(\mathbf{x})^\top\Tilde{\bm{a}}=\sum_{k=1}^{N_k}\Tilde{a_k}\psi_{k}(\mathbf{x}).
\end{align*}
It should be noted that the goal of the eDMD is to employ the pair of data-set of snapshots defined in \eqref{eq_data-setsnaps} to generate the compactified\footnote{finite rank representation of the infinite dimensional Koopman operator} version of the Koopman operator denoted by $\bm{\mathcal{K}}\in\mathbb{R}^{N_k}\times\mathbb{R}^{N_k}$ for some given coefficients $\bm{a}$ and $\Tilde{\bm{a}}$ such that $\mathfrak{r}=\left(\mathcal{K}\phi-\Tilde{\phi}\right)\in\mathcal{F}$, is minimum. We now provide the algorithm of eDMD.
{\small\begin{algo}The algorithm for the extended-Dynamic Mode Decomposition is given as follows:
    {\begin{enumerate}
    \item[Step 1] With the pair of data-set of snapshots as defined in \eqref{eq_data-setsnaps}, compute the following observation matrices with respect to the scalar observables (kernel) $\bm{\psi}$:
    \begin{align*}
        \bm{\Psi}_{\mathbf{x}}\triangleq\begin{bmatrix}
            \bm{\psi}\left(\mathbf{x}_1\right)^\top\\
            \bm{\psi}\left(\mathbf{x}_2\right)^\top\\
            \vdots\\
            \bm{\psi}\left(\mathbf{x}_M\right)^\top
        \end{bmatrix},\qquad \bm{\Psi}_{\mathbf{y}}\triangleq\begin{bmatrix}
            \bm{\psi}\left(\mathbf{y}_1\right)^\top\\
            \bm{\psi}\left(\mathbf{y}_2\right)^\top\\
            \vdots\\
            \bm{\psi}\left(\mathbf{y}_M\right)^\top
        \end{bmatrix}, 
    \end{align*}
    where $\bm{\Psi}_{\mathbf{x}}$ and $\bm{\Psi}_{\mathbf{y}}$ constitutes the matrix in $\mathbb{R}^{M\times N_k}$.
    \item[Step 2] Compute the following two matrices:
    \begin{align*}
        \bm{\mathfrak{G}}&=\bm{\Psi}_{\mathbf{x}}^\top\bm{\Psi}_{\mathbf{x}}=\sum_{m=1}^M\bm{\Psi}\left(\mathbf{x}_m\right)\bm{\Psi}^\top\left(\mathbf{x}_m\right)\\
        \bm{\mathds{A}}&=\bm{\Psi}_{\mathbf{x}}^\top\bm{\Psi}_{\mathbf{y}}=\sum_{m=1}^M\bm{\Psi}\left(\mathbf{x}_m\right)\bm{\Psi}^\top\left(\mathbf{y}_m\right).
    \end{align*}
    \item[Step 3]Determine the pseudo-inverse of $\bm{\mathfrak{G}}$ and denote it by $\bm{\mathfrak{G}}^{(-)\mathfrak{p}}$.
    \item[Step 4]Determine $\bm{\mathcal{K}}$ by setting 
    \begin{align*}
        \bm{\mathcal{K}}\triangleq\bm{\mathfrak{G}}^{(-)\mathfrak{p}}\bm{\mathds{A}}.
    \end{align*}
\end{enumerate}}
\end{algo}
}
In the spirit of singular value decomposition (SVD) based DMD via \cite{schmid2010dynamic}, the SVD of $\bm{\Psi}_{\mathbf{x}}$ can be used to construct a matrix similar to $\bm{\mathcal{K}}$; this is given as follows:
\begin{proposition}\label{prpstin_williamsKoopman}
    Let the SVD of $\bm{\Psi}_{\mathbf{x}}$ takes the following mathematical structures:
    \begin{align*}
        \bm{\Psi}_{\mathbf{x}}\triangleq\bm{Q\Sigma Z}^{\top}, 
    \end{align*}
    where $\bm{Q}$ and $\bm{\Sigma}\in \mathbb{R}^{M\times M}$ and $\bm{Z}\in\mathbb{R}^{N_k\times M}$. The pair of non-negative $\mu$ and $\hat{\bm{v}}$ are respective an eigenvalue and eigenvector of
    \begin{align}\label{eq_27rr}
        \hat{\bm{\mathcal{K}}}\triangleq\left(\bm{\Sigma}^{(-)\bm{\mathfrak{p}}}\bm{Q}^\top\right)\left(\bm{\Psi}_{\mathbf{y}}\bm{\Psi}_{\mathbf{x}}^\top\right)\left(\bm{Q}\bm{\Sigma}^{(-)\bm{\mathfrak{p}}}\right)=\left(\bm{\Sigma}^{(-)\bm{\mathfrak{p}}}\bm{Q}^\top\right)\hat{\bm{\mathds{A}}}\left(\bm{Q\Sigma}^{(-)\bm{\mathfrak{p}}}\right),
    \end{align}
if and only if $\mu$ and $\bm{v}=\bm{Z}\hat{\bm{v}}$ are an eigen-value and eigen-vector of $\bm{\mathcal{K}}$.
\end{proposition}
Now, that we have defined the basic ingredients including the action of the Koopman operator on the discrete dynamical system for the eDMD, we present the reproducing kernel variant of the eDMD.
{\small{\begin{algo}\label{algo_kerneleDMD}
The eDMD powered by the reproducing kernel function $K(\cdot,\cdot)$ is given as follows:
    \begin{enumerate}
        \item[Step 1] In the light of data-set snapshots given in \eqref{eq_data-setsnaps}, choose the reproducing kernel function $K(\cdot,\cdot)$, which generate the corresponding and unique RKHS.
        \item[Step 2] Compute the elements of $\hat{\bm{\mathfrak{G}}}\coloneqq\left[\hat{\bm{\mathfrak{G}}}\right]_{i\times j}$ and $\hat{\bm{\mathds{A}}}\coloneqq\left[\hat{\bm{\mathds{A}}}\right]_{i\times j}$ that are defined as follows:
        \begin{align*}
            \left[\hat{\bm{\mathfrak{G}}}\right]_{i\times j}&=K\left(\mathbf{x}_i,\mathbf{x}_j\right)\\
            \left[\hat{\bm{\mathds{A}}}\right]_{i\times j}&=K\left(\mathbf{y}_i,\mathbf{x}_j\right).
        \end{align*}
        \item[Step 3] Determine the spectral observables of the Gram-matrix $\hat{\bm{\mathfrak{G}}}$, i.e. $\bm{Q}$ and $\bm{\Sigma}$.
        \item[Step 4] Construct $\hat{\bm{\mathcal{K}}}$ via \eqref{eq_27rr}.
    \end{enumerate}
\end{algo}}}
In \cite{williams2015data}, the choice of kernel function were the polynomial kernel function which is given as $K_\alpha(\bm{x},\bm{y})=\left(1+\nicefrac{\langle\bm{x},\bm{y}\rangle}{d^2}\right)^\alpha$ and the Gaussian Radial Basis Kernel Function. Apparently, for some reasons the aforementioned kernel functions has always been the general choice of performing the DMD which involves the computation of Gram-matrix $\hat{\bm{\mathfrak{G}}}$, for example in \cite{rosenfeld2022dynamic}.
\subsection{On Limited Data Acquisition} 
All the variants of DMD expects to run on the availability of data-set snapshots that we have defined formally in \eqref{eq_data-setsnaps} under the additional assumption that moderate number of snapshots are actually provided. Explicitly, in \eqref{eq_data-setsnaps}, we can see that we have a total of $m-$data snapshots for the discrete dynamical system $(n,\mathcal{M},\mathbf{{F}})$ given in \autoref{dfn_4.1}. Now, the question in which we are interested is that would we be able to perform the extended-DMD if only limited data-set snapshots are available for any general discrete dynamical system. Following is the simple definition for the limited data acquisition on which the present paper is based upon.
\begin{dfn}\label{dfn_LDA}
    For the general discrete dynamical system $(n,\mathcal{M},\mathbf{F})$ as given in \autoref{dfn_4.1}, we regard the data-set snapshots as ideally the \emph{full data acquisition}, in which we have, here, as $m-$snapshots. If, we have the same data-set snapshots but for some positive $m_0\in\mathbb{Z}_+$ which satisfy $m_0<m$, then we simply say these data-set snapshots as \emph{limited data acquisition}.
\end{dfn}
\begin{example}\label{example_ffac}
    For the standard fluid flow across cylinder experiment (as such in \cite{bagheri2013koopman} etc.) with its Reynolds number as $100$ in \cite{kutz2016dynamic}, there are $151-$data-set snapshots and each snapshots is the column vector whose dimension is \emph{89,351}. In this case, the full data acquisition corresponds to data-set snapshots of matrix of dimension \emph{89,351-by-151}. However, for the same physical experiment, if we have data-set snapshots of matrix whose dimension is \emph{89,351-by-}$m_0$, where $0<m_0<151$, then we will regard it simply as the limited data acquisition.
\end{example}
Besides the crucial challenge of uncertainty in knowing the about the exact mathematical structure of $\mathbf{f}$ in \eqref{eq_1} that governs the concerned dynamical system, the next immediate challenge that we face in the theme of data-driven science is about the information of dynamical system when we do not have the luxury of enough data-set snapshots. It turns out that for the data-driven methodology, one can use the knowledge from the important topics of random matrix theory \cite{edelman2005random,tao2023topics} to construct the data-driven models for the related dynamical systems. 
\subsection{Using random matrix theory for limited data acquisition}
Assuming that the case of limited data acquisition in the direction of \autoref{dfn_LDA} for a certain dynamical system in which $0<m_0<m$ for the data-set snapshots as in \eqref{eq_data-setsnaps}, we will construct a data-set snapshots matrix in which random vector from certain probability density function will be drawn and augmented so that the final matrix of the data-set snapshots eventually contains (again) $m$ column vectors (snapshots). 
\begin{figure}[H]
    \centering
    \frame{\includegraphics[scale=.09]{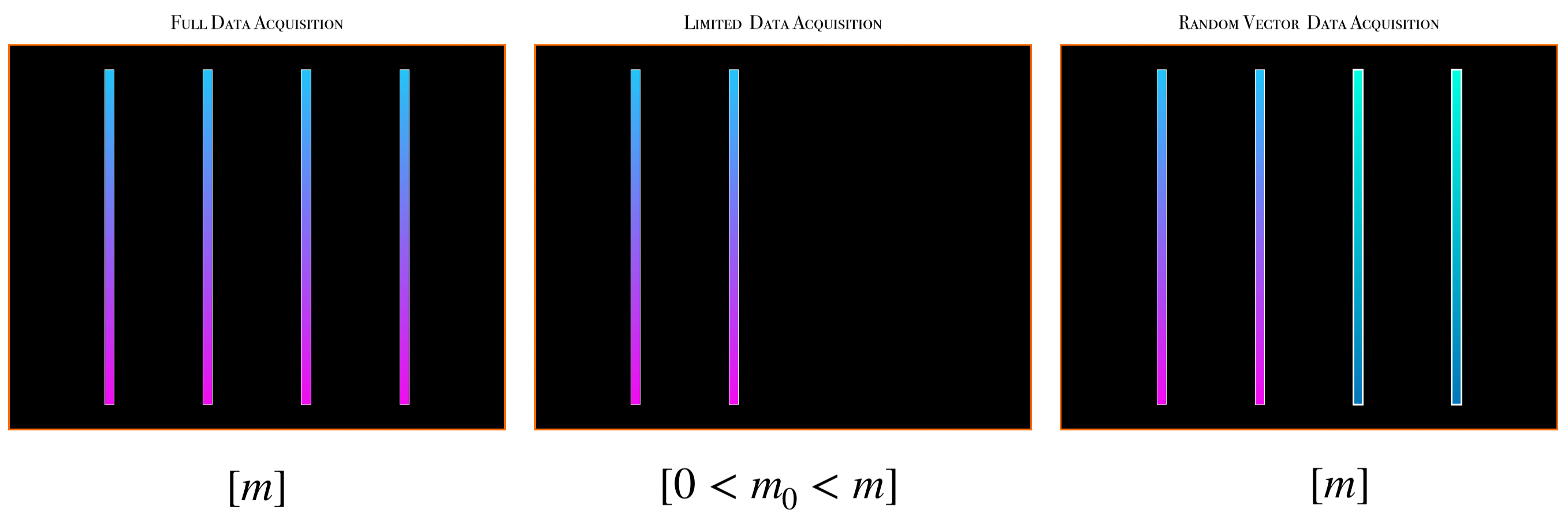}}
    \caption{A conceptual diagram to intercept the different notions of data acquisition. The very-first matrix in the \underline{\textsc{extreme left}} symbolised the ideal or at-least expected condition of having-supposedly $m-$snapshots. The unfortunate situation of acquiring only $m_0-$snapshots out of those $m-$snapshots is depicted in the \underline{\textsc{middle figure}}. In order to construct the ideal situation matrix in which $m-$snapshots are present, a simple padding of $m-m_0$ random vectors is inserted and this is symbolized in \underline{\textsc{extreme right}}.}
    \label{fig:Data-set snaps with FDA LDA and RV DA}
\end{figure}
\subsubsection{Preliminary to the Random Matrix Theory} 
The ensembles of well-studied random matrix for example \textsc{manova}, \emph{Gaussian, \& Wishart} etc. from the realm of random matrix theory has always played an integral role in various areas of mathematical probability and its application in mathematical physics \cite{zirnbauer1996riemannian,mehta2004random,caselle2004random,ivanov2002random}. We will present the definition of the Gaussian random matrix ensemble in the interest of our experiment.
\begin{dfn}[Gaussian Random Matrix \cite{edelman2005random}]
    For some positive integers $m$ and $n$, the Gaussian random matrix $G_1(m,n)$ is simply the $m\times n$ matrix of independent and identically distributed standard random normals. If $\bm{A}$ is the Gaussian random matrix of $m\times n$ dimension, then its joint element density is given as follows:
    \begin{align}\label{eq_jointgaussianPDF}
        \frac{1}{(2\pi)^{\nicefrac{mn}{2}}}\exp\left(-\frac{1}{2}\|\bm{A}\|_{F}^2\right),
    \end{align}
    where $\|\bullet\|_{F}$ is the Frobenius norm.
\end{dfn}
More concretely, one can also define the probability density function for some random vector $X$ in more rigorous way as follows.
\begin{dfn}[Multi-variate Gaussian random vector {\cite[Definition 3.2.1; Page 26]{tong1990fundamental}}]
    Let $X\coloneqq(X_1,\ldots,X_n)$ be a random vector. Then we say that $X$ is a Gaussian random vector if we can write $X=\bm{\mu}+\bm{A}\bm{Z}$, where $\bm{\mu}\in\mathbb{R}^n$, $\bm{A}$ is $n\times k-$matrix and $\bm{Z}$ is a $k-$vector of independently identically distributed standard normal random variables. We use the notation $X\sim \mathcal{N}_n\left(\bm{\mu},\bm{\Sigma}\right)$ where $\bm{\Sigma}=\bm{AA}^\top$, to imply that $X$ is the Gaussian random vector whose distribution is determined by $\bm{\mu}$ and $\bm{\Sigma}$. We also say that $X\sim\mathcal{N}_n\left(\bm{\mu},\bm{\Sigma}\right)$ is non-degenerate when $\bm{\Sigma}=\bm{AA}^\top$ is is positive definite (equivalently invertible). 

    In the case, when ${\Tilde{\mathbf{x}}_n}\sim\mathcal{N}_n\left(\bm{\mu},\bm{\Sigma}\right)$ and is non-degenerate, then the probability density function of $X$ is given as 
    \begin{align}\label{eq_gaussianPDF}
    \operatorname{pdf}_{\Tilde{\mathbf{x}}_n}\left(\mathbf{x}_0\right)=\frac{1}{(2\pi)^{\nicefrac{n}{2}}|\mathbf{\Sigma}|^{\nicefrac{1}{2}}}\exp\left((\mathbf{x}_0-\bm{\mu})^{\top}\mathbf{\Sigma}^{-1}(\mathbf{x}_0-\bm{\mu})\right),
    \end{align}
    for all $\mathbf{x}_0\in\mathbb{R}^n.$
\end{dfn}
\subsubsection{Choice of random matrix}It should be noted that there are several choices of considering random matrix other than what we have described for the Gaussian random matrix in either of the definitions given above. The most important property of (any) Gaussian random matrix is its \emph{orthogonal in-variance} \cite{edelman2005random} which means that if $\bm{A}$ is a Gaussian random matrix, then one fails to distinguish among $Q_1\bm{A}$, $\bm{A}$ and $\bm{A}Q_2$, where $Q_1$ and $Q_2$ are non-random orthogonal. The other choices of random matrix that one can use could be \emph{Uniform random matrices \& Rademacher matrices \cite{hitczenko1994rademacher}}, however the empirical evidence in the present manuscript establishes that augmenting the data-set snapshots with the Gaussian random matrix yields better results in contrast with any other random matrix.

Now, we will present the compelling theory behind the process of augmenting the data-set snapshots when limited data acquisition holds with the Gaussian random vector whose probability density function in general is given in \eqref{eq_gaussianPDF}.
\subsection{Theory for augmenting data-set snapshots with random matrix} As \autoref{fig:Data-set snaps with FDA LDA and RV DA} already demonstrates about the situation of limited data acquisition in which we tend to augment the data-set snapshots matrix with the suitable matrix size of random matrix, in particular Gaussian random matrix to mimic the exact correspondence of full data acquisition. However, the figurative description is merely not enough to explain how one can simply augment the data-set snapshots matrix with Gaussian random matrix and resume the Kernelized eDMD process to extract the governing feature of the underlying dynamical system. In particular, in doing so, we encounter with following two crucial questions which needed to answered in the best-robust possible way.
\begin{quesaug}\label{ques_1}
Does there exist a subspace inside a measurable space on which the sub-sequence of snapshots generated by the discrete dynamical system $\mathbf{x}_{n+1}=\mathbf{F}(\mathbf{x}_n)$ (here, $\mathbf{y}_n\coloneqq\mathbf{x}_{n+1}$ in \eqref{eq_data-setsnaps}) can be made arbitrarily close to the Gaussian random vector $\Tilde{\mathbf{x}}_n$ whose probability density function is defined in \eqref{eq_gaussianPDF}? Precisely, if the observables is in the domain of Koopman operator as $L^2(\mu)$ over a measurable space $X$, then given that the sequence of observables $\left\{\mathfrak{g}_M\right\}_M\to\operatorname{pdf}_{\Tilde{\mathbf{x}}_n}$ almost everywhere over $L^2(\mu)$, then does this convergence happens in measure as well?
\end{quesaug}
If $\mathbf{x}_n$ is the actual $n-$th snapshot which is evolved by the above discrete dynamical system, and let $\Tilde{\mathbf{x}}_n$ be Gaussian random vector whose probability density function is $\operatorname{pdf}_{\Tilde{\mathbf{x}}_n}$, then the above question is simply asking that can the difference between these two vectors can be made arbitrarily small in measure-theoretic sense. If so, then the answer (given in \autoref{prpstn_anstoQues_1}) to this provide the immediate justification of using the Gaussian random matrix of appropriate matrix dimensions in the light of limited data acquisition.
\begin{quesaug}\label{ques_2}
    If $\mathbf{x}_n$ is supposedly the actual $n-$th snapshot evolved by the above discrete dynamical system,  and let $\Tilde{\mathbf{x}}_n$ be Gaussian random vector whose probability density function is $\operatorname{pdf}_{\Tilde{\mathbf{x}}_n}$, then is it possible to minimize the difference between the observable $g$ evaluated at $\mathbf{x}_n$ and this  $\operatorname{pdf}_{\Tilde{\mathbf{x}}_n}$ as the difference between $\mathbf{x}_n$ and $\Tilde{\mathbf{x}}_n$ is negligibly small?
\end{quesaug}
The above question is raising the concern for the interpretation of the probability density function as the choice of some observable space on which the Koopman operator $\mathcal{K}$ can act upon, when we have constructed the data-set snapshot matrix by augmenting, here with Gaussian random vector at the $n-$th snapshot entry. If so, then the answer (given in \autoref{prpstn_ansQues2}) to this question allow us to treat the probability density function for the Gaussian random vectors to interpret as the observables for the Koopman operators.

Now, finally we provide the solutions and answers to these question as follows.
\subsubsection{Answer to Question 1} To answer the \autoref{ques_1}, we need to recall the notion of \emph{convergence in measure}. It is done so as follows:
\begin{dfn}\label{dfn_measureconvergence}
    Let $\mu$ be a positive measure on measurable space $X$. A sequence $f_N$ of complex measurable functions on $X$ is said to \emph{converge in measure} to the measurable function $f$ is to every $\epsilon>0$ there corresponds an $N$ such that 
    \begin{align*}
        \mu\left\{x\in X:|f_N(x)-f(x)|>\epsilon\right\}<\epsilon,
    \end{align*}
    for all $\mathfrak{n}>N$. The notation that we use to symbolize convergence of $f_N$ to $f$ in measure is $f_N\overset{\mu}{\to}f$.
\end{dfn}

The important application of the notion of \emph{convergence in measure} is capture in following theorem, whose proof can be found in the standard references such as \emph{Scheffe\'{e}'s Lemma} from \cite[Page 55]{williams1991probability} or \cite{scheffe1947useful} or \cite[Probelm 18, Page 74]{rudin1987real}. The following characterization is needed for this manuscript to provide the convergence in measure.
\begin{theorem}\label{theorem4.2}
    Let $f_N$ be sequence of $L^p\left(\mu\right)$ functions as defined in \autoref{dfn_measureconvergence} and if $\|f_N-f\|_{L^p(\mu)}\to0$ then $f_N\to f$ in measure.
\end{theorem}
We would like to recall an important measure theory related fact, which is true for any general $L^p$ spaces with respect to (positive) $\mu-$measure and $1\leq p<\infty$. However, here we stated only for $p=2$ case.
\begin{lemma}\label{lemma_simple_functions}
    Let $S$ be the class of all complex, measurable, simple functions on the $L^2\left(\mu\right)$ measurable space $X$ such that for $s\in S$, $\mu\left\{x\in X:s(x)\neq0\right\}<\infty.$
    Then, $S$ is dense in $L^2(\mu)$.
\end{lemma}
We provide the following proposition which provides the conclusion that how $\mathbf{x}_N$ can get arbitrarily close to the $\Tilde{\mathbf{x}}_N$ which follows the multivariate Gaussian probability density function and ultimately answering the \autoref{ques_1}. Essentially the following proposition can be realized as the nice application of the aforementioned \autoref{lemma_simple_functions}.
\begin{proposition}\label{prpstn_anstoQues_1}
    Let $\left\{\mathfrak{g}_M\right\}_M$ be the sequence of real (or complex, if needed) measurable functions on $X$ whose domain is the range of $\mathbf{F}_t$ as in \autoref{assumption_1} or as in $\mathbf{x}_{n+1}=\mathbf{F}(\mathbf{x}_n)$. If $\left\{\mathfrak{g}_M\right\}_M\to\operatorname{pdf}_{\Tilde{\mathbf{x}}_n}$ almost everywhere in $L^2\left(\mu\right)-$sense where $\operatorname{pdf}_{\Tilde{\mathbf{x}}_n}$ is the probability density function for the Gaussian random vector ${\Tilde{\mathbf{x}}_n}$ as given in \eqref{eq_gaussianPDF}, then $\left\{\mathfrak{g}_M\right\}_M$ converges to $\operatorname{pdf}_{\Tilde{\mathbf{x}}_n}$ in measure.
\end{proposition}
\begin{proof}
Define $S$ as in \autoref{lemma_simple_functions} and let $\epsilon_i>0$ for $i=1$ and $2$. We begin by assuming that the sequence of $\left\{\mathfrak{g}_M\right\}_M$ that are defined on locally compact Hausdorff measurable space $X$ such as $\mathbb{R}^n$. Now, fix $M\in\mathbb{Z}_+$ and for this $\mathfrak{g}_M$ defined over locally compact Hausdorff measurable space $X$, we have $\mathfrak{s}_M\in S$ so that $\mathfrak{g}_M=\mathfrak{s}_M$ except on a set of $\mu$-measure $<\epsilon_1$ for this fixed $M$. Also, $|\mathfrak{g}_M|\leq\|\mathfrak{s}_M\|_{\infty}$ by the direct virtue of the Lusin's theorem\footnote{Suppose $f$ is a complex measurable on a measurable space $X$, $\mu(A)<\infty$, $f(x)=0$ if $x\in A$, and $\epsilon>0$. Then there exists a $g\in C_c(X)$ such that $\mu\left(\left\{x:f(x)\neq g(x)\right\}\right)<\epsilon$. Furthermore, we may arrange it so that $\sup_{x\in X}|g(x)|\leq\sup_{x\in X}|f(x)|$.} (cf. \cite[Page 55]{rudin1987real}). Hence, 
\begin{align*}
    \|\mathfrak{g}_M-\mathfrak{s}_M\|_{L^2(\mu)}\leq2\epsilon_1^{\nicefrac{1}{2}}\|\mathfrak{s}_M\|_{\infty}.
\end{align*}

Similarly, for the probability density function $\operatorname{pdf}_{\Tilde{\mathbf{x}}_n}$ for the Gaussian random vector ${\Tilde{\mathbf{x}}_n}$ defined over the locally compact Hausdorff measurable space $X$ (again) such as $\mathbb{R}^n$, we have $\bm{s}\in S$ so that $\operatorname{pdf}_{\Tilde{\mathbf{x}}_n}=\bm{s}$ except on a set of measure $<\epsilon_2$. Additionally, $|\operatorname{pdf}_{\Tilde{\mathbf{x}}_n}|\leq\|\bm{s}\|_{\infty}$ by employing the Lusin's theorem.
Hence, 
\begin{align*}
    \|\operatorname{pdf}_{\Tilde{\mathbf{x}}_n}-\bm{s}\|_{L^2(\mu)}\leq2\epsilon_2^{\nicefrac{1}{2}}\|\bm{s}\|_{\infty}.
\end{align*}
Let $\epsilon>0$ such that $\max\left\{\epsilon_1^{\nicefrac{1}{2}}\|\mathfrak{s}_M\|_{\infty},\epsilon_2^{\nicefrac{1}{2}}\|\bm{s}\|_{\infty}\right\}<4^{-1}\left(\epsilon-\|\mathfrak{s}_M-\bm{s}\|_{L^2(\mu)}\right)$. Now, observe that
\begin{align*}
    \|\mathfrak{g}_M-\operatorname{pdf}_{\Tilde{\mathbf{x}}_n}\|_{L^2(\mu)}=&\|\mathfrak{g}_M-\mathfrak{s}_M+\mathfrak{s}_M-\bm{s}+\bm{s}-\operatorname{pdf}_{\Tilde{\mathbf{x}}_n}\|_{L^2(\mu)}\\
    \leq&\|\mathfrak{g}_M-\mathfrak{s}_M\|_{L^2(\mu)}+\|\mathfrak{s}_M-\bm{s}\|_{L^2(\mu)}+\|\bm{s}-\operatorname{pdf}_{\Tilde{\mathbf{x}}_n}\|_{L^2(\mu)}\\
    =&\|\mathfrak{g}_M-\mathfrak{s}_M\|_{L^2(\mu)}+\|\bm{s}-\operatorname{pdf}_{\Tilde{\mathbf{x}}_n}\|_{L^2(\mu)}+\|\mathfrak{s}_M-\bm{s}\|_{L^2(\mu)}\\
    \leq&2\epsilon_1^{\nicefrac{1}{2}}\|\mathfrak{s}\|_{\infty}+2\epsilon_2^{\nicefrac{1}{2}}\|\bm{s}\|_{\infty}+\|\mathfrak{s}-\bm{s}\|_{L^2(\mu)}\\
    <&2\cdot2\max\left\{\epsilon_1^{\nicefrac{1}{2}}\|\mathfrak{s}\|_{\infty},\epsilon_2^{\nicefrac{1}{2}}\|\bm{s}\|_{\infty}\right\}+\|\mathfrak{s}_M-\bm{s}\|_{L^2(\mu)}\\
    <&4\cdot4^{-1}\left(\epsilon-\|\mathfrak{s}_M-\bm{s}\|_{L^2(\mu)}\right)+\|\mathfrak{s}_M-\bm{s}\|_{L^2(\mu)}.
\end{align*}
The above calculation results into $\|\mathfrak{g}_M-\operatorname{pdf}_{\Tilde{\mathbf{x}}_n}\|_{L^2(\mu)}<\epsilon$ for some fixed $M$ and as this $\epsilon$ was arbitrary, we eventually see that $\|\mathfrak{g}_M-\operatorname{pdf}_{\Tilde{\mathbf{x}}_n}\|_{L^2(\mu)}\to0$. Therefore, by applying the \autoref{theorem4.2}, we have that $\left\{\mathfrak{g}_M\right\}_M$ converges to $\operatorname{pdf}_{\Tilde{\mathbf{x}}_n}$ in measure.
\end{proof}
\subsubsection{Answer to Question 2}
We provide the solution for the \autoref{ques_2} as follows which relies on the \emph{projection valued measure form} of the spectral theorem \cite[Chapter 8, Thereom VIII.6, Page 263]{reedmethods} and with additional mathematical techniques from \cite[Chapter 7 and Chapter 8]{hall2013quantum}. 
\begin{proposition}\label{prpstn_ansQues2}
    Consider the discrete dynamical system as $\mathbf{x}_{n+1}=\mathbf{F}(\mathbf{x}_{n})$ or as the $\mathbf{F}_t$ defined in \autoref{assumption_1} and let $g$ be an observable function for the Koopman operator $\mathcal{K}$ whose domain is $\mathcal{D}(\mathcal{K})=L^2(\mu)$. Let $g_n\coloneqq g(\mathbf{x}_n)$ and let the $n-$th Gaussian random vector $\Tilde{\mathbf{x}}_n$ is drawn by the probability density function $\operatorname{pdf}_{\Tilde{\mathbf{x}}_n}(\mathbf{x}_{0})$. Then following holds true 
    \begin{align*}
        |g(\mathbf{x}_{n})-\operatorname{pdf}_{\Tilde{\mathbf{x}}_n}(\mathbf{x}_{0})|\to0\quad\text{where}\quad g_n\overset{\mu}{\to}\operatorname{pdf}_{\Tilde{\mathbf{x}}_n}.
    \end{align*}
\end{proposition}
\begin{proof}
    The proof of the above theorem involves the detail involvement of the \emph{projection-valued measure} $\mathcal{E}$ defined over the support on the spectrum of $\mathcal{K}$ that is $\operatorname{spec}\left(\mathcal{K}\right)$. For $g\in\mathcal{D}\left(\mathcal{K}\right)=L^2(\mu)$, we can have following in the same way as we have in \cite[Section 2.1.2, Page 229]{colbrook2024rigorous},
    \begin{align*}
        g=\left(\int_{\mathbb{T}}d\mathcal{E}(\lambda)\right)g\quad\text{and}\quad\mathcal{K}g=\left(\int_{\mathbb{T}}\lambda d\mathcal{E}(\lambda)\right)g.
    \end{align*}
    In the spirit of above, one can have following in the light of $\operatorname{pdf}_{\Tilde{\mathbf{x}}_n}$ as follows:
    \begin{align}\label{eq_pdfobeservable}
        \operatorname{pdf}_{\Tilde{\mathbf{x}}_n}=\left(\int_{\mathbb{T}}d\mathcal{E}(\lambda)\right)\operatorname{pdf}_{\Tilde{\mathbf{x}}_n}\quad\text{and}\quad\mathcal{K}\operatorname{pdf}_{\Tilde{\mathbf{x}}_n}=\left(\int_{\mathbb{T}}\lambda d\mathcal{E}(\lambda)\right)\operatorname{pdf}_{\Tilde{\mathbf{x}}_n}
    \end{align}
    The very-defining action of the Koopman operator allow us to write following for the observable function $g$
    \begin{align}\label{eq_21r}
        g\left(\mathbf{x}_n\right)=\left[\mathcal{K}^ng\right]\left(\mathbf{x}_0\right)=\left[\left(\int_{\mathbb{T}}\lambda^nd\mathcal{E}(\lambda)\right)g\right]\left(\mathbf{x}_0\right).
    \end{align} As from the general theory of probability density functions that they are $L^1-$integrable with respect to the (arbitrary) positive finite Borel $\mu-$probability measure, thus this implies that $\operatorname{pdf}_{\Tilde{\mathbf{x}}_n}$ is in $L^2(\mu)$\footnote{due to the ordering or set inclusion of various $L^p$ with respect to the finite Borel measure, follow \cite{villani1985another}.} which is the observable space for the Koopman operator. Similarly for the $\operatorname{pdf}_{\Tilde{\mathbf{x}}_n}$ as an observable, we have following:
    \begin{align}\label{eq_22r}
        \operatorname{pdf}_{\Tilde{\mathbf{x}}_n}\left(\mathbf{x}_0\right)=\left[\mathcal{K}^n\operatorname{pdf}_{\Tilde{\mathbf{x}}_n}\right]\left(\mathbf{x}_0\right)=\left[\left(\int_{\mathbb{T}}\lambda^nd\mathcal{E}(\lambda)\right)\operatorname{pdf}_{\Tilde{\mathbf{x}}_n}\right]\left(\mathbf{x}_0\right).
    \end{align}
    where the first equality in above is because of viewing the $n-$th Gaussian random vector $\Tilde{\mathbf{x}}_n$ as the $n-$th snapshot and hence the relation of exponentiation the Koopman operator is justified. The second equality in above is simply the continuation of what we attempted for any arbitrary observable with the action of the Koopman operator to present in terms of its projected valued measure.
    Now, that we have formulated the decomposition of both $g$ and $\operatorname{pdf}_{\Tilde{\mathbf{x}}_n}$ according to the spectral content of the Koopman operator $\mathcal{K}$ in \eqref{eq_21r} and \eqref{eq_22r}, we can subtract both of them as follows:
    \begin{align*}
        g\left(\mathbf{x}_n\right)-\operatorname{pdf}_{\Tilde{\mathbf{x}}_n}\left(\mathbf{x}_0\right)=&\left[\left(\int_{\mathbb{T}}\lambda^nd\mathcal{E}(\lambda)\right)g\right]\left(\mathbf{x}_0\right)-\left[\left(\int_{\mathbb{T}}\lambda^nd\mathcal{E}(\lambda)\right)\operatorname{pdf}_{\Tilde{\mathbf{x}}_n}\right]\left(\mathbf{x}_0\right)\\
        =&\left[\left(\int_{\mathbb{T}}\lambda^nd\mathcal{E}(\lambda)\right)\left\{g-\operatorname{pdf}_{\Tilde{\mathbf{x}}_n}\right\}\right]\left(\mathbf{x}_0\right)\\
        \implies|g\left(\mathbf{x}_n\right)-\operatorname{pdf}_{\Tilde{\mathbf{x}}_n}\left(\mathbf{x}_0\right)|=&\left|\left[\left(\int_{\mathbb{T}}\lambda^nd\mathcal{E}(\lambda)\right)\left\{g-\operatorname{pdf}_{\Tilde{\mathbf{x}}_n}\right\}\right]\left(\mathbf{x}_0\right)\right|\\
        \leq&\left|\left(\int_{\mathbb{T}}\lambda^nd\mathcal{E}(\lambda)\right)\right|\left\|\left\{g-\operatorname{pdf}_{\Tilde{\mathbf{x}}_n}\right\}\right\|\|\mathbf{x}_0\|.
    \end{align*}
    With $g_n\coloneqq g(\mathbf{x}_n)$ as already defined and given that $g_n\overset{\mu}{\to}\operatorname{pdf}_{\Tilde{\mathbf{x}}_n}$ 
    let there exist an $\epsilon>0$ such that $\left\|\left\{g_n-\operatorname{pdf}_{\Tilde{\mathbf{x}}_n}\right\}\right\|<\epsilon\cdot(\mathcal{E}(\mathbb{T})\|\mathbf{x}_0\|)^{-1}$. Hence the desired result holds.
\end{proof}
Now, we are ready to introduce the normalized Laplacian measure over $\Cn$ formally through which the corresponding RKHS will be constructed in the subsequent section of this paper.
\subsection{Laplacian Kernel as \texorpdfstring{$L^2-$}{}measure}
As of now, the main challenge that we face while we try to understand the dynamical system is limited data availability. Having stated that, in the theme of already-existing DMD algorithms, surprisingly they do not provide a framework which deals with this challenge. Therefore, the present goal of this manuscript is to perform the DMD in the interest of limited data acquisition and analyze the developed framework for better understanding. To achieve this goal, we will leverage the understanding of Koopman operators and its operator-theoretic behavior over the RKHS.

We will be considering the RKHS of holomorphic functions $f:\mathbb{C}^n\to\mathbb{C}$ which are $L^2-$integrable over $\mathbb{C}^n$ against the normalized-probability $L^2-$ measure $d\mu_{\sigma,1,\mathbb{C}^n}(\bm{z})$ given as:
\begin{align}\label{eq_13Laplacemeasure}
    d\mu_{\sigma,1,\mathbb{C}^n}(\bm{z})\coloneqq(2\pi\sigma^2)^{-n}\exp\left(-\frac{\|\bm{z}\|_2}{\sigma}\right)dV(\bm{z}),
\end{align}
\text{~where $\sigma>0$ and $\bm{z}\in\mathbb{C}^n$} and $\|\bmz\|_2=\sqrt{|z_1|^2+\cdots+|z_n|^2}$ for $\bmz=(z_1,\ldots,z_n)\in\Cn$. Here, $dV(\bm{z})$ is the usual $L^2-$ measure over the entire complex plane $\mathbb{C}^n$. 

\subsubsection{A word on Laplacian Kernel viewed as \texorpdfstring{$L^2-$}{}measure}
This perspective of considering the RKHS is very-rich as it is based on determining the \emph{orthonormal basis} which immediately builds the RKHS \cite{aronszajn1950theory,saitoh1997integral,saitoh1983hilbert}. This ideology recently garnered attention in \cite{singh2023appointment} to discuss the functionality of a new kernel function \emph{Generalised Gaussian Radial Basis Function Kernel} on various AI learning architecture routine which was first introduced in \cite{singh2023new}. Additionally, this way of studying the corresponding function spaces has its deep-rooted history with \emph{quantization in quantum dynamic system in \cite{berezin1975general}} which is referred as \emph{Gauge transformations}.
\begin{dfn}[Gauge Transformations \cite{janson1987hankel,berezin1975general,peetre1990berezin}]
    Consider a closed subspace $\mathfrak{H}$ of $L^2(\mu)$ over a measurable space $\Omega$ equipped by the positive measure $\mu$. If we get an essentially equivalent function theory if we replace $f$ by $\phi f$ and $\mu$ by $|\phi|^{-2}\mu$ where $\phi$ is any non-vanishing measurable function, then this transformation is called as the Gauge transformation.
\end{dfn}
To provide more insights on this, we provide the following table \autoref{table-1exponentialtype} with various (yet limited) various kernels that generate their corresponding RKHS. 
{\tiny{
\begin{table}[H]
    \centering
    \caption{Kernel functions viewed as $L^2-$ measure via Gauge transformations}
    \begin{tabular}{p{2.55cm}||p{1.79cm}p{.57cm}p{3.45cm}p{3.1cm}
    l
        S[table-format = 3]
        S[table-format = 2]
        S[table-format = 1.3]
        S[table-format = -2.2]
        S[table-format = 1.3]
        S[table-format = 1.3]
        S[table-format = 0.0]
        }
        \toprule
        \midrule
        \multicolumn{3}{c}{} & 
        \multicolumn{3}{c}{\textsc{Reproducing Kernel Hilbert Spaces}}\\
        \cmidrule(lr){4-8}        
        &  &  & {\tiny{\textsc{Kernel via}}} & {\tiny{\textsc{Reproducing}}} & {\tiny{\textsc{}}} & {} &{} 
        \\
        \tiny{\textsc{$L^2-$ measure}} & {\tiny{\textsc{Parameters}}} & {\tiny{\textsc{Ref.}}} & {$|z|\mapsto\|\cdot-\cdot\|_2$} & {\tiny{\textsc{Kernel}} $K(\cdot,\cdot)$} & {\tiny{\textsc{Domain}}} 
        \\
        \midrule
        $e^{-\nicefrac{|\bm{z}|}{\sigma}}$ & {$\sigma>0$} & {\cite{chen2020deep}}  & {$e^{-\nicefrac{\|\bm{x}-\bm{z}\|_2}{\sigma}}$} & {-} & {$\mathbb{C}^n$} 
        \\
        $e^{-\sigma |\bm{z}|^2}$ & {$\sigma>0$} & {\cite{zhu2012analysis}}  & {$e^{-\sigma\|\bm{x}-\bm{z}\|_2^2}$} & {$\exp\left(\sigma \bm{x}^\top \overline{\bm{z}}\right)$} & {$\mathbb{C}^n$} 
        \\
        $e^{-\sigma^2|\bm{z}|^2}e^{e^{-\sigma_0^2|\bm{z}|^2}-1}$ &  {$\sigma>0,\sigma_0\geq0$} & {\cite{singh2023appointment}}  & {$e^{-\sigma^2\|\bm{x}-\bm{z}\|_2^2}e^{e^{-\sigma_0^2\|\bm{x}-\bm{z}\|_2^2}-1}$} & {$\sum_{N_1,\ldots,N_n=0}^\infty\lambda_{N}\left(\bm{x}^\top\overline{\bm{z}}\right)^N$} &{$\mathbb{C}^n$} 
        \\
        $|z|^{2/q-2}e^{-|z|^{2/q}}$ &  {$q>0$} & {\cite{rosenfeld2018mittag}} & {-} & {$\sum_{n=0}^\infty\nicefrac{x\overline{z}}{\Gamma(nq+1)}$} & {$\mathbb{C}$} 
        \\
        \midrule
        \bottomrule
    \end{tabular}\label{table-1exponentialtype}
\end{table}
}}
At this stage, we take the inspiration from \cite{NEURIPS2020_1006ff12,chen2020deep,geifman2020similarity} to employ the Laplacian Kernel in the $L^2-$measure theoretic sense into our study for performing DMD with limited data acquisition. It should be noted that the practitioners in \cite{belkin2018understand} demands the understanding of various kernel function, in particular Laplacian Kernel; which eventually is fulfilled in \cite{geifman2020similarity,chen2020deep,NEURIPS2020_1006ff12} however their domain of action for the kernel function is limited or should we say \emph{restricted} to $\mathbb{S}^{n-1}\coloneqq\left\{\bm{x}\in\mathbb{R}^n:\|\bm{x}\|_2=1\right\}$. This is due to their interest in collaborating their result with deep neural tangent kernel (NTK) \cite{jacot2018neural}. Following figure compiles with an interactive and attractive assembly of various exponential measure as mathematical functions which finds monumental application in powering machine learning and artificial intelligence algorithms.
\begin{figure}[H]
    \centering
    \includegraphics[scale=.19]{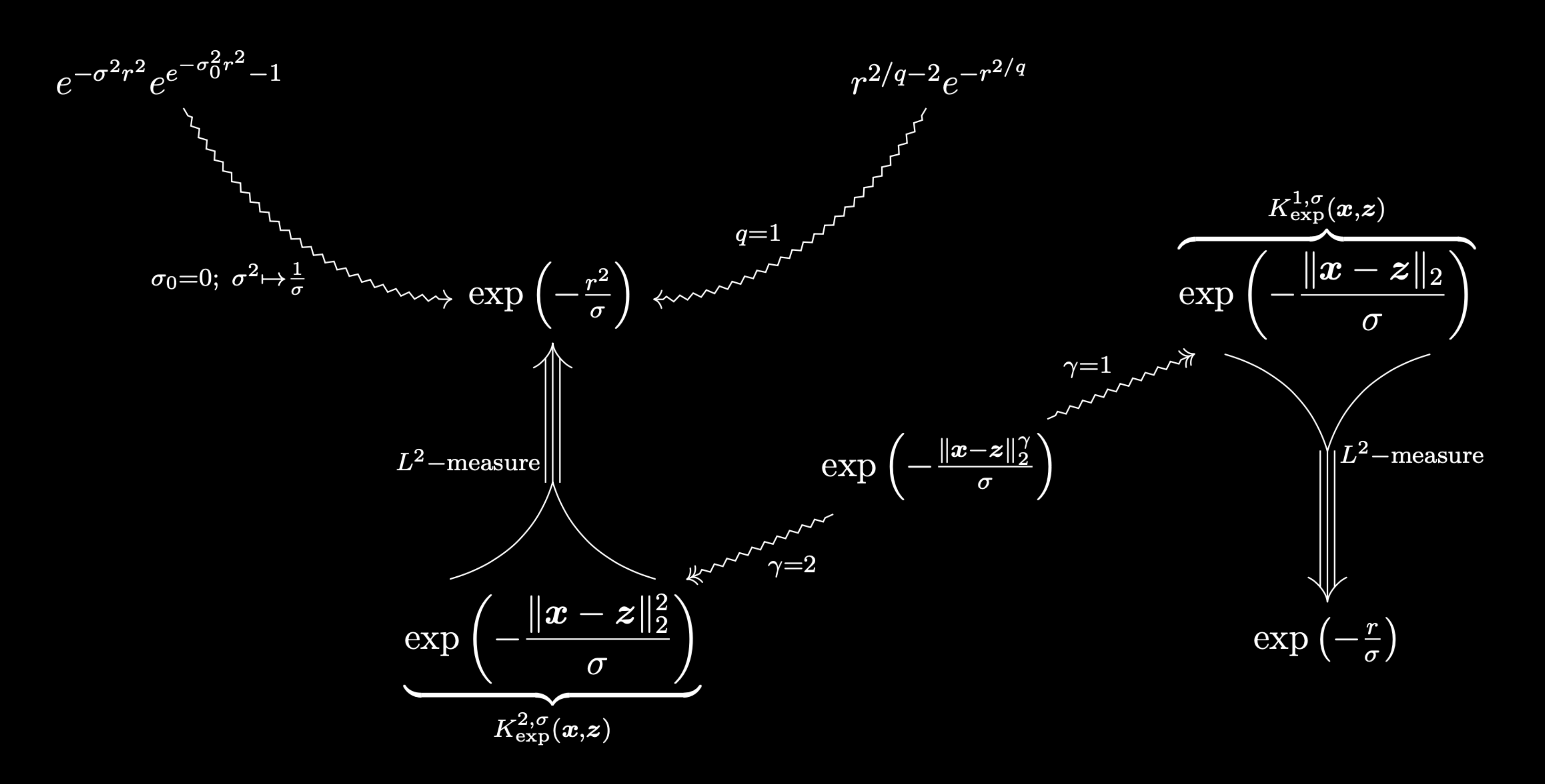}
    \caption{Interactive figure for inter-connection of various exponential type measures as depicted in \textsc{\autoref{table-1exponentialtype}}.}
    \label{fig:enter-label}
\end{figure}
\subsection{eDMD algorithm with Limited Data Acquisition} The algorithm for the Kernelized eDMD with limited data acquisition is actually powered by two key component. The first key component is using the Laplacian Kernel which is given as $K_{\text{exp}}^{1,\sigma}(\bm{x},\bm{z})\coloneqq\exp\left(-\frac{\|\bm{x}-\bm{z}\|_2}{\sigma}\right)$ and the second key component is augmenting the data-set snapshot matrix by padding the Gaussian random matrix of appropriate matrix dimension in the case of not having enough data-set snapshots. Once, this new data-set snapshot matrix is constructed, then follow the already provided Kernelized eDMD algorithm in \cite{williams2015data} or here given in \autoref{algo_kerneleDMD}.
{\small
{\begin{algo}\label{algo_lapkerneleDMD}
The eDMD with limited data acquisition executed by the Laplacian Kernel Function $K_{\text{exp}}^{1,\sigma}(\bm{\cdot},\bm{\cdot})$ and Gaussian random matrix is given as follows:
    \begin{enumerate}
    \item[Step 1] Construct data-set snapshots matrix by padding appropriate number of Gaussian random vectors whose probability density function is given as \eqref{eq_gaussianPDF} or with joint density given as in \eqref{eq_jointgaussianPDF}. Let a finite $\sigma>0$ is also chosen.
        \item[Step 2] Compute the elements of Gram-Matrix ${\bm{\mathfrak{G}}}\coloneqq\left[{\bm{\mathfrak{G}}}\right]_{i\times j}$ and Interaction-Matrix  ${\bm{\mathcal{I}}}\coloneqq\left[{\bm{\mathcal{I}}}\right]_{i\times j}$ that are defined as follows:
        \begin{align*}
            \left[{\bm{\mathfrak{G}}}\right]_{i\times j}&=K_{\text{exp}}^{1,\sigma}\left(\mathbf{x}_i,\mathbf{x}_j\right)\\
            \left[{\bm{\mathcal{I}}}\right]_{i\times j}&=K_{\text{exp}}^{1,\sigma}\left(\mathbf{y}_i,\mathbf{x}_j\right).
        \end{align*}
        \item[Step 3] Determine the spectral observables of the Gram-matrix $\hat{\bm{\mathfrak{G}}}$, i.e. $\bm{Q}$ and $\bm{\Sigma}$.
        \item[Step 4] Construct $\hat{\bm{\mathcal{K}}}$ via \eqref{eq_27rr}.
    \end{enumerate}
\end{algo}}}
\begin{remark}
    As one can easily observe that the above algorithm takes the direct advantage of the Kernelized eDMD algorithm as given in \cite{williams2015data,MatthewO.Williams2015JournalofComputationalDynamics}, however their work is either relied on the choice of polynomial kernel function or the Gaussian radial basis kernel function and so other kernel based DMD algorithm such as in \cite{rosenfeld2021theoretical,rosenfeld2022dynamic} where the aforementioned authors do not explore the same with the Laplacian Kernel Function coupled with the nuances of random matrix theory. Most importantly, it is worth-while to mention that the solutions to the crucial concerns raised in \autoref{ques_1} and \autoref{ques_2} which basically demands that how the Gaussian random matrix can be used to construct the data-set snapshot matrix of the desired matrix size when only a few number of snapshots are provided, helps in giving the robust justification of using the random matrix theory. Additionally, the concerns in \autoref{ques_1} and \autoref{ques_2} can be thought of as an underlying assumption to construct the data matrix so that the practitioner has enough number of data-set snapshots to operate on via the Laplacian Kernel based eDMD.
\end{remark}
\begin{note}
    In this section, we developed the Laplacian Kernel based eDMD which uses the Gaussian random vectors as its snapshots to discover the Koopman modes as directed in \eqref{eq_27rr} in \autoref{prpstin_williamsKoopman}. We will employ this algorithm to construct the governing features via the Koopman modes of fluid flow across cylinder experiment (whose data acquisition related details are given in \autoref{example_ffac}) in which we will provide true $3,~7,~20,~55$ and $151$ number of snapshots out of actual available $151-$snapshots and each snapshot is a column vector of size $89,351$. As already discussed, the respective remaining numbers i.e. $148,~144,~131,~96$ and $0$ of snapshots will be  Gaussian random column vectors which will be augmenting the data-set snapshots matrix. These results will directly be compared with the Gaussian Radial Basis Kernel Function as well. These results will be immediately presented once we have demonstrated the compactification of the Koopman operators over the RKHS generated by the normalized Laplacian measure $d\mu_{\sigma,1,\mathbb{C}^n}(\bm{z})$ in \eqref{eq_13Laplacemeasure}; this detailed analysis performed in the following two sections.
\end{note}
\section{RKHS from the Laplacian measure}\label{section_RKHSLaplaceMeasure}
\subsection{Preliminaries}
Now, with enough inspiration to understand Laplacian Kernel in the perspective of $L^2-$measure, we began to study the same after considering following two important notions, that will be used quite often in this paper.
\subsubsection{Notation}
The set of natural numbers is denoted by either $\mathbb{N}$ or by $\mathbb{Z}_+$ and in union with $0$ is denoted by $\mathbb{W}$, that is $\mathbb{W}\coloneqq0,1,2,\ldots$. We use Kronecker delta $\delta_{nm}$ on non-negative integers $n$ and $m$ to depict that, $\delta_{nm}=1$ whenever $n=m$ and $\delta_{nm}=0$ if $n\neq m$. We denote a complex number $z=x+iy$ where $x$ and $y\in\mathbb{R}$. With that $z$, its conjugate-part is given as $\overline{z}=x-iy$ along with its absolute value as $|z|^2=z\cdot\overline{z}=x^2+y^2$. We reserve symbol $\mathbb{K}$ to treat with choice of fields on which we will operate upon; in particular $\mathbb{K}$ can either be $\mathbb{R}$ or $\mathbb{C}$. The multi-index notation is employed as 
$j=(j_1,\ldots,j_n)$ and $|j|=j_1+\cdots+j_n$. If $\bmz=(z_1,\ldots,z_n)$ then $\bmz^j=(z_1^{j_1},\ldots,z_n^{j_n})$, and $d\bmz = dz_1\cdots dz_n$.
\subsubsection{Tensor Product Notation}\label{subsection_tensorProdNotation}
We recall the \emph{tensor product} between two functions, say $f_1,~f_2:X\to\mathbb{K}$ given as $f_1\otimes f_2:X\times X\to\mathbb{K}$. Then, for all $x,x'\in X$ the tensor product $f_1\otimes f_2$ is defined as $f_1\otimes f_2(x,x')\coloneqq f_1(x)f_2(x')$.
\subsection{Function space corresponding to the Laplacian measure in \texorpdfstring{$L^2-$}{}sense}
As already stated that, we will be incorporating the Laplace Kernel $K_{\text{exp}}^{1,\sigma}(\bm{x},\bm{z})$ for our DMD study with limited data availability which takes inspiration from the fact that it outperforms deep learning architectures when shorter training time is provided. It is not hard to comprehend that we are in void of the understanding of reproducing kernel and RKHS perspective (in \autoref{table-1exponentialtype}) when Laplace Kernel is viewed as the Laplacian measure \eqref{eq_13Laplacemeasure}.

To that end, we provide the inner product between two holomorphic functions $f:\Cn\to\mathbb{C}$ and $g:\Cn\to\mathbb{C}$ associated with this measure as:
\begin{align}\label{eq_normLaplacemeasure}
    \langle f,g\rangle_{\sigma,1,\mathbb{C}^n}\coloneqq\mathcal{N}_{\sigma,1,n}\int_{\mathbb{C}^n}f\overline{g}d\mu_{\sigma,1,\mathbb{C}^n}(\bm{z})\triangleq\mathcal{N}_{\sigma,1,n}\int_{\mathbb{C}^n}f(\bm{z})\overline{g(\bm{z})}e^{-\frac{\|\bmz\|_2}{\sigma}}dV(\bm{z}).
\end{align}
Here, `$\mathcal{N}_{\sigma,1,n}$' is the normalization constant and its explicit value is $\left(2\pi\sigma^2\right)^{-n}$. Once we have define the inner product for the space, the square of the norm for holomorphic function $f:\mathbb{C}^n\to\mathbb{C}$ is:
\begin{align}\label{eq_6_normf}
    \|f\|_{\sigma,1,\mathbb{C}^n}^2\coloneqq&\frac{1}{\left(2\pi\sigma^2\right)^n}
    \int_{\mathbb{C}^n}|f(\bm{z})|^2d\mu_{\sigma,1,\mathbb{C}^n}(\bm{z})\triangleq\frac{1}{\left(2\pi\sigma^2\right)^n}\int_{\mathbb{C}^n}|f(\bm{z})|^2e^{-\frac{\|\bmz\|_2}{\sigma}}dV(\bm{z}).
\end{align}
After defining the norm of the holomorphic function $f:\mathbb{C}^n\to\mathbb{C}$, we can have following as the Hilbert function space.
\begin{align}\label{eq_8hilbertspace}
    H_{\sigma,1,\mathbb{C}^n}\coloneqq\left\{\text{holomorphic function $f:\mathbb{C}^n\to\mathbb{C}$}:\|f\|_{\sigma,1,\mathbb{C}^n}<\infty\right\}.
\end{align}
As of now, we have collected basic yet useful facts regarding the Hilbert space $H_{\sigma,1,\Cn}$. Now, we will formulate towards the orthonormal basis of this Hilbert space.
\subsection{Orthonormal Basis}
We will be providing the orthonormal basis for the Hilbert space $H_{\sigma,1,\Cn}$ generated by the measure $d\mu_{\sigma,1,\mathbb{C}^n}(\bm{z})$ embedded into the $L^2-$measure. The following lemma directs us in that direction.
\begin{lemma}\label{LEMMA_lemmaiii.3}
    For $\sigma>0$, $N$ and $M\in\mathbb{W}$, we have 
    \begin{align*}
        \langle z^N,z^M\rangle_{\sigma,\mathbb{C}}=\int_{\mathbb{C}}z^N\overline{z^M}e^{-\frac{|z|}{\sigma}}dA(z)=2\pi\sigma^2\sigma^{N+M}(N+M+1)!\delta_{NM},
    \end{align*}
    where $\langle\cdot,\cdot\rangle_{\sigma,\mathbb{C}}$ is the same inner-product as given in \eqref{eq_normLaplacemeasure} but over $\mathbb{C}$.
\end{lemma}
\begin{proof}
We will initiate the process of proving this by employing the polar coordinate in $\mathbb{C}$.
\begin{align*}
    \langle z^N,z^M\rangle_{\sigma,\mathbb{C}}=&\int_{\mathbb{C}}z^N\overline{z^M}e^{-\frac{|z|}{\sigma}}dA(z)\\
    =&\int_0^{2\pi}\int_{0}^\infty r^Nr^Me^{i(N-M)\theta}e^{-\frac{r}{\sigma}}rdrd\theta\\
    =&2\pi\delta_{NM}\int_0^{\infty}r^{N+M+1}e^{-\frac{r}{\sigma}}dr\\
    =&2\pi\delta_{NM}\frac{\Gamma(N+M+1+1)}{(\nicefrac{1}{\sigma})^{N+M+2}}\\
    =&2\pi\delta_{NM}\frac{(N+M+1)!}{(\nicefrac{1}{\sigma})^{N+M+2}}\\
    =&2\pi\sigma^2\sigma^{N+M}(N+M+1)!\delta_{NM}.
\end{align*}
Also, from the last step, we can immediately learn that the square of the length of monomial $\left\{z^N\right\}_N$ with respect to the normalized Laplacian measure $e^{-\nicefrac{|z|}{\sigma}}$ in the single-variable complex plane $\mathbb{C}$ is
\begin{align*}
    \|z^N\|_{\sigma,\mathbb{C}}^2=2\pi\sigma^2\sigma^{2N}(2N+1)!.
\end{align*}
\end{proof}
As we have now determined the length of monomial above, we can now provide the orthonormal basis of the $H_{\sigma,1,\mathbb{C}^n}$ by incorporating the tensor-product notation; this is performed in the following theorem.
\begin{theorem}
    For $\sigma>0$ and $N\in\mathbb{W}$, define $\left\{\mathbf{e}_N\right\}_{N\in\mathbb{W}}:\mathbb{C}\to\mathbb{C}$ by
    \begin{align}\label{eq_ONB_lap}
        \mathbf{e}_N(z)\coloneqq\sqrt{\frac{1}{\sigma^{2N}(2N+1)!} }z^N.
    \end{align}
    Then the tensor-product system $\left(\mathbf{e}_{N_1}\otimes\cdots\otimes\mathbf{e}_{N_n}\right)_{N_1,\ldots N_n\geq0}$ is the orthonormal basis of Hilbert space $H_{\sigma,1,\mathbb{C}^n}$.
\end{theorem}
\begin{proof}
    We establish our initial stage of result for single-dimension case to ease our understanding. For this, let us show that $\left\{\mathbf{e}_N\right\}_{N\in\mathbb{W}}$ forms an orthonormal system. So, consider $z\in\mathbb{C}$ and let $M,N\in\mathbb{W}$. Then,
    \begin{align*}
    \langle \mathbf{e}_N,\mathbf{e}_M\rangle_{\sigma,\mathbb{C}}=&\int_{\mathbb{C}}\mathbf{e}_N(z)\overline{\mathbf{e}_M(z)}d\mu_{\sigma,\mathbb{C}}(z)\\
    =&\frac{1}{2\pi\sigma^2}\int_{\mathbb{C}}\sqrt{\frac{1}{\sigma^{2N}(2N+1)!} }z^N\sqrt{\frac{1}{\sigma^{2M}(2M+1)!} }\overline{z^M}e^{-\frac{|z|}{\sigma}}dA(z)\\
    =&\frac{1}{2\pi\sigma^2}\sqrt{\frac{1}{\sigma^{2N}(2N+1)!}}\sqrt{\frac{1}{\sigma^{2M}(2M+1)!}}\int_{\mathbb{C}}z^N\overline{z^M}e^{-\frac{|z|}{\sigma}}dA(z)\\
    =&\frac{1}{2\pi\sigma^2}\sqrt{\frac{1}{\sigma^{2N}(2N+1)!}}\sqrt{\frac{1}{\sigma^{2M}(2M+1)!}}\cdot2\pi\sigma^2\sigma^{N+M}(N+M+1)!\delta_{NM}.\\
    =&\begin{cases}
        1
            &\text{if $N=M$}\\
            0&\text{otherwise}
        \end{cases}\quad\text{\emph{(use \autoref{LEMMA_lemmaiii.3})}}.
    \end{align*}
The above result concludes that $\left\{\mathbf{e}_n\right\}_{n\in\mathbb{W}}$ is actually an orthonormal system with respect to normalized Laplacian measure in complex plane $\mathbb{C}$. Now, we have to establish that this orthonormal system is also complete. So, for this, pick $f\in H_{\sigma,1,\mathbb{C}}$ with $f(z)=\sum_{l=0}^\infty a_lz^l$ and observe that
\begin{align*}
    \langle f,\mathbf{e}_N\rangle_{\sigma,\mathbb{C}}=&\frac{1}{2\pi\sigma^2}\int_{\mathbb{C}}f(z)\overline{\mathbf{e}_N(z)}d\mu_{\sigma,\mathbb{C}}\\
    =&\frac{1}{2\pi\sigma^2}\sqrt{\frac{1}{\sigma^{2N}(2N+1)!}}\int_{\mathbb{C}}f(z)\overline{z^N}e^{-\frac{|z|}{\sigma}}dA(z)\\
    =&\frac{1}{2\pi\sigma^2}\sqrt{\frac{1}{\sigma^{2N}(2N+1)!}}\sum_{l=0}^\infty a_l\int_{\mathbb{C}}z^l\overline{z^N}e^{-\frac{|z|}{\sigma}}dA(z)\\
    =&\frac{1}{2\pi\sigma^2}\sqrt{\frac{1}{\sigma^{2N}(2N+1)!}}\sum_{l=0}^\infty a_l\cdot 2\pi\sigma^2\sigma^{l+N}(l+N+1)!\delta_{lN}\\
    =&\sqrt{{\sigma^{2N}(2N+1)!}}a_N.
\end{align*}For a given $\sigma>0$, since the constant $\sqrt{\frac{1}{\sigma^{2N}(2N+1)!}}\neq0$ for any $N\in\mathbb{W}$, therefore this directly imply that $\langle f,\mathbf{e}_N\rangle_{\sigma,\mathbb{C}}=0$ if and only if $a_N=0$, which results into $f\equiv0$. Hence, $\left\{\mathbf{e}_N\right\}_{N\in\mathbb{W}}$ is complete.

Now, we establish the same result but in $n-$dimensional case and to this end, we will employ the tensor product notation \textsc{\autoref{subsection_tensorProdNotation}} by considering multi-index notation for $N,M\in\mathbb{W}$ as follows:
\begin{align*}
\langle\mathbf{e}_{N_1}\otimes\cdots\otimes\mathbf{e}_{N_n},\mathbf{e}_{M_1}\otimes\cdots\otimes\mathbf{e}_{M_n}\rangle_{\sigma,\mathbb{C}^n}=\prod_{j=1}^n\langle\mathbf{e}_{N_j},\mathbf{e}_{M_j}\rangle_{\sigma,\mathbb{C}}.
\end{align*}
Hence the orthonormality of $\left\{\mathbf{e}_{N_1}\otimes\cdots\otimes\mathbf{e}_{N_n}\right\}_{N_1,\ldots N_n\in,\mathbb{W}^n}$ is established due to the orthonormality of each individual $\langle\mathbf{e}_{N_j},\mathbf{e}_{M_j}\rangle_{\sigma,\mathbb{C}}$. We still need to ensure that this $n-$dimensional orthonormal system is complete. We observe
    \begin{align*}
        \langle f,\mathbf{e}_{N_1}\otimes\cdots\otimes\mathbf{e}_{N_n}\rangle_{\sigma,1,\mathbb{C}^n}
        =&\left(\frac{1}{2\pi\sigma^2}\right)^n\int_{\mathbb{C}^n}f(\bm{z})\overline{\mathbf{e}_{N_1}\otimes\cdots\otimes\mathbf{e}_{N_n}\left(\bm{z}\right)}d\mu_{\sigma,1,\mathbb{C}^n}\left(\bm{z}\right)\\
        =&\left(\frac{1}{2\pi\sigma^2}\right)^n\sum_{l_1,\ldots,l_n}^\infty
        a_{l_1,\ldots,l_n}
        \mathds{I}_{l,n}, 
    \end{align*}
    where $\mathds{I}_{l,n}=
\int_{\mathbb{C}^n}\bm{z}^l\left(\mathbf{e}_{N_1}\otimes\cdots\otimes\mathbf{e}_{N_n}\left(\overline{\bm{z}}\right)\right)d\mu_{\sigma,1,\mathbb{C}^n}\left(\bm{z}\right)$. We further can simplify $\mathds{I}_{l,n}$ as:    
\begin{align*}
        \mathds{I}_{l,n}=&
        \int_{\mathbb{C}^n}\bm{z}^l\mathbf{e}_{N_1}(\overline{z_1})\wedge\cdots\wedge\mathbf{e}_{N_n}(\overline{z_d})
        d\mu_{\sigma,\mathbb{C}}(z_1)\wedge\cdots\wedge d\mu_{\sigma,\mathbb{C}}(z_n)\\
        =&\prod_{j=1}^n
        \left(\int_\mathbb{C}z_j^{l_j}\mathbf{e}_{N_j}(\overline{z_j})d\mu_{\sigma,\mathbb{C}}(z_j)\right)\\
        =&\prod_{j=1}^n\left(\int_{\mathbb{C}}z_j^{l_j}\overline{z_j}^{N_j}d\mu_{\sigma,\mathbb{C}}(z_j)\right)\\
        =&\prod_{j=1}^n\left(2\pi\sigma^2\sigma^{l_j+N_j}(l_j+N_j+1)!\right)\delta_{l_jN_j}.
    \end{align*}
    Finally, \begin{align*}
    \left(\frac{1}{2\pi\sigma^2}\right)^n\sum_{l_1,\ldots,l_n}^\infty
        a_{l_1,\ldots,l_n}\mathds{I}_{l,n}=&\left(\prod_{j=1}^n\sqrt{{\sigma^{2N_j}(2N_j+1)!}}
        \right)a_{N_1,\ldots,N_n}.
    \end{align*}
    The further result for completeness in $n-$dimension follows a routine procedure from single-dimension case as already discussed before. 
    \end{proof} 
    Now, that we have determined the orthonormal basis for the Hilbert space $H_{\sigma,1,\Cn}$ we will use it to construct the reproducing kernel for $H_{\sigma,1,\Cn}$ which eventually makes it the RKHS.
    \subsection{Reproducing Kernel}The result of the above theorem provides us the orthonormal basis for the RKHS $H_{\sigma,1,\mathbb{C}^n}$. Once we have the information for the orthonormal basis of the RKHS, we can employ the \textsc{Moore-Aronszajn} Theorem to construct the reproducing kernel for the associated RKHS. Therefore, in the light of determining the reproducing kernel for $H_{\sigma,1,\Cn}$, the following theorem provides the same for $H_{\sigma,1,\mathbb{C}^n}$ by using the \textsc{Moore-Aronszajn} Theorem.
    \begin{theorem}\label{theorem_RKviaOrthonormalbasis}
    For $\sigma>0$, let $\bmz=(z_1,\ldots,z_n)$ and $\bm{w}=(w_1,\ldots,w_n)$ be in $\Cn$. Then the reproducing kernel for RKHS $H_{\sigma,1,\mathbb{C}^n}$ is given as 
    \begin{align}\label{eq_21}
        K_{\bm{w}}^\sigma\left(\bm{z}\right)=K^\sigma\left(\bm{z},\bm{w}\right)\coloneqq\frac{\sinh{\left(\sqrt{ \frac{\langle\bm{z},\bm{w}\rangle_{\mathbb{C}^n}}{\sigma^2}}
        \right)}}{\sqrt{\frac{\langle\bm{z},\bm{w}\rangle_{\mathbb{C}^n }}{\sigma^2}}
        },
    \end{align}
    where $\langle\bmz,\bm{w}\rangle_{\Cn}=\bmz\overline{\bm{w}^\top}=\sum_{i=1}^nz_i\overline{w_i}$.
\end{theorem}
\begin{proof}
    We will be needing the result of \autoref{theorem_aronsjan} to prove this result. In \autoref{theorem_aronsjan}, we take $\mathcal{I}$ as $\mathbb{W}$ as the index set, then we have following:
    \begin{align*}
        K^\sigma\left(\bm{z},\bm{w}\right)=&\sum_{N\in \mathbb{W}}\mathbf{e}_N\left(\bm{z}\right)\overline{\mathbf{e}_N\left(\bm{w}\right)}\\
        =&\sum_{N\in \mathbb{W}}\sqrt{\frac{1}{\sigma^{2N}(2N+1)!}}\bm{z}^N\sqrt{\frac{1}{\sigma^{2N}(2N+1)!}}\overline{\bm{w}^N}\\
        =&\sum_{N=0}^\infty{\frac{1}{(2N+1)!}}\left(\frac{\bm{z}\overline{\bm{w}^\top}}{\sigma^2}\right)^N\\
        =&\frac{\sinh\left(\sqrt{\nicefrac{\bm{z}\overline{\bm{w}^\top}}{\sigma^2}}
        \right)}{\sqrt{\left(\nicefrac{\bm{z}\overline{\bm{w}^\top}}{\sigma^2}\right)}
        }.
    \end{align*}
Obviously, the other notable and useful way of representing the reproducing kernel $K^\sigma\left(\bm{z},\bm{w}\right)$ is by writing the argument in the scalar inner-product format of $\langle\bm{z},\bm{w}\rangle_{\mathbb{C}^n}=\bm{z}\overline{\bm{w}^\top}$ over $\Cn$. Hence,
\begin{align*}
    K_{\bm{w}}^\sigma\left(\bm{z}\right)=K^\sigma(\bm{z},\bm{w})=\frac{\sinh{\left(\sqrt{\dfrac{\langle\bm{z},\bm{w}\rangle_{\mathbb{C}^n}}{\sigma^2}}
    \right)}}{\sqrt{\dfrac{\langle\bm{z},\bm{w}\rangle_{\mathbb{C}^n}}{\sigma^2}}}.
\end{align*}
Hence, the result is now established.
\end{proof}
\subsection{Weakly converging sequence in RKHS \texorpdfstring{$H_{\sigma,1,\mathbb{C}^n}$}{}}
As now, we have eventually established the reproducing kernel for RKHS $H_{\sigma,1,\mathbb{C}^n}$ we are now in position to describe about the weakly converging sequence in RKHS $H_{\sigma,1,\mathbb{C}^n}$ by the application of \autoref{lemma_weaklyconverging}.
\begin{theorem}
    Let $\sigma>0$ and $H_{\sigma,1,\mathbb{C}^n}$ be the RKHS as defined in \eqref{eq_8hilbertspace} with reproducing kernel $K_{\bm{w}}^\sigma(\bm{z})=\frac{\sinh{\left(\nicefrac{\langle\bm{z},\bm{w}\rangle_{\mathbb{C}^n}}{\sigma^2}\right)^{\nicefrac{1}{2}}}}{\left(\nicefrac{\langle\bm{z},\bm{w}\rangle_{\mathbb{C}^n}}{\sigma^2}\right)^{\nicefrac{1}{2}}}$ for $\bm{z},\bm{w}\in\mathbb{C}^n$. Let $\left\{\bm{z}_M\right\}_M$ be the sequence of points in $\mathbb{C}^n$ such that $\|\bm{z}_M\|_2\to\infty$ as $M\to\infty$, then following holds true:
    \begin{align*}
    \lim_{\|\bm{z}_M\|_2\to\infty}\left[\frac{\sinh{\nicefrac{\|\bm{z}_M\|_2}{\sigma}}}{\nicefrac{\|\bm{z}_M\|_2}{\sigma}}\right]^{-\frac{1}{2}}K_{\bm{z}_M}^\sigma=0.
    \end{align*}
\end{theorem} 
\begin{proof}
Suppose $\sigma>0$. Fix a $\bm{z}_M\in\mathbb{C}^n$ for some $M\in\mathbb{W}$. We will use \eqref{eq_4563} to determine $\|K_{\bm{z}}\|$ as follows:
\begin{align*}
    \|K_{\bm{z}_M}^\sigma\|^2=\langle K_{\bm{z}_M}^\sigma,K_{\bm{z}_M}^\sigma\rangle=K_{\bm{z}_M}^\sigma\left({\bm{z}}_M\right)=&\frac{\sinh\left(\sqrt{\nicefrac{\langle\bm{z}_M,\bm{z}_M\rangle_{\mathbb{C}^n}}{\sigma^2}}\right)}{\sqrt{\nicefrac{\langle\bm{z}_M,\bm{z}_M\rangle_{\mathbb{C}^n}}{\sigma^2}}}\\=&\frac{\sinh\left(\sqrt{\nicefrac{\|\bm{z}_M\|_2^2}{\sigma^2}}
    \right)}{\sqrt{\nicefrac{\|\bm{z}_M\|_2^2}{\sigma^2}}}\\=&\frac{\sinh\left(\nicefrac{\|\bm{z}_M\|_2}{\sigma}\right)}{\nicefrac{\|\bm{z}_M\|_2}{\sigma}}.\numberthis\label{eq_25}
\end{align*}
Once we have determined the quantity $\|K_{\bm{z}}^\sigma\|$, which is $\|K_{\bm{z}}^\sigma\|=\left(\frac{\sinh\left(\nicefrac{\|\bm{z}_M\|_2}{\sigma}\right)}{\nicefrac{\|\bm{z}_M\|_2}{\sigma}}\right)^{\nicefrac{1}{2}}$, we can combine this result with \autoref{lemma_weaklyconverging} in the light of $M\to\infty$ and therefore the desired result is achieved; demonstrated as follows:
\begin{align*}
    \underbrace{0=\lim_{\|\bm{z}_M\|_2\to\infty}\|K_{\bm{z}_M}^\sigma\|^{-1}K_{\bm{z}_M}^\sigma}_{\textsc{\autoref{lemma_weaklyconverging}}}
    \overbrace{=\lim_{\|\bm{z}_M\|_2\to\infty}\left(\frac{\sinh\left(\nicefrac{\|\bm{z}_M\|_2}{\sigma}\right)}{\nicefrac{\|\bm{z}_M\|_2}{\sigma}}\right)^{-\frac{1}{2}}K_{\bm{z}_M}^\sigma=0}^{\textsc{From \eqref{eq_25}}}.
\end{align*}
Hence, the desired result is achieved.
\end{proof}
\subsection{Point-wise evaluation inequality}
Now, that we have figured out the norm of the reproducing kernel $K_{\bmz}^\sigma$ in \eqref{eq_25} for the RKHS $H_{\sigma,1,\Cn}$, we can provide the point-wise estimate for the holomorphic function $f:\Cn\to\mathbb{C}$ that lives in $H_{\sigma,1,\Cn}$ in terms of $\|f\|$ and $\|K_{\bmz}^\sigma\|$. This discussion is performed in the immediate theorem.
\begin{theorem}Let $\sigma>0$. 
    For all $\bmz\in\Cn$ and consider $f$ in the RKHS $H_{\sigma,1,\Cn}$, then following point-wise evaluation inequality holds:
    \begin{align*}
        |f(\bmz)|\leq\|K_{\bmz}^\sigma\|\|f\|,
    \end{align*}
    where $K_{\bmz}^\sigma$ is the reproducing kernel for the RKHS $H_{\sigma,1,\Cn}$.
\end{theorem}
\begin{proof}
    Since, it is provided that $f\in H_{\sigma,1,\Cn}$. This immediately imply that for any $\bmz\in\Cn$, we have following due to the virtue of reproducing property of the reproducing kernel $K_{\bmz}^\sigma$:
    \begin{align*}
        f(\bmz)&\overset{\textsc{RP in \eqref{eq_4563}}}{=}\langle f,K_{\bmz}^\sigma\rangle\\
        \implies|f(\bmz)|&\leq|\langle f,K_{\bmz}^\sigma\rangle|\overset{\textsc{C-B-S}}{\leq}\|f\|\|K_{\bmz}^\sigma\|,
    \end{align*}
    where C-B-S simply implies for the Cauchy-Bunyakowsky-Schwarz inequality (cf. \cite[Chapter 1, Page 3]{conway2019course}). 
\end{proof} 
\subsection{Bounds for norm of reproducing kernel}As we have now gathered the structure of the reproducing kernel $K_{\bmz}^\sigma$ for the RKHS $H_{\sigma,1,\Cn}$, we shall determine the bounds of the norm of $K_{\bmz}^\sigma$ as well. In doing so, the result derived in \eqref{eq_25} can be used to establish an interesting results in terms of lower and upper bound of $\|K_{\bm{z}}^\sigma\|$ for the reproducing kernel $K_{\bm{z}}^\sigma$ of $H_{\sigma,1,\mathbb{C}^n}$. This is demonstrated as follows:
\begin{lemma}\label{lemma_kernel_lowerandupperbound}
Let $\sigma>0$. For all ${\bm{z}}\in\mathbb{C}^n$, the norm of the reproducing kernel $K_{\bm{z}}^\sigma$ of RKHS $H_{\sigma,1,\mathbb{C}^n}$ satisfies the following inequality:
    \begin{align}\label{eq_26}
        \exp\left(\frac{1}{2}\left[\frac{\|\bm{z}\|_2}{\sigma}\coth{\frac{\|\bm{z}\|_2}{\sigma}}-1\right]\right)<_{(\text{\textsc{1}})}
\|K_{\bm{z}}^\sigma\|^2
        <_{(\text{\textsc{2}})}\exp{\left(\frac{\|\bm{z}\|_2^2}{6\sigma^2}\right)}.
    \end{align}
\end{lemma}
\begin{proof}
We will employ the Weierstrass factorization theorem (cf. \cite[Chapter 5]{stein2010complex}) for the function $\nicefrac{\sinh{\zeta}}{\zeta}$ followed by taking the `$\log$' as demonstrated follows:
\begin{align*}
            \frac{\sinh{\zeta}}{\zeta}=\prod_{j=1}^\infty\left(1+\frac{\zeta^2}{j^2\pi^2}\right)
            \implies\log\frac{\sinh{\zeta}}{\zeta}=\sum_{j=1}^\infty\log\left(1+\frac{\zeta^2}{j^2\pi^2}\right).
        \end{align*}
The respective proofs for both lower and upper bound inequality are given as follows. We begin with the upper bound inequality:
\begin{center}
    {\text{\textsc{Inequality (2)}}}
\end{center}To establish {\text{\textsc{Inequality (2)}}}, we define a function $\mathfrak{g}_2:\mathbb{R}_+\cup\left\{0\right\}\to\mathbb{R}_+\cup\left\{0\right\}$ by $\mathfrak{g}_2\left(x\right)=x-\log(1+x)$ for $x\in\mathbb{R}_+\cup\left\{0\right\}$. Then a simple calculation demonstrates that $\mathfrak{g}_2(x)\geq0$ whenever $x\in\mathbb{R}_+\cup\left\{0\right\}$ and hence we can easily conclude now that $\mathfrak{g}_2\left(\frac{\zeta^2}{j^2\pi^2}\right)>0$ whenever $j\geq1.$ Therefore $\log\left(1+\frac{\zeta^2}{j^2\pi^2}\right)<\frac{\zeta^2}{j^2\pi^2}$. Hence, taking the summation of this over $j\in\mathbb{Z}_+$, we have
\begin{align*}
    \log\frac{\sinh{\zeta}}{\zeta}=\sum_{j=1}^\infty\log\left(1+\frac{\zeta^2}{j^2\pi^2}\right)<\sum_{j=1}^\infty\frac{\zeta^2}{j^2\pi^2}=\frac{\zeta^2}{\pi^2}\sum_{j=1}^\infty\frac{1}{j^2}=\frac{\zeta^2}{\pi^2}\cdot\frac{\pi^2}{6}=\frac{\zeta^2}{6}.
\end{align*}
Taking the exponentiation of above yields $\frac{\sinh{\zeta}}{\zeta}<\exp\left(\frac{\zeta^2}{6}\right)
$ and with $\zeta\mapsto\nicefrac{\|\bm{z}\|_{\mathbb{C}^n}}{\sigma}$, we have finally:
\begin{align*}
    \|K_{\bm{z}}^\sigma\|^2=\frac{\sinh{\left(\dfrac{\|\bm{z}\|_2}{\sigma}\right)}}{\left(\dfrac{\|\bm{z}\|_{2}}{\sigma}\right)}<\exp{\left(\frac{\|\bm{z}\|_{2}^2}{6\sigma^2}\right)}.\numberthis\label{eq_27}
\end{align*}
The above inequality provides the upper bound for the norm of reproducing kernel $K_{\bmz}^\sigma$ for the RKHS $H_{\sigma,1,\Cn}$. Now, we provide the lower bound for the same as follows:
\begin{center}
    {\text{\textsc{Inequality (1)}}}
\end{center}
To establish {\text{\textsc{Inequality (1)}}}, we define a function $\mathfrak{g}_1:\mathbb{R}_+\cup\left\{0\right\}\to\mathbb{R}_+\cup\left\{0\right\}$ by $\mathfrak{g}_1\left(x\right)=\log(1+x)-\frac{x}{1+x}$ for $x\in\mathbb{R}_+\cup\left\{0\right\}$. Then a simple calculation demonstrates that $\mathfrak{g}_1(x)\geq0$ whenever $x\in\mathbb{R}_+\cup\left\{0\right\}$ and hence we can easily conclude now that $\mathfrak{g}_1\left(\frac{\zeta^2}{j^2\pi^2}\right)>0$ whenever $j\geq1.$ Therefore,  $\frac{\zeta^2}{j^2\pi^2+\zeta^2}<\log\left(1+\frac{\zeta^2}{j^2\pi^2}\right)$. Hence, taking the summation of this over $j\in\mathbb{Z}_+$, we have
\begin{align*}
    \log\frac{\sinh{\zeta}}{\zeta}=&\sum_{j=1}^\infty\log\left(1+\frac{\zeta^2}{j^2\pi^2}\right)\\>&\sum_{j=1}^\infty\frac{\zeta^2}{j^2\pi^2+\zeta^2}\\=&\frac{\zeta^2}{\pi^2}\sum_{j=1}^\infty\frac{1}{j^2+\frac{\zeta^2}{\pi^2}}\\
    =&\frac{\zeta^2}{\pi^2}\left[\frac{1}{2}\left[\frac{\pi}{\frac{\zeta}{\pi}}\coth{\left(\frac{\zeta}{\pi}\pi\right)}-\frac{1}{\frac{\zeta^2}{\pi^2}}\right]\right]\quad\textit{(use \cite[Page 128, Probelm-6]{stein2010complex})}\\
    =&\frac{1}{2}\left[\zeta\coth{\zeta}-1\right].
\end{align*}
Again, taking the exponentiation of above yields $\exp\left(\frac{1}{2}\left[\zeta\coth{\zeta}-1\right]\right)<\frac{\sinh{\zeta}}{\zeta}$ and with $\zeta\mapsto\nicefrac{\|\bm{z}\|_{2}}{\sigma}$, we have finally:
\begin{align*}
    \exp\left(\frac{1}{2}\left[\frac{\|\bm{z}\|_{2}}{\sigma}\coth{\frac{\|\bm{z}\|_{2}}{\sigma}}-1\right]\right)<\frac{\sinh{\dfrac{\|\bm{z}\|_{2}}{\sigma}}}{\dfrac{\|\bm{z}\|_{2}}{\sigma}}=\|K_{\bm{z}}^\sigma\|^2.\numberthis\label{eq_28}
\end{align*}
Combining \eqref{eq_27} and \eqref{eq_28} produces the desired result.
\end{proof}
It is now an easy exercise to observe that by taking the square-root in \eqref{eq_26}, following can be achieved,
\begin{align*}
    \exp\frac{1}{4}\left[\frac{\|\bm{z}\|_{2}}{\sigma}\coth{\frac{\|\bm{z}\|_{2}}{\sigma}}-1\right]<
\|K_{\bm{z}}^\sigma\|<\exp{\frac{\|\bm{z}\|_{2}^2}{12\sigma^2}}.
\end{align*}
\section{Koopman Operators on RKHS from the Laplacian measure}\label{section_Koopman}
We have already given the definition of Koopman or \emph{composition} operators in \autoref{def_KoopmanOperators} between the Hilbert space given by $L^2\left(\mu\right)$ in the context of dynamical system under the additional assumption of sampling-flow assumption. Traditionally, the study of the Koopman operators over the Hilbert space has been at the center stage of core operator-theoretic analysis, where mathematicians try to understand the function theoretic properties of the symbol of the Koopman operator that may impact it and vice-versa (cf. \cite{bayart2010parabolic,bayart2011composition,shapiro2012composition,cowen2019composition,chacon2007composition,doan2017composition}).

With $L^2-$measure as the usual Gaussian measure as $e^{-\sigma|\bm{z}|^2}$ (\emph{2nd entry in \autoref{table-1exponentialtype}}), we get the very-much-celebrated \emph{\textbf{Bergman-Segal-Fock space}} $\mathcal{F}^2\left(\mathbb{C}^n\right)$ (cf. \cite{janson1987hankel,zhu2012analysis}), where the Koopman operators $\mathcal{K}_\varphi$ induced by a holomorphic function $\varphi:\mathbb{C}^n\to\mathbb{C}^n$ have been completely characterised by \cite{carswell2003composition}. 

Their results are given as follows for $\mathcal{K}_\varphi:\mathcal{F}^2\left(\mathbb{C}^n\right)\to\mathcal{F}^2\left(\mathbb{C}^n\right)$:
\begin{theorem}[Koopman operators over the Bergman-Segal-Fock space of holomorphic functions \cite{carswell2003composition}]\label{theorem_CMS}
    Suppose $\varphi:\mathbb{C}^n\to\mathbb{C}^n$ be a holomorphic mapping. Then
    \begin{enumerate}
        \item $\mathcal{K}_{\varphi}$ is \textsc{\textbf{bounded}} on $\mathcal{F}^2\left(\mathbb{C}^n\right)$ if and only if $\varphi(\bm{z})=A\bm{z}+b$ where $A$ is an $n\times n$ matrix with $\|A\|\leq1$ and $b$ is $n\times1$ complex vector such that $\langle A\zeta,b\rangle=0$ whenever $|A\zeta|=|\zeta|$.
        \item $\mathcal{K}_\varphi$ is \textbf{\textsc{compact}} on $\mathcal{F}^2\left(\mathbb{C}^n\right)$ if and only if $\varphi(\bm{z})=A\bm{z}+b$ where $\|A\|<1$ and $b$ is any $n\times1$ complex vector.
    \end{enumerate}
\end{theorem}
We take the motivation from \autoref{theorem_CMS} to address the concern of the operator-theoretic characterisation for the Koopman operators over the newly-developed RKHS $H_{\sigma,1,\mathbb{C}^n}$ out of the normalized Laplacian measure. 
\subsection{Koopman operators over RKHS \texorpdfstring{$H_{\sigma,1,\mathbb{C}^n}$}{}}
In this subsection, we recall that if $\varphi:\mathbb{C}^n\to\mathbb{C}^n$ is a holomorphic function in which every coordinate function of it are holomorphic from $\Cn\to\mathbb{C}$, then the Koopman operator induced by $\varphi$ is given as $\mathcal{K}_{\varphi}$ which takes the mapping from the domain of itself in $H_{\sigma,1,\mathbb{C}^n}$ to itself. Recall that $H_{\sigma,1,\mathbb{C}^n}$ is the RKHS whose reproducing kernel function is given as $K_{\bm{w}}^\sigma\left(\bm{z}\right)=\frac{\sinh{\left(\sqrt{\nicefrac{\langle\bm{z},\bm{w}\rangle_{\mathbb{C}^n}}{\sigma^2}}
    \right)}}{\sqrt{\nicefrac{\langle\bm{z},\bm{w}\rangle_{\mathbb{C}^n}}{\sigma^2}}}.$ We provide the definition of the Koopman operators in the setting of the newly developed RKHS $H_{\sigma,1,\mathbb{C}^n}$ as follows.
\begin{dfn}
    Let $\varphi:\mathbb{C}^n\to\mathbb{C}^n$ be a holomorphic function in which every coordinate function of it are holomorphic functions from $\Cn\to\mathbb{C}$. Then, the Koopman operator induced by $\varphi$ is denoted by $\mathcal{K}_{\varphi}:\mathcal{D}\left(\mathcal{K}_\varphi\right)\subset H_{\sigma,1,\mathbb{C}^n}\to H_{\sigma,1,\mathbb{C}^n}$ and is the linear operator defined by
    \begin{align*}
        \mathcal{K}_{\varphi}(f)\coloneqq f\circ\varphi
    \end{align*}
    and the domain of $\mathcal{K}_\varphi$ is $\mathcal{D}\left(\mathcal{K}_\varphi\right)$ is given as
\begin{align*}
    \mathcal{D}\left(\mathcal{K}_\varphi\right)\coloneqq\left\{f\in H_{\sigma,1,\mathbb{C}^n}:f\circ\varphi\in H_{\sigma,1,\mathbb{C}^n}\right\}.
\end{align*}
\end{dfn}
As we learn that the Koopman operators are closed, then one can employ the closed graph theorem to result into its boundedness and therefore we have the well-defined adjoint relationship. In this case, the adjoint of the Koopman operator over the RKHS $H_{\sigma,1,\mathbb{C}^n}$ is given in the following lemma. 
    \begin{lemma}\label{lemma3.3}
        Let $\varphi:\mathbb{C}^n\to\mathbb{C}^n$ be the holomorphic mapping with every coordinate function as holomorphic function from $\Cn\to\mathbb{C}$. Then, the Koopman operator $\mathcal{K}_{\varphi}:\mathcal{D}\left(\mathcal{K}_\varphi\right)\to H_{\sigma,1,\mathbb{C}^n}$ induced by $\varphi$ satisfies the following adjoint relation with the reproducing kernel $K_{\bm{z}}^\sigma$ of RKHS $H_{\sigma,1,\mathbb{C}^n}$:
    \begin{align*}
        \mathcal{K}_{\varphi}^*K_{\bm{z}}^\sigma=K_{\varphi\left(\bm{z}\right)}^\sigma.
    \end{align*}
    \end{lemma}
    \begin{proof}
        Let $f\in H_{\sigma,1,\mathbb{C}^n}$ and pick an arbitrary $\bm{z}\in\mathbb{C}^n$, then
        \begin{align*}
            \langle f,\mathcal{K}_{\varphi}^*K_{\bm{z}}^\sigma\rangle=\langle\mathcal{K}_{\varphi}f,K_{\bm{z}}^\sigma\rangle=\mathcal{K}_\varphi f(\bm{z})=f\left(\varphi\left(\bm{z}\right)\right)=\langle f,K_{\varphi\left(\bm{z}\right)}^\sigma\rangle.
        \end{align*}
        Hence, the desired result is achieved.
    \end{proof}
    \autoref{lemma3.3} makes us realize that the set of reproducing kernel $\left\{K_{\bm{z}}^\sigma:\bm{z}\in\Cn\right\}$ is invariant under the adjoint of $\mathcal{K}_\varphi$ \cite[Chapter 1]{cowen1983composition}. Additionally the relation defined in the above lemma provides the unique relationship of $\koop^*K_{\bmz}^\sigma$ via the inner product of the RKHS $H_{\sigma,1,\Cn}$ and hence we can have the Koopman operator $\koop$ as to be densely defined over the RKHS $H_{\sigma,1,\Cn}$ \cite[Chapter 13, Page 348]{rudin1991functional} and also the adjoint of the Koopman operator is now close in the RKHS $H_{\sigma,1,\Cn}$ \cite[Theorem 13.9]{rudin1991functional}. 
    \subsection{Boundedness of Koopman operators over RKHS \texorpdfstring{$H_{\sigma,1,\mathbb{C}^n}$}{}}
    We begin by providing a preparatory result that will be further used in upcoming proves.
    \begin{lemma}\label{lemma_3.3}
Let $\Psi:\mathbb{C}^n\to\mathbb{C}$ be holomorphic on a complex domain containing the closed unit ball.
If $\Psi(\bmz)=\sum_{|j|=0}^\infty a_jz^j$, then
\[(2\pi)^{-n}\int_{-\pi}^{\pi}\cdots \int_{-\pi}^{\pi} |\Psi(\bmz)|^2 d{\bm{\vartheta}}=\sum_{|j|=0}|a_j|^2r^{2j},\quad a_j\in \mathbb{C}.\]
\end{lemma}
 \begin{proof}
 Recall that $(2\pi)^{-1}\int_{-\pi}^\pi e^{i(j-k)\vartheta}\, d\vartheta = \delta_{j,k}$. 
 As $\Psi(\bmz)=\sum_{|j|=0}a_j\bmz^j$, this implies that we have following consequences:
 \begin{align*}
     |\Psi(\bmz)|^2&=\Psi(\bmz)\overline{\Psi(\bmz)}\\
     &=\sum_{|j|=0}\sum_{|k|=0}a_j\overline{a_k}z^j\overline{z}^k\\
     &=\sum_{|j|=0}\sum_{|k|=0}a_j\overline{a_k}r^{j+k}e^{i(j-k)\theta}\\
     &=\sum_{|j|=0}\sum_{|k|=0}a_j\overline{a_k}r_1^{j_1+k_1}\cdots r_n^{j_n+k_n}\left(\prod_{l=1}^ne^{i(j_l-k_l)\theta_l}\right).
     \end{align*}
     So, 
     \begin{align*}
     &\int_{-\pi}^\pi\cdots\int_{-\pi}^\pi|\Psi(\bmz)|^2d\bm{\vartheta}\\
     =&\sum_{|j|=0}\sum_{|k|=0}a_j\overline{a_k}r_1^{j_1+k_1}\cdots r_n^{j_n+k_n}\overbrace{\int_{-\pi}^\pi\cdots\int_{-\pi}^\pi}^{\text{$l-$}times} \left(\prod_{l=1}^ne^{i(j_l-k_l)\vartheta_l}d\vartheta_l\right)\\
     =&\sum_{|j|=0}\sum_{|k|=0}a_j\overline{a_k}r_1^{j_1+k_1}\cdots r_n^{j_n+k_n}[(2\pi)^n\delta_{j_1k_1}\cdots\delta_{j_nk_n}]\\
    =&(2\pi)^n\sum_{|j|=0}|a_j|^2r^{2j}.
 \end{align*}
 Therefore, multiplying by $(2\pi)^{-n}$ in the last equality furnishes the desired proof.
 \end{proof}
 \begin{proposition}[Jensen's convex inequality]\label{prpstn_Jensen's Inequality}
Let $(\Omega, \Sigma, \mu)$ be a probability space, and $g$ a real-valued function that is $\mu$-integrable. If $\psi$ is a convex function then, 
\[\psi\left(\int_{\Omega} g\, d\mu\right)\leq \int_{\Omega} \psi\circ g \, d\mu.\]  
\end{proposition}
\begin{proof}
See \cite[Lemma 6.1, Page 33]{garnett2007bounded}.
\end{proof}
\begin{lemma}\label{lemma_5.6}
    Let $\Xi:\Cn\to\Cn$ be a holomorphic mapping with $\Xi\equiv\left(\xi_1,\ldots,\xi_n\right)\in\Cn$, where each $\left\{\xi\right\}_{i=1,\ldots,n}$ are the coordinate functions of $\Xi$ from $\Cn\to\mathbb{C}$ which are holomorphic. As $\|\Xi\left(\bmz\right)\|_2$ be the Euclidean-norm in $\mathbb{C}^n$ for some $\bmz\in\Cn$, then following inequality is satisfied for any $\alpha\geq1$,
    \begin{align*}
        \left(\int_{r\mathbb{B}_n}\|\Xi\left(\bmz\right)\|_2^2dV(\bmz)\right)^\alpha\leq\int_{r\mathbb{B}_n}\|\Xi\left(\bmz\right)\|_2^{2\alpha}dV(\bmz),
    \end{align*}
    where $r\mathbb{B}_n=\left\{\bmz\in\Cn:\|\bmz\|_2\leq r\right\}$ for some $r>0$. 
\end{lemma}
\begin{proof}
    In order to prove the mentioned inequality in the above lemma, we will use the Jensen's convex inequality \autoref{prpstn_Jensen's Inequality}. Note that if we have a uni-variate function $\psi_\alpha:\mathbb{R}_+\to\mathbb{R}_{+}$ defined by $\psi_\alpha(x)=x^\alpha$ for some $\alpha\geq1$, then $\psi_\alpha$ is a convex function (follow \cite{boyd2004convex}). Now, as we see that $\|\Xi\left(\bmz\right)\|_2^2=\sum_{i=1}^n|\xi_i(\bmz)|^2$ is a real-valued function which is integrable with respect to $dV(\bmz)$ over the set $r\mathbb{B}_n$. Thus, we have
    \begin{align}\label{eq_39}
        \psi_\alpha\left(\int_{r\mathbb{B}_n}\|\Xi\left(\bmz\right)\|_2^2dV(\bmz)\right)=&\left(\int_{r\mathbb{B}_n}\|\Xi\left(\bmz\right)\|_2^2dV(\bmz)\right)^\alpha
    \end{align}
    by the definition of $\psi_\alpha$. On the other hand,
    \begin{align}\label{eq_40}
        \int_{r\mathbb{B}_n}\psi_\alpha\circ\|\Xi(\bmz)\|_2^2dV(\bmz)=\int_{r\mathbb{B}_n}\left(\|\Xi(\bmz)\|_2^2\right)^{\alpha}dV(\bmz)=\int_{r\mathbb{B}_n}\|\Xi(\bmz)\|_2^{2\alpha}dV(\bmz).
    \end{align}
  Combining the result of \eqref{eq_39}, \eqref{eq_40} together with the result of \autoref{prpstn_Jensen's Inequality}, we achieve the desired result. 
\end{proof}Before we provide a key proposition for the action of the Koopman operators over the RKHS $H_{\sigma,1,\Cn}$, we shall recall an interesting application of \emph{Cauchy's inequalities} or the maximum modulus principle, which usually is a standard complex analysis fact.
\begin{proposition}\label{prpstn_cauchyestimate}
    If $F$ is an entire function that satisfy $\sup_{|z|=R}|F(z)|\leq {A}R^k+B$ for all $R>0$ and for some integer $k\geq0$ and some constants ${A},B>0$, then $F$ is a polynomial of degree $\leq k$.
\end{proposition}
\begin{proof}
    Follow \cite[Chapter-3, Exercise-15(a), Page 105-106]{stein2010complex}.
\end{proof}
    We give a key proposition before we give the main theorem related to the boundedness of the Koopman operators over the RKHS $H_{\sigma,1,\mathbb{C}^n}$.
\begin{prpstn}\label{key proposition}
        Let there be a positive finite $M$ such that for a holomorphic function $\varphi:\mathbb{C}^n\to\mathbb{C}^n$, in which every coordinate function of $\varphi$ is holomorphic from $\Cn\to\mathbb{C}$. Let $\bm{z}\in\mathbb{C}^n$, and suppose following holds:
        \begin{align*}
            \left[\frac{\sinh{\left(\nicefrac{\|\varphi(\bm{z})\|_2}{\sigma}\right)}}{\nicefrac{\|\varphi(\bm{z})\|_2}{\sigma}}\right]^{\frac{1}{2}}\cdot\left[\frac{\sinh{\left(\nicefrac{\|\bm{z}\|_2}{\sigma}\right)}}{\nicefrac{\|\bm{z}\|_2}{\sigma}}\right]^{-\frac{1}{2}}<M,
        \end{align*}where $\sigma>0$. 
         If $0<\|\varphi\|_2\leq\pi\sigma$, then  $\varphi$ admits an affine structure on $\Cn$, that is, $\varphi(\bm{z})=\mathcal{A}\bm{z}+B$ where $\mathcal{A}\in\mathbb{C}^{n\times n}$ with $0<\|\mathcal{A}\|_2\leq1$ and $B$ is a complex vector in $\Cn$.
    \end{prpstn}
    \begin{proof}
        The validity of given inequality still holds even if we square it, therefore 
        \begin{align*}
            \left[\frac{\sinh{\left(\nicefrac{\|\varphi(\bm{z})\|_2}{\sigma}\right)}}{\nicefrac{\|\varphi(\bm{z})\|_2}{\sigma}}\right]\cdot\left[\frac{\sinh{\left(\nicefrac{\|\bm{z}\|_2}{\sigma}\right)}}{\nicefrac{\|\bm{z}\|_2}{\sigma}}\right]^{-1}<M^2.
        \end{align*}
        Now, if we take the logarithm of above, we have following
        \begin{align*}
            \log\left[\frac{\sinh{\left(\nicefrac{\|\varphi(\bm{z})\|_2}{\sigma}\right)}}{\nicefrac{\|\varphi(\bm{z})\|_2}{\sigma}}\right]<\log(M^2)+\log\left[\frac{\sinh{\left(\nicefrac{\|\bm{z}\|_2}{\sigma}\right)}}{\nicefrac{\|\bm{z}\|_2}{\sigma}}\right].
        \end{align*}
        We can now use the results of \autoref{lemma_kernel_lowerandupperbound} in the above inequality to result into following observations:
        \begin{align*}
            \frac{1}{2}\left[\frac{\|\varphi(\bm{z})\|_2}{\sigma}\coth\left({\frac{\|\varphi(\bm{z})\|_2}{\sigma}}\right)-1\right]<&\log(M^2)+\frac{\|\bm{z}\|_2^2}{6\sigma^2}\\
            \frac{\|\varphi(\bm{z})\|_2}{\sigma}\coth\left({\frac{\|\varphi(\bm{z})\|_2}{\sigma}}\right)<&2\log(M^2)+1+\frac{\|\bm{z}\|_2^2}{3\sigma^2}.\numberthis\label{eq_47r}
        \end{align*}
        Now, to further simplify the above inequality, we will simply employ the infinite series expansion of an entire $\coth(\bullet)$ which involves the presence of Bernoulli's number $\left\{B_j\right\}_{j\in\mathbb{W}}$; defined as $\coth{x}=\sum_{j=0}^\infty\frac{2^{2n}B_{2n}}{(2n)!}x^{2n-1}.$

The above equation can be explicitly written as $x\coth{x}=1+\frac{2^2B_2}{2!}x^2+\sum_{j=2}^\infty\frac{2^{2j}B_{2j}}{(2j)!}x^{2j}$ under the additional assumption of $0<|x|<\pi$. 
Here, for $j=1$, we have $B_2=\nicefrac{1}{6}$, so $x\mapsto\nicefrac{\|\varphi(\bm{\bm{z}})\|_2}{\sigma}$ in above yields:
\begin{align*}\numberthis\label{eq_50r}
    \frac{\|\varphi(\bm{z})\|_2}{\sigma}\coth\left({\frac{\|\varphi(\bm{z})\|_2}{\sigma}}\right)=1+\frac{1}{3}\left(\frac{\|\varphi(\bm{\bm{z}})\|_2}{\sigma}\right)^2+\sum_{j=2}^\infty\frac{2^{2n}B_{2n}}{(2n)!}\left(\frac{\|\varphi(\bm{\bm{z}})\|_2}{\sigma}\right)^{2j}.
\end{align*}
Using the result of \eqref{eq_50r} in \eqref{eq_47r} to have following:
\begin{align*}
    1+\frac{1}{3}\left(\frac{\|\varphi(\bm{\bm{z}})\|_2}{\sigma}\right)^2+\sum_{j=2}^\infty\frac{2^{2j}B_{2j}}{(2j)!}\left(\frac{\|\varphi(\bm{\bm{z}})\|_2}{\sigma}\right)^{2j}<&2\log(M^2)+1+\frac{\|\bm{z}\|_2^2}{3\sigma^2}\\
    \|\varphi(\bm{z})\|_2^2+(3\sigma^2)\sum_{j=2}^\infty\frac{2^{2j}B_{2j}}{(2j)!\sigma^{2j}}\|\varphi(\bm{\bm{z}})\|_2^{2j}<&2\log(M^2)+\|\bm{z}\|_2^2.\numberthis\label{eq_45r}
\end{align*}
Considering $\varphi\equiv\left(\varphi_1(\bmz),\ldots,\varphi_n(\bmz)\right)^\top\in\Cn$, where each $\left\{\varphi_i\right\}_{i=1,\ldots,n}$ is a coordinate function of $\varphi$ and is a holomorphic mapping from $\Cn\to\mathbb{C}$, then $\|\varphi(\bmz)\|_2^2=\sum_{i=1}^n|\varphi_i(\bmz)|^2$. Therefore, with $\bmz=\left(z_1,\ldots,z_n\right)^\top\in\Cn$ and $\|\bmz\|_2^2=\sum_{i=1}^n|z_i|^2$, we see that
\begin{align*}
    \left\{\sum_{i=1}^n\left[|\varphi_i(\bmz)|^2-|z_i|^2\right]\right\}+\left\{(3\sigma^2)\sum_{j=2}^\infty\frac{2^{2j}B_{2j}}{(2j)!\sigma^{2j}}\|\varphi(\bm{\bm{z}})\|_2^{2j}\right\}
    <2\log(M^2).
\end{align*}
Integrating above with respect to $\bm{\vartheta}$ on $r\mathbb{B}_n$ to have 
\begin{align*}\label{eq_45}
    &\left\{\int_{r\mathbb{B}_n}\sum_{i=1}^n\left[|\varphi_i(\bmz)|^2-|z_i|^2\right]\frac{d\bm{\vartheta}}{(2\pi)^n}\right\}\\+&\left\{(3\sigma^2)\sum_{j=2}^\infty\frac{2^{2j}B_{2j}}{(2j)!\sigma^{2j}}\int_{r\mathbb{B}_n}\|\varphi(\bm{\bm{z}})\|_2^{2j}\frac{d\bm{\vartheta}}{(2\pi)^n}\right\}<2\log(M^2)\numberthis.
\end{align*}
As here, the infinite summation starts with $j\geq2>1$, therefore by the application of \autoref{lemma_5.6}, we can have following conclusion
\begin{align*}
&\left\{\int_{r\mathbb{B}_n}\sum_{i=1}^n\left[|\varphi_i(\bmz)|^2-|z_i|^2\right]\frac{d\bm{\vartheta}}{(2\pi)^n}\right\}_{\scriptscriptstyle{(\spadesuit)}}\\+&\left\{(3\sigma^2)\sum_{j=2}^\infty\frac{2^{2j}B_{2j}}{(2j)!\sigma^{2j}{(2\pi)^{nj}}}\left(\int_{r\mathbb{B}_n}\|\varphi(\bm{\bm{z}})\|_2^{2}d\bm{\vartheta}\right)^j\right\}_{\scriptscriptstyle{(\clubsuit)}}\leq\text{LHS of \eqref{eq_45}}\\<&2\log(M^2).
\end{align*}
Here we are using $\left\{\bullet\right\}_{\scriptscriptstyle{(\spadesuit)}}$ and $\left\{\bullet\right\}_{\scriptscriptstyle{(\clubsuit)}}$ to provide smooth understanding of respective manipulations that are happening on respective quantities present inside the curly brackets. Therefore, the above inequality results into
\begin{align*}
    &\left\{\int_{r\mathbb{B}_n}\sum_{i=1}^n\left[|\varphi_i(\bmz)|^2-|z_i|^2\right]\frac{d\bm{\vartheta}}{(2\pi)^n}\right\}_{\scriptscriptstyle{(\spadesuit)}}\\+&\left\{(3\sigma^2)\sum_{j=2}^\infty\frac{2^{2j}B_{2j}}{(2j)!\sigma^{2j}{(2\pi)^{nj}}}\left(\int_{r\mathbb{B}_n}\|\varphi(\bm{\bm{z}})\|_2^{2}d\bm{\vartheta}\right)^j\right\}_{\scriptscriptstyle{(\clubsuit)}}<2\log(M^2).
\end{align*}
The above inequality can be simplified into following (term by term) in terms of $r$ along with the help of multi-index notation:
\begin{align*}
&\left\{\sum_{|j|=0}|a^k_j|^2r_1^{2j_1} \cdot \ldots \cdot r_n^{2j_n}-\sum_{\ell = 1}^n r_\ell^2\right\}_{\scriptscriptstyle{(\spadesuit)}}\\+&\left\{(3\sigma^2)\sum_{j=2}^\infty\frac{2^{2j}B_{2j}}{(2j)!\sigma^{2j}{(2\pi)^{nj}}}\left(\sum_{|j|=0}|a^k_j|^2r_1^{2j_1} \cdot \ldots \cdot r_n^{2j_n}\right)^j\right\}_{\scriptscriptstyle{(\clubsuit)}}<2\log(M^2).
\end{align*}
Note that, here we used a super-script of $k$ to indicate this a decomposition of the $k$-th component of the function $\varphi$. Further, if we let $e_i$ be the multi-index with a $1$ in the $i$-th spot and \emph{zeros} else where, then the above rearranges to:
\begin{align*}
    &\left\{\overbrace{\sum_{|j|=2}|a^k_j|^2r_1^{2j_1} \cdot \ldots \cdot r_n^{2j_n}}^{\text{$Q_1(r)$}} + |a^k_0|^2 + \sum_{\ell=0}^n (|a^k_{e_\ell}|^2-1)r^2_{\ell}\right\}_{\scriptscriptstyle{(\spadesuit)}}\\+&\left\{(3\sigma^2)\sum_{j=2}^\infty\frac{2^{2j}B_{2j}}{(2j)!\sigma^{2j}{(2\pi)^{nj}}}\left(\underbrace{\sum_{|j|=0}|a^k_j|^2r_1^{2j_1} \cdot \ldots \cdot r_n^{2j_n}}_{\text{$Q_2(r)$}}\right)^j\right\}_{\scriptscriptstyle{(\clubsuit)}}<2\log(M^2).
\end{align*}
This inequality is true for all $r=(r_1,\ldots, r_n)^\top\in \mathbb{R}_+^n$. We see that both quantities {\text{$Q_1(r)$}} and {\text{$Q_2(r)$}} grow in the \emph{Big-O} complexity rate, that is 
    ${\text{$Q_1(r)$}}\propto O(r)~\&~{\text{$Q_2(r)$}}\propto O(r^j)$ as $r\to\infty.$ Therefore, by the application of Cauchy's estimate (or the maximum modulus principle as in \autoref{prpstn_cauchyestimate}) in this inequality, we immediately conclude that $|a_j^k|=0$ for $j\geq2$ in ${\text{$Q_1(r)$}}$. Similarly, we also see that Cauchy's estimate (or the maximum modulus principle as in \autoref{prpstn_cauchyestimate}) forces to ${\text{$Q_2(r)$}}$ be $0$ as well. Hence, with the relabelling of the coordinate function of $\varphi$ as $\varphi_k(z)=a_{k,1}z_1+\cdots+ a_{k,n}z_n + b_k$. Thus, $\varphi(\bmz) = \mathcal{A}\bmz + B$ where, $\mathcal{A} = [a_{k,j}]_{k,j=1}^{n,n}$ and $B = (b_1, \ldots, b_n)^\top.$

    Now, that we have observed that both ${\text{$Q_1(r)$}}$ and ${\text{$Q_2(r)$}}$ are $0$ and $\varphi(\bmz) = \mathcal{A}\bmz + B$, we revisit \eqref{eq_45r} to have 
    \begin{align*}
    \|\mathcal{A}\bmz + B\|_2^2<&2\log(M^2)+\|\bmz\|_2^2\\
        \frac{\|\mathcal{A}\bmz + B\|_2^2}{\|\bmz\|_2^2}<&\frac{2\log(M^2)}{\|\bmz\|_2^2}+1\\
        \implies\lim_{\|\bmz\|\to\infty}\frac{\|\mathcal{A}\bmz + B\|_2^2}{\|\bmz\|_2^2}<&1.
    \end{align*}
    Now, suppose that $\|\mathcal{A}\zeta\|_2>\|\zeta\|_2=1$ for some $\zeta\in\Cn$ whose norm is $1$. Setting $\bmz=t\zeta$ and $t>0$ in above yields following:
\begin{align*}
    \lim_{t\to\infty}\frac{\left\|\mathcal{A}\zeta+\dfrac{1}{t}B\right\|_2}{\|\zeta\|_2}<1,
\end{align*}
which is a contradiction and therefore, $\|\mathcal{A}\|_2\leq1$.
    \end{proof} 
    \subsubsection{Boundedness of Koopman operators}
    \begin{dfn}Let $\sigma$ be a positive and finite real number. Let $\mathcal{K}_\varphi$ be the Koopman operator induced by holomorphic function $\varphi:\Cn\to\Cn$ acting over the RKHS $H_{\sigma,1,\Cn}$ whose reproducing kernel is given as $K_{\bmz}^{\sigma}$ at $\bmz\in\Cn$. We define
        \begin{align}\label{eq_48}
            \Pi_{\bmz}(\varphi;\sigma)
            \coloneqq\frac{\|\mathcal{K}_\varphi^*K_{\bmz}^\sigma\|^2}{\|K_{\bmz}^\sigma\|^2}.
        \end{align}
        Additionally, we also define the supremum of above over $\bmz\in\Cn$ as follows:
        \begin{align}\label{eq_49r}
            \Pi\left(\varphi;\sigma\right)
            \coloneqq\sup_{\bmz\in\Cn}\Pi_{\bmz}(\varphi;\sigma)=\sup_{\bmz\in\Cn}\frac{\|\mathcal{K}_\varphi^*K_{\bmz}^\sigma\|^2}{\|K_{\bmz}^\sigma\|^2}.
        \end{align}
    \end{dfn}
    Now, that we have defined two important quantities given in \eqref{eq_48} and \eqref{eq_49r} which essentially help characterizing the behaviour of the Koopman operators over the RKHS $H_{\sigma,1,\Cn}$ in terms of its reproducing kernel. We can now give an important inequality for the action of the Koopman operators on the normalized reproducing kernel $\bm{k}_{\bmz}^\sigma\coloneqq\nicefrac{K_{\bmz}^\sigma}{\|K_{\bmz}^\sigma\|}$ satisfying $\|\bm{k}_{\bmz}^\sigma\|=1$ at $\bmz\in\Cn$ where $\sigma>0$. This result is captured in the following lemma.
    \begin{lemma}\label{lemma_5.10}
        Let $\sigma>0$. Let $\varphi:\Cn\to\Cn$ be a holomorphic function over $\Cn$ in which every coordinate functions of $\varphi$ are holomorphic from $\Cn\to\mathbb{C}$. Consider the Koopman operator $\koop:\mathcal{D}\left(\koop\right)\to H_{\sigma,1,\Cn}$ induced by $\varphi$. If for some $\bmz\in\Cn$, $\bm{k}_{\bmz}^\sigma\in\mathcal{D}\left(\koop\right)$, then:
        \begin{align*}
            \sqrt{\Pi_{\bmz}\left(\varphi;\sigma\right)}\leq\|\koop \bm{k}_{\bmz}^\sigma\|.
        \end{align*}
    \end{lemma}
    \begin{proof}
        The proof of the above result involves the application of point-evaluation inequality for the RKHS $H_{\sigma,1,\Cn}$. Further details are given as follows:
        \begin{align*}
            \|\koop\bm{k}_{\bmz}^\sigma\|^2\|K_{\bmz}^\sigma\|^2\geq|\koop\bm{k}_{\varphi(\bmz)}^\sigma\left(\bmz\right)|^2=|\bm{k}_{\varphi\left(\bmz\right)}^{\sigma}\left(\varphi(\bmz)\right)|^2.
        \end{align*}
        As $|\bm{k}_{\varphi(\bmz)}^\sigma|=|\nicefrac{K_{\varphi(\bmz)}^\sigma}{\|K_{\varphi(\bmz)}^\sigma\|}|$, hence $|\bm{k}_{\varphi(\bmz)}^\sigma\left(\varphi(\bmz)\right)|=|\nicefrac{K_{\varphi(\bmz)}^\sigma\left(\varphi(\bmz)\right)}{\|K_{\varphi(\bmz)}^\sigma\|}|=|\nicefrac{\|K_{\varphi(\bmz)}^\sigma\|^2}{\|K_{\varphi(\bmz)}^\sigma\|}|=\|K_{\varphi(\bmz)}^\sigma\|.$ From this result and above, we have
        \begin{align*}
           \|\koop\bm{k}_{\bmz}^\sigma\|^2\|K_{\bmz}^\sigma\|^2\geq \|K_{\varphi(\bmz)}^\sigma\|^2=\|\koop^*K_{\bmz}^\sigma\|^2.
        \end{align*}
        Therefore, further dividing the above inequality by $\|K_{\bmz}^\sigma\|^2\neq0$ to have
        \begin{align*}
            \|\koop\bm{k}_{\bmz}^\sigma\|^2\geq\frac{\|\koop^*K_{\bmz}^\sigma\|^2}{\|K_{\bmz}^\sigma\|^2}=\Pi_{\bmz}\left(\varphi;\sigma\right).
        \end{align*}
        The desired result follows by taking the square-root of above. Hence proved.
    \end{proof}
    The following theorem provides the boundedness characterization for the Koopman operators $\koop$ induced by the holomorphic function $\varphi$.
    \begin{theorem}\label{theorem_boundedKoopmanoverRKHS}
        Let $\sigma>0$ and $\varphi:\Cn\to\Cn$ be a holomorphic function in which every coordinate function of $\varphi$ is holomorphic from $\Cn\to\mathbb{C}$. Let $\mathcal{K}_\varphi:\mathcal{D}\left(\mathcal{K}_\varphi\right)\to H_{\sigma,1,\Cn}$ be the Koopman operator induced by $\varphi$ over the RKHS $H_{\sigma,1,\Cn}$. Then, the Koopman operator $\mathcal{K}_\varphi$ acts boundedly over $H_{\sigma,1,\Cn}$ if and only $\varphi$ admits the affine structure, that is $\varphi(\bmz)=\mathcal{A}\bmz+B$ where $\|\mathcal{A}\|_2\leq1$ and $\Pi\left(\varphi;\sigma\right)<\infty$.
    \end{theorem}
    \begin{proof} Let $\sigma>0$ and $\varphi$ be a holomorphic function on $\Cn$ as given in the statement. Consider the $\varphi-$induced Koopman operator acting over the RKHS $H_{\sigma,1,\Cn}$ as $\koop:\mathcal{D}\left(\koop\right)\to H_{\sigma,1,\Cn}$. \begin{enumerate}
        \item[$\implies$]Suppose that the $\varphi$-induced Koopman operator is bounded over RKHS $H_{\sigma,1,\Cn}$, which means that there exists a finite positive $M$ such that $\|\koop\|^2<M$. A general operator theory argument allow us to have $\|\koop^*\|^2=\|\koop\|^2$ (cf. \cite{hall2013quantum}) and hence $\|\koop^*\|^2<M<\infty$. Now, observe that
        \begin{align*}
            \infty>M>\|\koop^*\|^2=\sup_{\bmz\in\Cn}\frac{\|\koop^*K_{\bmz}^\sigma\|^2}{\|K_{\bmz}^\sigma\|^2}\geq\frac{\|\koop^*K_{\bmz}^\sigma\|^2}{\|K_{\bmz}^\sigma\|^2}.
        \end{align*}
        The above inequality allow us to have $\frac{\|\koop^*K_{\bmz}^\sigma\|^2}{\|K_{\bmz}^\sigma\|^2}<M$. Thus, employing the result of (our key proposition) \autoref{key proposition}, we have the affine structure of $\varphi$, which is $\varphi(\bmz)=\mathcal{A}\bmz+B$ along with $\|\mathcal{A}\|_2\leq1$.
        \item[$\impliedby$]Now, suppose that we have the affine structure of $\varphi(\bmz)=\mathcal{A}z+B$, where $\mathcal{A}\in\Cn\times\Cn$ with $\|\mathcal{A}\|_2\leq1$. Additionally, suppose that for this $\varphi$, $\Pi\left(\varphi;\sigma\right)<\infty$ also holds. 

        Recall the \emph{normalized reproducing kernel} $\bm{k}_{\bmz}^\sigma$ at some $\bmz\in\Cn$ which is given as $\bm{k}_{\bmz}^\sigma=\nicefrac{K_{\bmz}^\sigma}{\|K_{\bmz}^\sigma\|}
        $. Then,
        \begin{align*}
            \|\mathcal{K}_\varphi^*\|^2=&\sup_{\bmz\in\Cn}\left\{\sup_{\|\bm{k}_{\bmz}^\sigma\|=1}\frac{\|\mathcal{K}_\varphi^*\bm{k}_{\bmz}^\sigma\|^2}{\|\bm{k}_{\bmz}\|^2}\right\}\\
            =&\sup_{\bmz\in\Cn}\left\{\sup_{\|\bm{k}_{\bmz}^\sigma\|=1}\frac{\left\|\mathcal{K}_\varphi^*\dfrac{K_{\bmz}^\sigma}{\|K_{\bmz}^\sigma\|}\right\|^2}{\left\|\dfrac{K_{\bmz}^\sigma}{\left\|K_{\bmz}^\sigma\right\|}\right\|^2}\right\}\\
            =&\sup_{\bmz\in\Cn}\left\{\sup_{\|\bm{k}_{\bmz}^\sigma\|=1}\left[\frac{\|K_{\bmz}^\sigma\|^{-2}}{\|K_{\bmz}^\sigma\|^{-2}}\cdot\frac{\|\mathcal{K}_\varphi^*K_{\bmz}^\sigma\|^2}{\|K_{\bmz}^\sigma\|^2}\right]\right\}\\
            =&\sup_{\bmz\in\Cn}\frac{\|\mathcal{K}^*_\varphi K_{\bmz}^\sigma\|^2}{\|K_{\bmz}^\sigma\|^2}\\
            =&\Pi\left(\varphi;\sigma\right)\\
            <&\infty.
        \end{align*}
        The above chain of inequalities implies that $\|\mathcal{K}_\varphi^*\|$ is bounded.
    \end{enumerate} 
Hence, combining the arguments of the above two points delivers the desired result.
    \end{proof}
    After that we have derived the boundedness characterization of the Koopman operator over the RKHS $H_{\sigma,1,\Cn}$, we would like to close this subsection by providing an important yet an easy-exercise corollary.
    \begin{corollary}
        Let $\sigma>0$. Let $\varphi:\Cn\to\Cn$ be a holomorphic function in which each coordinate function of $\varphi$ is holomorphic from $\Cn\to\mathbb{C}$. Let $\koop:\mathcal{D}\left(\koop\right)\to H_{\sigma,1,\Cn}$ be the Koopman operator induced by $\varphi$ over the RKHS $H_{\sigma,1,\Cn}$. If the Koopman operator $\koop$ acts boundedly over the RKHS $H_{\sigma,1,\Cn}$ then $\varphi(\bmz)=\mathcal{A}\bmz+B$ with $0<\|\mathcal{A}\|_2\leq1$ and
        \begin{align*}
            \sup_{\bmz\in\Cn}\left\{\exp\left[\frac{1}{2}\left(\frac{\|\mathcal{A}\bmz+B\|_2}{\sigma}\coth{\frac{\|\mathcal{A}\bmz+B\|_2}{\sigma}}-1\right)-\frac{\|\bmz\|_2^2}{6\sigma^2}\right]\right\}<C<\infty.
        \end{align*}
    \end{corollary}
    \begin{proof}
        The proof of this corollary involves the application of both lower and upper bounds of the reproducing kernel $K_{\bmz}^\sigma$ of RKHS $H_{\sigma,1,\Cn}$ that we derived in \autoref{lemma_kernel_lowerandupperbound} combining together with the consequence of \autoref{theorem_boundedKoopmanoverRKHS}. As the Koopman operator $\koop:\mathcal{D}\left(\koop\right)\to H_{\sigma,1,\Cn}$ is acting boundedly, therefore $\varphi(\bmz)=\mathcal{A}\bmz+B$ where $\mathcal{A}\in\mathbb{C}^{n\times n}$ and $\|\mathcal{A}\|_{2}\leq1$. Additionally, we also see that the quantity $\Pi\left(\varphi;\sigma\right)$ is finite where the structure of $\varphi$ is already determined. With this information in our hand, we proceed as follows:
        \begin{align}\label{eq_52r}
            \infty>\Pi\left(\varphi;\sigma\right)=\sup_{\bmz\in\Cn}\Pi_{\bmz}\left(\varphi;\sigma\right)\geq\Pi_{\bmz}\left(\varphi;\sigma\right)=\|\koop^*K_{\bmz}^\sigma\|^2\cdot\|K_{\bmz}^\sigma\|^{-2}=\|K_{\varphi(\bmz)}^\sigma\|^2\cdot\|K_{\bmz}^\sigma\|^{-2}.
        \end{align}
        Since by the conclusion of \autoref{lemma_kernel_lowerandupperbound}, we have following:
        \begin{align}
            \|K_{\bm{z}}^\sigma\|^2
        <_{(\text{\textsc{2}})}\exp{\left(\frac{\|\bm{z}\|_{2}^2}{6\sigma^2}\right)}\implies\|K_{\bmz}^\sigma\|^{-2}>\exp{\left(-\frac{\|\bm{z}\|_{2}^2}{6\sigma^2}\right)}.\label{eq_53r}
        \end{align}
        Also, if we revisit \autoref{lemma_kernel_lowerandupperbound}, we see that 
        \begin{align*}
            \|K_{\bm{z}}^\sigma\|^2&>_{(\text{\textsc{1}})}\exp\left(\frac{1}{2}\left[\frac{\|\bm{z}\|_{2}}{\sigma}\coth{\frac{\|\bm{z}\|_{2}}{\sigma}}-1\right]\right)\\\implies\|K_{\varphi(\bm{z})}^\sigma\|^2&>\exp\left(\frac{1}{2}\left[\frac{\|\varphi(\bm{z})\|_{2}}{\sigma}\coth{\frac{\|\varphi(\bm{z})\|_{2}}{\sigma}}-1\right]\right).\numberthis\label{eq_55r}
        \end{align*}
        Therefore, combining \eqref{eq_55r}, \eqref{eq_53r} and \eqref{eq_52r}, we have
        \begin{align*}
            \Pi_{\bmz}\left(\varphi;\sigma\right)\geq\exp\left(\frac{1}{2}\left[\frac{\|\varphi(\bm{z})\|_{2}}{\sigma}\coth{\frac{\|\varphi(\bm{z})\|_{2}}{\sigma}}-1\right]\right)\cdot\exp{\left(-\frac{\|\bm{z}\|_{2}^2}{6\sigma^2}\right)}.
        \end{align*}
        To this end, now we take the supremum over $\bmz\in\Cn$ of above to have
        \begin{align*}
            \sup_{\bmz\in\Cn}\exp\left(\frac{1}{2}\left[\frac{\|\varphi(\bm{z})\|_{2}}{\sigma}\coth{\frac{\|\varphi(\bm{z})\|_{2}}{\sigma}}-1-\frac{\|\bm{z}\|_{2}^2}{6\sigma^2}\right]\right)\leq\sup_{\bmz\in\Cn}\Pi_{\bmz}\left(\varphi;\sigma\right)=\Pi\left(\varphi;\sigma\right).
        \end{align*}
        As for affine $\varphi(\bmz)=\mathcal{A}\bmz+B$, the quantity $\Pi\left(\varphi;\sigma\right)$ is already finite due to the Koopman operator being bounded over the RKHS $H_{\sigma,1,\Cn}$, hence the result prevails.
    \end{proof}
    \subsubsection{Some words on the Affine structure of \texorpdfstring{$\varphi$}{}}
    If $\mathcal{A}$ is an invertible $n\times n$ complex-valued matrix, then $\varphi=\mathcal{A}\bmz+B$ is an injective entire self-mapping on $\mathbb{C}^n$. For the remainder of the paper, $\varphi$ will be an affine self-map on $\mathbb{C}^n$; that is to say that $\varphi(z)=\mathcal{A}z+B$ is an injective self-mapping entire function on $\mathbb{C}^n$ and $\mathcal{A}$ is invertible. This convention follows that found in \cite{hai2021weighted,carswell2003composition,zhao2015invertible}, and \cite{hai2018complex} for the \emph{affine} structure of $\varphi$.
    
    If $\varphi:\Cn\to\Cn$ admits an affine structure with an additional condition that $0\not\equiv\mathcal{A}\in\mathbb{C}^{n\times n}$ and also is invertible (that is $\det\left(\mathcal{A}\right)\neq0$), then one can define $\varpi\left(\bm{u}\right)=\mathcal{A}^{-1}\bm{u}-\mathcal{A}^{-1}B$ as the inverse map of $\varphi$. In this case, if we define 
    \begin{align}\label{eq_natural52}
        \natural(\varpi(\bm{u}))=\Pi_{\varpi\left(\bm{u}\right)}\left(\varphi;\sigma\right)\implies\natural(\varphi)=\Pi_z\left(\varphi;\sigma\right).
    \end{align}
    Therefore,
    \begin{align}\label{eq_natural53}
        \natural\left(\varphi\right)\leq\Pi\left(\varphi;\sigma\right).
    \end{align}Following are the vector field in two dimension for both an affine dynamical system and non-affine dynamical system.
\begin{figure}[H]
    \centering
    \includegraphics{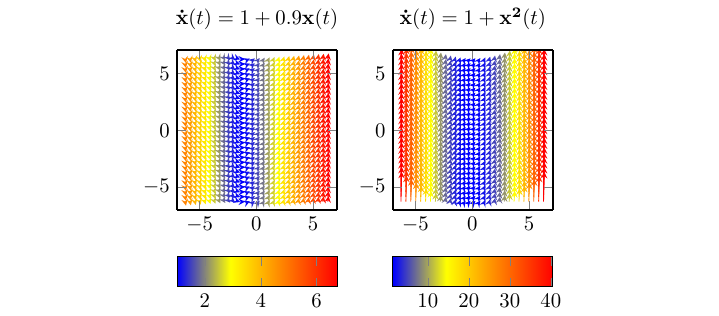}
    \caption{\emph{Vector field}: Affine dynamic (\emph{L}) \& Non-affine dynamic (\emph{R}).}
    \label{fig:my_label1+x}
\end{figure}
\subsection{Essential norm estimates of Koopman operators over RKHS \texorpdfstring{$H_{\sigma,1,\mathbb{C}^n}$}{}}Following is the basic definition for the \emph{essential norm} of a bounded linear operator acting between the Banach space.
        \begin{dfn}
For two Banach spaces $\mathbb{X}_1$ and $\mathbb{X}_2$ we denote by $K(\mathbb{X}_1, \mathbb{X}_2)$ the set of all compact operators from $\mathbb{X}_1$ into $\mathbb{X}_2$. The essential norm of a bounded linear operator $\mathrm{A}:\mathbb{X}_1 \to \mathbb{X}_2$, denoted as $\|\mathrm{A}\|_{\ess}$ is defined as
\begin{equation}\label{eq_16essnorm}
    \|\mathrm{A}\|_{\ess}\coloneqq\inf\left\{\|\mathrm{A}-T\|:T\in K(\mathbb{X}_1, \mathbb{X}_2)\right\}.
\end{equation}
\end{dfn}
    Recall that a holomorphic function $f:\Cn\to\mathbb{C}$ exhibiting $f(\bmz)=\sum_{m}a_m\bmz^m$ for $\bmz\in\Cn$ where the summation is over all multi-indexes $m=\left(m_1,\cdots,m_n\right)$ where each $\left\{m_i\right\}$ are positive integer and $\bmz=z_1^{m_1}\cdots z_n^{m_n}$. Follow \cite{zhu2005spaces} for more details on this.

    By letting $P_k(\bmz)=\sum_{|m|=k}a_m\bmz^m$ for each $k\geq0$ where $|m|=\sum_{i=1}^nm_i$, then the Taylor series of $f$ can be re-written as 
    \begin{align}\label{eq_49}
        f(\bmz)=\sum_{k=0}^\infty P_k(\bmz).
    \end{align}
    The result in \eqref{eq_49} is called as the \emph{homogeneous polynomial expansion} of holomorphic function $f$ having the degree of $k$ which is uniquely determined by $f$. Now, for each $m\in\mathbb{Z}_+$, we define the operator $\mathcal{P}_m$ on holomorphic function $f$ as follows:
    \begin{align*}
        \mathcal{P}_mf\left(\bmz\right)=\sum_{k=m}^\infty P_k\left(\bmz\right).
    \end{align*}
    If we consider the action of the operator $\mathcal{P}_m$ defined above on the reproducing kernel $K_{\bm{w}}^\sigma$ of the RKHS $H_{\sigma,1,\Cn}$ from \eqref{eq_21} in \autoref{theorem_RKviaOrthonormalbasis}, then we get the following result:
    \begin{align*}
        \mathcal{P}_mK_{\bm{w}}^\sigma\left(\bmz\right)=\mathcal{P}_m\left(\frac{\sinh{\left(\sqrt{ \frac{\langle\bm{z},\bm{w}\rangle_{\mathbb{C}^n}}{\sigma^2}}
        \right)}}{\sqrt{\frac{\langle\bm{z},\bm{w}\rangle_{\mathbb{C}^n }}{\sigma^2}}
        }\right)=\mathcal{P}_m\left(\sum_{N=0}^\infty\frac{\langle\bmz,\bm{w}\rangle_{\Cn}}{(2N+1)!}\right)=\sum_{N=m}^\infty\frac{\langle\bmz,\bm{w}\rangle_{\Cn}}{(2N+1)!}.
    \end{align*}
    \begin{proposition}
        Let $\sigma>0$ and $K_{\bmz}^\sigma$ be the reproducing kernel at $\bmz\in\Cn$ of the RKHS $H_{\sigma,1,\Cn}$. Then, 
        \begin{align}\label{eq_57}
            |\mathcal{P}_mf(\bmz)|\leq\|f\|\sqrt{\sum_{N=m}^\infty\frac{\|\bmz\|_2^{2N}}{(2N+1)!}},
        \end{align}
        for all $f\in H_{\sigma,1,\Cn}$.
    \end{proposition}
    \begin{proof}Let $\sigma>0$. We proceed by considering $f\in H_{\sigma,1,\Cn}$ and then employing the reproducing property of the reproducing kernel $K_{\bmz}^\sigma$ at $\bmz\in\Cn$, then $\mathcal{P}_mf\left(\bmz\right)=\langle \mathcal{P}_mf,K_{\bmz}^\sigma\rangle$. Then, 
    \begin{align}\label{eq_55}
        |\mathcal{P}_mf\left(\bmz\right)|^2=|\langle \mathcal{P}_mf,K_{\bmz}^\sigma\rangle|^2=|\langle f,\mathcal{P}_m^*K_{\bmz}^\sigma\rangle|^2=|\langle f,\mathcal{P}_mK_{\bmz}^\sigma\rangle|^2,
    \end{align}
    where last two step uses the property of $\mathcal{P}_m$ being self-adjoint and idempotent. Now,
    \begin{align*}
        |\langle f,\mathcal{P}_mK_{\bmz}^\sigma\rangle|^2\leq\|f\|^2\|\mathcal{P}_mK_{\bmz}^\sigma\|^2=\|f\|^2\langle\mathcal{P}_mK_{\bmz}^\sigma,\mathcal{P}_mK_{\bmz}^\sigma\rangle=&\langle\mathcal{P}_m^*\mathcal{P}_mK_{\bmz}^\sigma,K_{\bmz}^\sigma\rangle\\
        =&\langle\mathcal{P}_mK_{\bmz}^\sigma,K_{\bmz}^\sigma\rangle\\
        =&\mathcal{P}_mK_{\bmz}^\sigma\left(\bmz\right)\\
        =&\sum_{N=m}^\infty\frac{\|\bmz\|_2^{2N}}{(2N+1)!}.\numberthis\label{eq_56}
    \end{align*}
    Thus, combining \eqref{eq_55} and \eqref{eq_56} followed by taking the square-root and hence the result is proved.
    \end{proof}
    Following details in regards of Hilbert spaces in the light of RKHS $H_{\sigma,1,\Cn}$, are respectfully borrowed from \cite{reed2012methods} or \cite[Chapter VI]{reedmethods}.
    \begin{proposition}
        A linear and bounded operator $\mathfrak{B}$ is compact over the RKHS $H_{\sigma,1,\Cn}$ if and only if $\lim_{M\to\infty}\|\mathfrak{B}h_M-\mathfrak{B}h\|=0$ provided that $h_M\to h$ weakly in RKHS $H_{\sigma,1,\Cn}$.
    \end{proposition}
    We use the following criteria for the \emph{weakly convergence sequence} in the RKHS $H_{\sigma,1,\Cn}$ and as a general approach can be learnt from standard references such as \cite{hai2021weighted}, \cite{le2014normal} and \cite{le2017composition}.
    \begin{proposition}\label{proposition_5.14}
        The sequence $\left\{h_M\right\}_{M}$ in the RKHS $H_{\sigma,1,\Cn}$ posses the weakly convergence to $0$ in $H_{\sigma,1,\Cn}$ if and only if following conditions are true:
        \begin{enumerate}
            \item bounded in the norm topology of the RKHS $H_{\sigma,1,\Cn}$
            \item uniformly convergent to $0$ over the compact subsets of the RKHS $H_{\sigma,1,\Cn}$.
        \end{enumerate}
    \end{proposition}
    \autoref{proposition_5.14} can be used to express the following corollary.
    \begin{corollary}
        Let $\sigma>0$ and $\varphi(\bmz)=\mathcal{A}\bmz+B$ where $\mathcal{A}\not\equiv0\in\mathbb{C}^{n\times n}$ and $\|\mathcal{A}\|_2\leq1$. Consider a sequence of points $\left\{\bmz_M\right\}_M\in\Cn$ such that $\|\bmz_M\|_2\to\infty$ as $M\to\infty$. Then, the sequence of normalized reproducing kernels $\left\{\bm{k}_{\varphi(\bmz_M)}^\sigma\right\}_M$ of $H_{\sigma,1,\Cn}$ evaluated at $\left\{\varphi\left(\bmz_M\right)\right\}_M$ converges weakly to $0$ over the RKHS $H_{\sigma,1,\Cn}$.
    \end{corollary}
    \begin{lemma}\label{lemma_5.16}
        Let $\sigma>0$. Consider $\varphi:\Cn\to\Cn$ as a holomorphic mapping in which every coordinate function of $\varphi$ are holomorphic from $\Cn\to\mathbb{C}$. Let $\koop:\mathcal{D}\left(\koop\right)\to H_{\sigma,1,\Cn}$ be the Koopman operator induced by $\varphi$ on the RKHS $H_{\sigma,1,\Cn}$. If the Koopman operator $\koop$ is bounded over the RKHS $H_{\sigma,1,\Cn}$ then the essential norm of $\koop$ denoted by $\|\koop\|_{\operatorname{ess}}$ satisfies following
        \begin{align*}
            \|\koop\|_{\operatorname{ess}}\leq\liminf_{M\to\infty}\|\koop\mathcal{P}_M\|,
        \end{align*}
        where $\varphi(\bmz)=\mathcal{A}\bmz+B$ with $\|\mathcal{A}\|_2\leq1$
    \end{lemma}
    \begin{proof}
    As the Koopman operator $\koop:\mathcal{D}\left(\koop\right)\to H_{\sigma,1,\Cn}$ induced by $\varphi$ is bounded, thus the result of \autoref{key proposition} holds and therefore $\varphi(\bmz)=\mathcal{A}\bmz+B$ with $\|\mathcal{A}\|_2\leq1$. 
        Let $\mathfrak{C}$ be a compact operator over the RKHS $H_{\sigma,1,\Cn}$. Then, observe following chain of inequalities for some $M\in\mathbb{Z}_+$:
        \begin{align*}
            \|\koop-\mathfrak{C}\|=&\|\koop\left(\mathcal{P}_M+P_M\right)-\mathfrak{C}\|\\
            \leq&\|\koop\mathcal{P}_M\|+\|\koop P_M-\mathfrak{C}\|.\numberthis\label{eq_53}
        \end{align*}
        Since $P_M$ is finite rank and hence compact. Therefore $\|\koop P_M-\mathfrak{C}\|=0$ in the light of fact that $\mathfrak{C}$ is also compact over the RKHS $H_{\sigma,1\Cn}$. Therefore, taking the $\liminf$ as $M\to\infty$, we have
        \begin{align*}
            \|\koop\|_{\operatorname{ess}}\overset{{\textsc{ via \eqref{eq_16essnorm}}}}{\coloneqq}\liminf_{M\to\infty}\|\koop-\mathfrak{C}\|\overset{{\textsc{via \eqref{eq_53}}}}{\leq}\liminf_{M\to\infty}\|\koop\mathcal{P}_M\|.
        \end{align*}
        Hence proved!
    \end{proof}
    Now, the following theorem provides the essential norm estimates for the bounded Koopman operators over the RKHS $H_{\sigma,1,\Cn}$.
    \begin{theorem}
        Let $\sigma>0$. Let $\varphi:\Cn\to\Cn$ be a holomorphic function in which each coordinate functions of $\varphi$ are holomorphic from $\Cn\to\mathbb{C}$. If the induced Koopman operator by $\varphi$, $\koop:\mathcal{D}\left(\koop\right)\to H_{\sigma,1,\Cn}$ is bounded, then the essential norm of $\koop$ satisfies following estimates inequality:
        \begin{align*}
            \lim_{\|\bmz\|_2\to\infty}\sqrt{\Pi_{\bmz}\left(\varphi;\sigma\right)}\leq_{(1)}\|\koop\|_{\operatorname{ess}}\leq_{(2)}\left(\left|\det\left(\mathcal{A}^{-1}\right)\right|\right)^{\nicefrac{n}{2}}\lim_{\|\bmz\|_2\to\infty}\sqrt{\Pi_{\bmz}\left(\varphi;\sigma\right)}
        \end{align*} where $\varphi\left(\bmz\right)=\mathcal{A}\bmz+B$ with invertible $\mathcal{A}$,  $0\not\equiv\mathcal{A}\in\mathbb{C}^{n\times n}$ and $\|\mathcal{A}\|_2<1$.
    \end{theorem}
    \begin{proof}Given that $\koop$ acts boundedly over the RKHS $H_{\sigma,1,\Cn}$, thus $\varphi$ is of the affine structure, that is $\varphi(\bmz)=\mathcal{A}\bmz+B$, where $\|\mathcal{A}\|_2\leq1$. Now, with this $\varphi$,
    we begin now the proof for the \textsc{Inequality} (1) as follows.
        Let $\mathfrak{C}$ be a compact operator over the RKHS $H_{\sigma,1,\Cn}$, then the following chain of inequalities holds true:
        \begin{align*}
            \|\koop-\mathfrak{C}\|\geq&\limsup_{M\to\infty}\|\left(\koop-\mathfrak{C}\right)\bm{k}_{\varphi(\bmz_M)}^\sigma\|\\
            \geq&\limsup_{M\to\infty}\left[\|\koop\bm{k}_{\varphi(\bmz_M)}^\sigma\|-\|\mathfrak{C}\bm{k}_{\varphi(\bmz_M)}^\sigma\|\right].\numberthis\label{eq_59}
        \end{align*} 
As the sequence of normalized reproducing kernels $\left\{\bm{k}_{\varphi(\bmz_M)}^\sigma\right\}_M$ converges weakly to $0$ over the RKHS $H_{\sigma,1,\Cn}$, therefore $\|\mathfrak{C}\bm{k}_{\varphi(\bmz_M)}^\sigma\|\to0$ as $M\to\infty$. Thus, 
        \begin{align*}
\overbrace{\limsup_{M\to\infty}\sqrt{\Pi_{\bmz_M}\left(\varphi;\sigma\right)}\leq}^{\textsc{via \autoref{lemma_5.10}}}\underbrace{
\limsup_{M\to\infty}\|\koop\bm{k}_{\varphi(\bmz_M)}^\sigma\|\leq}_{\textsc{via \eqref{eq_59}}}\|\koop-\mathfrak{C}\|.
        \end{align*}
        Therefore, we have now successfully established the lower bound for the essential norm of the Koopman operator over the RKHS $H_{\sigma,1,\Cn}$. 
        Now, we will work on the upper bound of the same and to that end, we fix some positive $r$ and $M\in\mathbb{Z}_+$. Then, pick an arbitrary $f\in H_{\sigma,1,\Cn}$ and proceed as follows:
        \begin{align*}
            \|\koop\mathcal{P}_Mf\|^2=&\int_{\Cn}|\koop\mathcal{P}_Mf(\bmz)|^2d\mu_{\sigma,1,\Cn}(\bmz).
        \end{align*}
Note that $\varphi$ admits an affine structure with $\varphi(\bmz)=\mathcal{A}\bmz+B$ where $0\not\equiv\mathcal{A}\in\mathbb{C}^{n\times n}$. Therefore, we can recall \eqref{eq_natural52} and \eqref{eq_natural53} to have following conclusion:
\begin{align*}
    \|\koop\mathcal{P}_Mf\|^2=&\frac{1}{(2\pi\sigma^2)^n}\int_{\Cn}\natural\left(\varpi\left(\bm{u}\right)\right)|\mathcal{P}_Mf\left(\bm{u}\right)|^2\exp\left(-\frac{\|\bm{u}\|_2}{\sigma}\right)\left\{\left|\det\left(\mathcal{A}^{-1}\right)\right|^n\right\}dV\left(\bm{u}\right)\\
    =&\frac{\left\{\left|\det\left(\mathcal{A}^{-1}\right)\right|^n\right\}}{(2\pi\sigma^2)^n}\int_{\Cn}\natural\left(\varpi\left(\bm{u}\right)\right)|\mathcal{P}_Mf\left(\bm{u}\right)|^2\exp\left(-\frac{\|\bm{u}\|_2}{\sigma}\right)dV\left(\bm{u}\right)\\
    =&\left[\frac{\left|\det\left(\mathcal{A}^{-1}\right)\right|}{2\pi\sigma^2}\right]^n\\&\int_{\Cn}\natural\left(\varpi\left(\bm{u}\right)\right)|\mathcal{P}_Mf\left(\bm{u}\right)|^2\left(\bm{\chi}_{r\mathbb{B}_n}\oplus\bm{\chi}_{\Cn\setminus r\mathbb{B}_n}\right)\exp\left(-\frac{\|\bm{u}\|_2}{\sigma}\right)dV\left(\bm{u}\right),
\end{align*}
where $\bm{\chi}_{\boxdot}$ is the indicator function for the sub-space $\boxdot\subset\Cn$. Due to the general theory from the linearity of integral, we have now two parts of integral which will be treated simultaneously as follows:
\begin{align*}
    \|\koop\mathcal{P}_Mf\|^2=&\left[\frac{\left|\det\left(\mathcal{A}^{-1}\right)\right|}{2\pi\sigma^2}\right]^n\cdot\int_{\Cn}\natural\left(\varpi\left(\bm{u}\right)\right)|\mathcal{P}_Mf\left(\bm{u}\right)|^2\bm{\chi}_{r\mathbb{B}_n}\exp\left(-\frac{\|\bm{u}\|_2}{\sigma}\right)dV\left(\bm{u}\right)\\&\qquad+\int_{\Cn}\natural\left(\varpi\left(\bm{u}\right)\right)|\mathcal{P}_Mf\left(\bm{u}\right)|^2\bm{\chi}_{\Cn\setminus r\mathbb{B}_n}\exp\left(-\frac{\|\bm{u}\|_2}{\sigma}\right)dV\left(\bm{u}\right).\numberthis\label{eq_62}
\end{align*}
Then, 
\begin{align*}
    \mathcal{I}_{r\mathbb{B}_n}^{\left\{M\right\}}=&\left[\frac{\left|\det\left(\mathcal{A}^{-1}\right)\right|}{2\pi\sigma^2}\right]^n\int_{\Cn}\natural\left(\varpi\left(\bm{u}\right)\right)|\mathcal{P}_Mf\left(\bm{u}\right)|^2\bm{\chi}_{r\mathbb{B}_n}\exp\left(-\frac{\|\bm{u}\|_2}{\sigma}\right)dV\left(\bm{u}\right)\\
    =&\left[\frac{\left|\det\left(\mathcal{A}^{-1}\right)\right|}{2\pi\sigma^2}\right]^n\int_{\Cn\cap r\mathbb{B}_n}\natural\left(\varpi\left(\bm{u}\right)\right)|\mathcal{P}_Mf\left(\bm{u}\right)|^2\exp\left(-\frac{\|\bm{u}\|_2}{\sigma}\right)dV\left(\bm{u}\right)\\
    =&\left[\frac{\left|\det\left(\mathcal{A}^{-1}\right)\right|}{2\pi\sigma^2}\right]^n\int_{r\mathbb{B}_n}\natural\left(\varpi\left(\bm{u}\right)\right)|\mathcal{P}_Mf\left(\bm{u}\right)|^2\exp\left(-\frac{\|\bm{u}\|_2}{\sigma}\right)dV\left(\bm{u}\right)\\
    \leq&\left[\frac{\left|\det\left(\mathcal{A}^{-1}\right)\right|}{2\pi\sigma^2}\right]^n\Pi\left(\varphi;\sigma\right)\|f\|^2\left(\sum_{N=M}^\infty\frac{r^{2N}}{(2N+1)!}\right)\int_{r\mathbb{B}_n}\exp\left(-\frac{\|\bm{u}\|_2}{\sigma}\right)dV\left(\bm{u}\right),
\end{align*}
where in the last step above we used inequality from \eqref{eq_natural53} and \eqref{eq_57}. Now, as we allow $M\to\infty$, the quantity $\sum_{N=M}^\infty\frac{r^{2N}}{(2N+1)!}\to0$ and hence 
\begin{align*}
    \lim_{M\to\infty}\mathcal{I}_{r\mathbb{B}_n}^{\left\{M\right\}}=0.
\end{align*}
After that we have optimized the limit over $r\mathbb{B}_n$, we now revisit \eqref{eq_62} to optimize the limit for the second part, that is over the compliment of $r\mathbb{B}_n$ as follows:
\begin{align*}
    \mathcal{I}_{(r\mathbb{B}_n)^\complement}
    =&\left[\frac{\left|\det\left(\mathcal{A}^{-1}\right)\right|}{2\pi\sigma^2}\right]^n\int_{\Cn}\natural\left(\varpi\left(\bm{u}\right)\right)|\mathcal{P}_Mf\left(\bm{u}\right)|^2\bm{\chi}_{\Cn\setminus r\mathbb{B}_n}\exp\left(-\frac{\|\bm{u}\|_2}{\sigma}\right)dV\left(\bm{u}\right)\\
    =&\left[\frac{\left|\det\left(\mathcal{A}^{-1}\right)\right|}{2\pi\sigma^2}\right]^n\int_{(r\mathbb{B}_n)^\complement}\natural\left(\varpi\left(\bm{u}\right)\right)|\mathcal{P}_Mf\left(\bm{u}\right)|^2\exp\left(-\frac{\|\bm{u}\|_2}{\sigma}\right)dV\left(\bm{u}\right)\\
    \leq&\left[\frac{\left|\det\left(\mathcal{A}^{-1}\right)\right|}{2\pi\sigma^2}\right]^n\int_{(r\mathbb{B}_n)^\complement}\left\{\sup_{\|\bm{u}\|_2\geq r}\natural\left(\varpi\left(\bm{u}\right)\right)\right\}|\mathcal{P}_Mf\left(\bm{u}\right)|^2\exp\left(-\frac{\|\bm{u}\|_2}{\sigma}\right)dV\left(\bm{u}\right)\\
    =&\left[\frac{\left|\det\left(\mathcal{A}^{-1}\right)\right|}{2\pi\sigma^2}\right]^n\left\{\sup_{\|\bm{u}\|_2\geq r}\natural\left(\varpi\left(\bm{u}\right)\right)\right\}\int_{(r\mathbb{B})^\complement}|\mathcal{P}_Mf\left(\bm{u}\right)|^2\exp\left(-\frac{\|\bm{u}\|_2}{\sigma}\right)dV\left(\bm{u}\right)\\
    =&\left(\left|\det\left(\mathcal{A}^{-1}\right)\right|\right)^n
    \left\{\sup_{\|\bm{u}\|_2\geq r}\natural\left(\varpi\left(\bm{u}\right)\right)\right\}\|\mathcal{P}_Mf\|^2\\
    \leq&\left(\left|\det\left(\mathcal{A}^{-1}\right)\right|\right)^n
    \left\{\sup_{\|\bm{u}\|_2\geq r}\natural\left(\varpi\left(\bm{u}\right)\right)\right\}\|f\|^2.\numberthis\label{eq_63}
\end{align*}
Letting $r\to\infty$ and combining the result of \eqref{eq_63} and result from \autoref{lemma_5.16}, we eventually have following
\begin{align*}
    \|\koop\|_{\operatorname{ess}}\leq&\sqrt{\left(\left|\det\left(\mathcal{A}^{-1}\right)\right|\right)^n}\|f\|\lim_{r\to\infty}\sqrt{\left\{\sup_{\|\bm{u}\|_2\geq r}\natural\left(\varpi\left(\bm{u}\right)\right)\right\}}\\
    =&\left(\left|\det\left(\mathcal{A}^{-1}\right)\right|\right)^{\nicefrac{n}{2}}\|f\|\lim_{r\to\infty}\sup_{\|\bm{u}\|_2\geq r}\left\{\sqrt{\natural\left(\varpi\left(\bm{u}\right)\right)}\right\}.
\end{align*}
Thus, the result prevails.
    \end{proof}
    \subsection{Compactness of Koopman operators over RKHS \texorpdfstring{$H_{\sigma,1,\mathbb{C}^n}$}{}}
    In the previous subsection, we provided the essential norm estimates for the Koopman operators over the RKHS $H_{\sigma,1,\Cn}$. Now, we will provide the compactness characterization of the Koopman operator over the same to extract the finite rank representation. Thus, the compactification criteria is given as follows:
    \begin{theorem}
        Let $\sigma>0$ and $\varphi:\Cn\to\Cn$ be a holomorphic mapping from $\Cn\to\Cn$ where every coordinate functions of $\varphi$ are holomorphic function from $\Cn\to\mathbb{C}$. Let $\koop:\mathcal{D}\left(\koop\right)\to H_{\sigma,1,\Cn}$ be the Koopman operator induced by $\varphi$ over the RKHS $H_{\sigma,1,\Cn}$. If $\koop$ is bounded over the RKHS $H_{\sigma,1,\Cn}$, then the \textbf{compactification} of the Koopman operator over RKHS $H_{\sigma,1,\Cn}$ is possible if and only if $\lim_{\|\bmz\|_2\to\infty}\Pi_{\bmz}\left(\varphi;\sigma\right)=0$, where $\varphi(\bmz)=\mathcal{A}\bmz+B$ with $0\not\equiv\mathcal{A}\in\mathbb{C}^{n\times n}$ and $\mathcal{A}$ is also invertible.
    \end{theorem}
\section{Experimental Results for Fluid flow across cylinder}\label{section_experimentsffac}
To demonstrate the ability of executing the Kernelized eDMD by the advantage of the compactification of the Koopman operator over the RKHS $H_{\sigma,1,\Cn}$ which is generated by the $L^2-$embedding into the normalized Laplacian measure, fluid flow around cylinder, which is the standard experiment, is considered. The data is generated from the snapshots of the numerical simulation of the in-compressible Navier-Stokes equation \cite[Page 286]{brunton2022data}:
\begin{align*}
    \frac{\partial}{\partial t}\mathbf{u}(x,y,t)+\mathbf{u}(x,y,t)\cdot\nabla\mathbf{u}(x,y,t)+\nabla p(x,y,t)-\frac{1}{\textsc{Re}}\nabla^2\mathbf{u}(x,y,t)=0,
\end{align*}
where $\mathbf{u}(x,y,t)$ represents the $2D$ velocity and $p(x,y,t)$ is the (corresponding) pressure field, along with the incompressibility constraints $\nabla\cdot\mathbf{u}=0$. The boundary conditions governs a constant flow given as $\mathbf{u}=(1,0)^\top$ at $x=-15$ (the entry of the domain), constant pressure of $p=0$ at $x=25$ (the end of the domain) and $\nicefrac{\partial\mathbf{u}}{\partial\mathbf{n}}=0$, on the boundary of the domain.
\subsection{Construction of Koopman Modes via Laplacian and GRBF Kernel Functions}
In this subsection, we will provide the dominant Koopman modes (both real and imaginary values) for fluid flow across cylinder experiment for whose all the experimental details were already provided in the introduction of this section. Here, will will consider both of the kernel functions for the construction of the Koopman modes. Before, we give the results, recall the mathematical structure of the class of exponential power kernels given in \eqref{eq_expkernel}; from here, we get the Laplacian Kernel Function and the Gaussian RBF Kernel.
\begin{note}
    In the upcoming sections, the result for the Koopman modes will be recorded which are dominant for fluid flow across cylinder experiment. The experimental images mentions explicitly {\tt{via GRBF Kernel}} if we performed the experiment by exploiting Gaussian RBF Kernel, otherwise we employed the Laplacian Kernel Function. We also write LDA to mean \emph{Limited Data Acquisition} and FDA to mean \emph{Full Data Acquisition} in the caption of figures.
\end{note}
\subsubsection{Real dominant Koopman modes} 
To provide the figurative details for the Koopman modes of fluid flow across cylinder experiment, we will consider the various numbers of snapshots that are available as a \emph{known part} and remaining gets padded by the Gaussian random matrix of suitable matrix dimension. To this end, we will consider the situations when only $3,~7,~20,~55$ and $151$ snapshots were available to construct the desired Koopman modes for the experiment. It should be noted that $151$ are the actual or the total number of snapshots assembled for this experiment. Now, we present the images for the real part of the dominant Koopman modes when the aforementioned snapshots were provided while the experiment was conducted.
\begin{figure}[H]
    \centering
    \includegraphics[scale=.27]{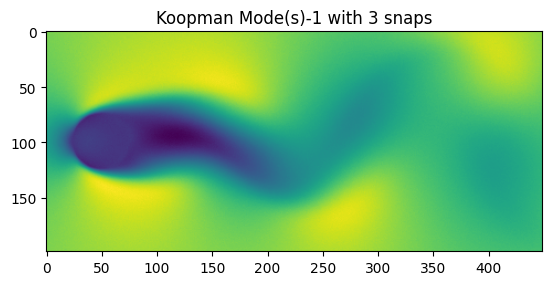}
    \includegraphics[scale=.27]{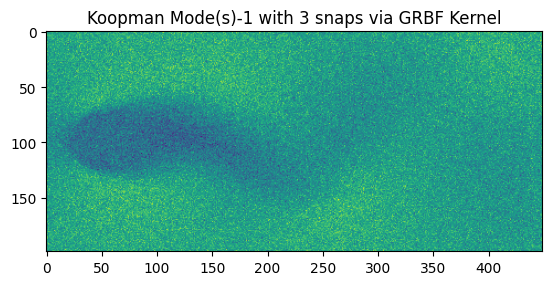}
    \includegraphics[scale=.27]{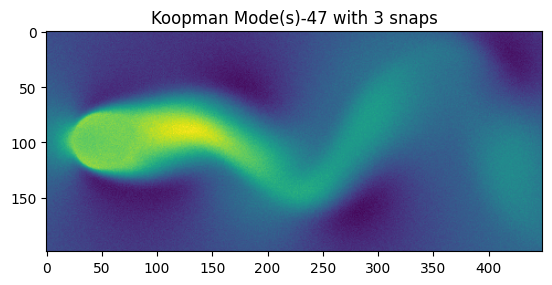}
    \includegraphics[scale=.27]{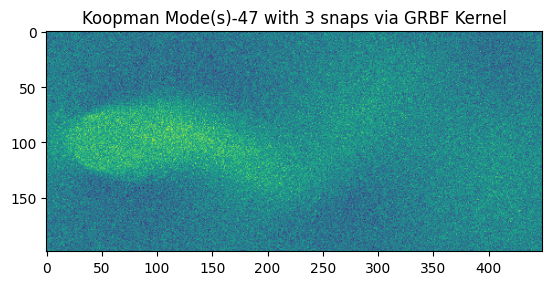}
    \includegraphics[scale=.27]{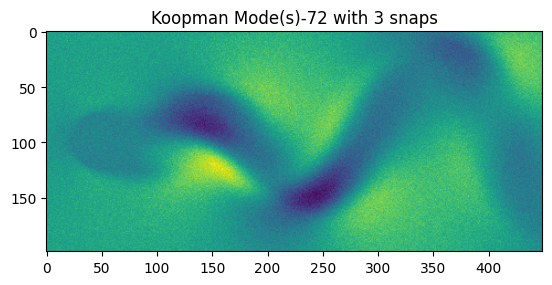}
    \includegraphics[scale=.27]{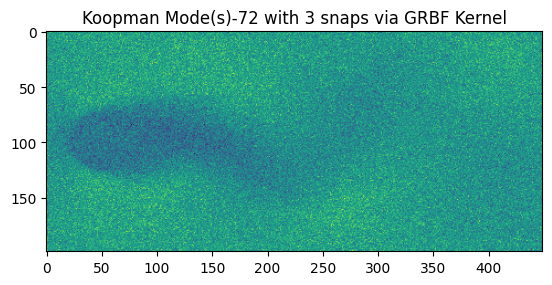}
    \includegraphics[scale=.27]{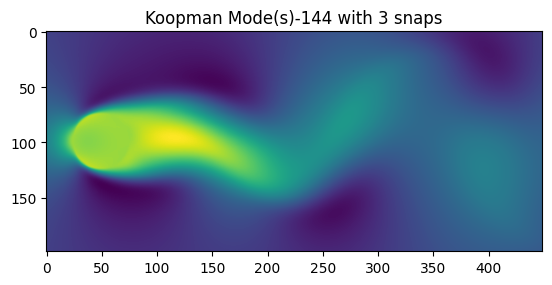}
    \includegraphics[scale=.27]{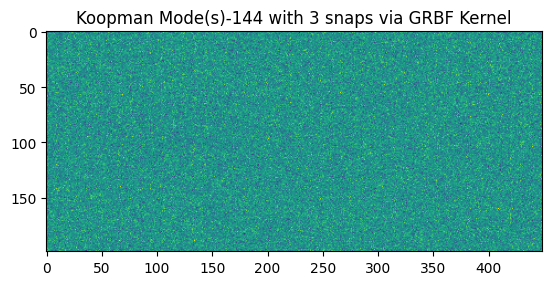}
    \caption{Dominant Koopman Modes with only $03-$snapshots [LDA]}
    \label{fig:realKoopman3snaps}
\end{figure}
\begin{figure}[H]
    \centering
    \includegraphics[scale=.27]{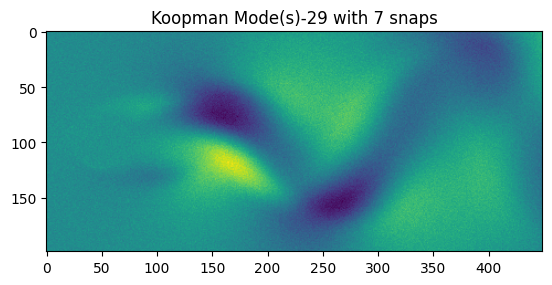}
    \includegraphics[scale=.27]{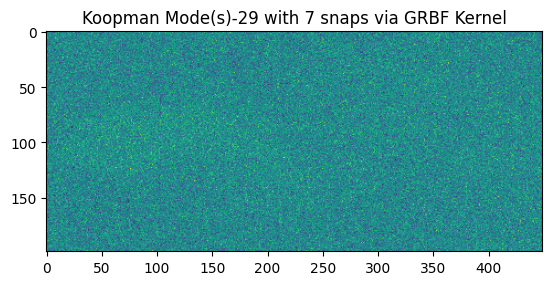}
    \includegraphics[scale=.27]{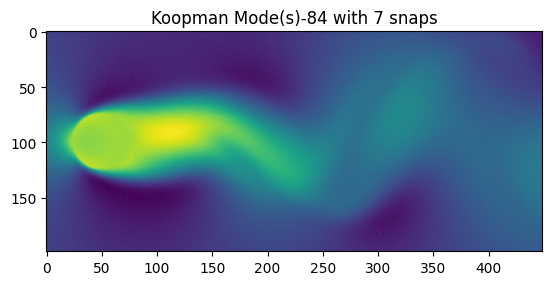}
    \includegraphics[scale=.27]{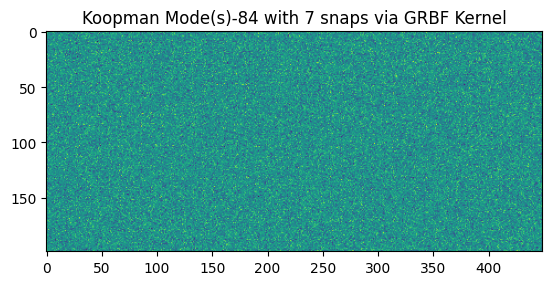}
    \includegraphics[scale=.27]{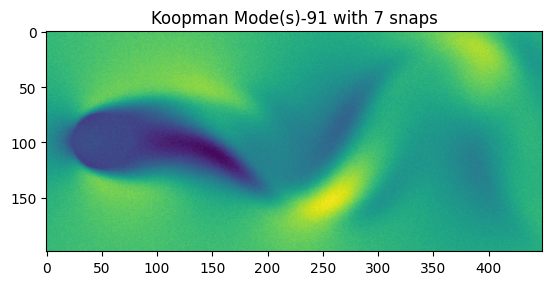}
    \includegraphics[scale=.27]{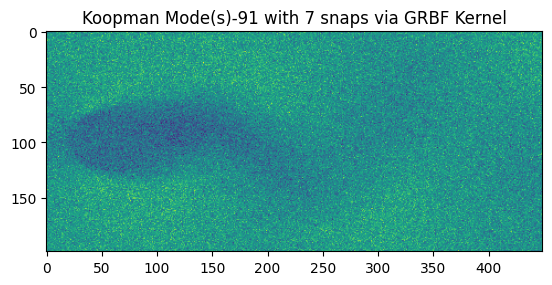}
    \includegraphics[scale=.27]{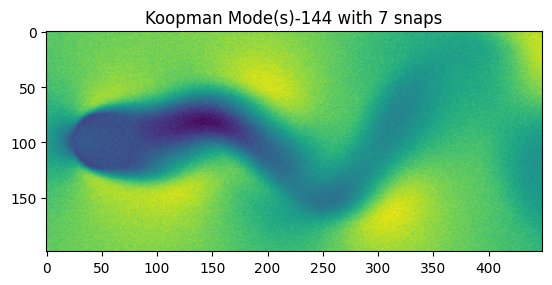}
    \includegraphics[scale=.27]{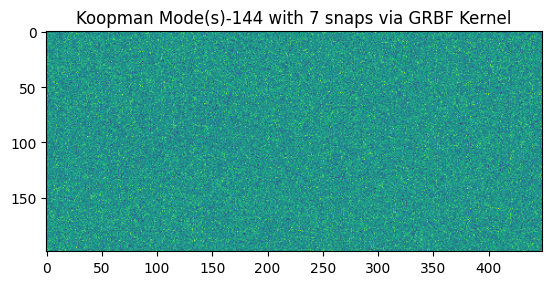}
    \caption{Dominant Koopman Modes with only $07-$snapshots [LDA]}
    \label{fig:realKoopman7snaps}
\end{figure}
\begin{figure}[H]
    \centering
    \includegraphics[scale=.27]{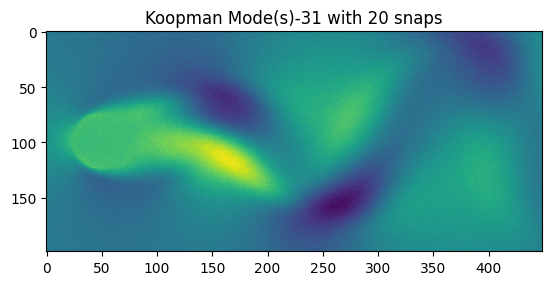}
    \includegraphics[scale=.27]{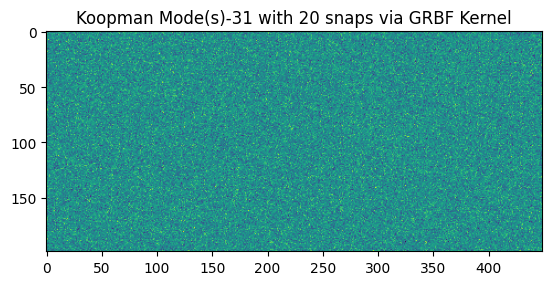}
    \includegraphics[scale=.27]{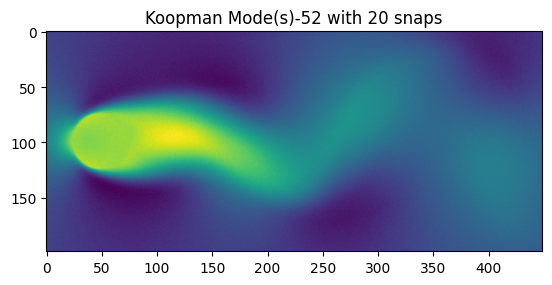}
    \includegraphics[scale=.27]{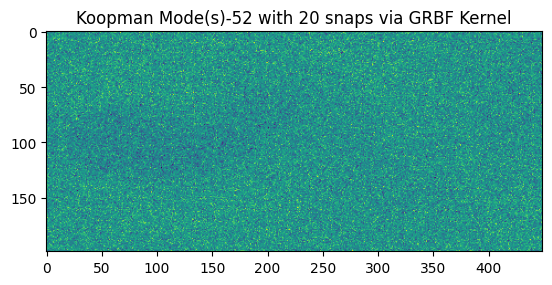}
    \includegraphics[scale=.27]{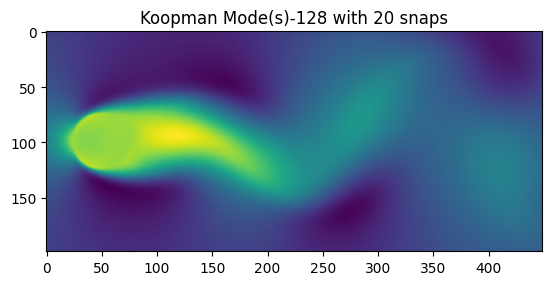}
    \includegraphics[scale=.27]{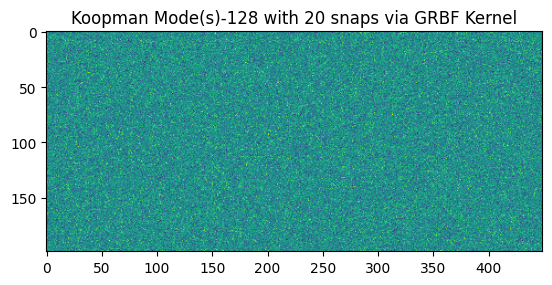}
    \includegraphics[scale=.27]{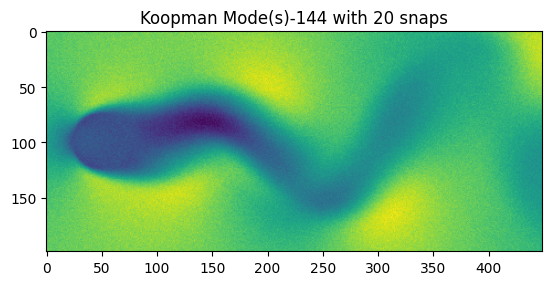}
    \includegraphics[scale=.27]{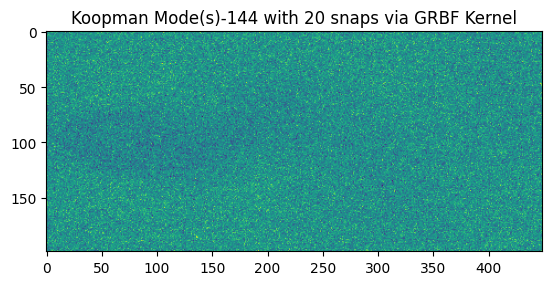}
    \caption{Dominant Koopman Modes with only $20-$snapshots [LDA]}
    \label{fig:realKoopman20snaps}
\end{figure}
\begin{figure}[H]
    \centering
    \includegraphics[scale=.27]{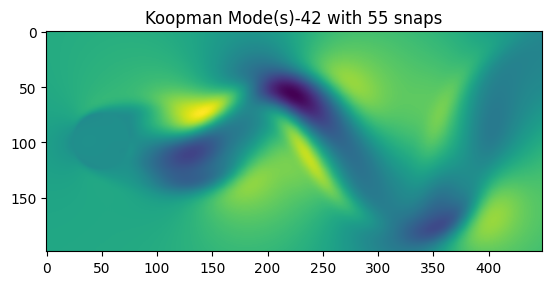}
    \includegraphics[scale=.27]{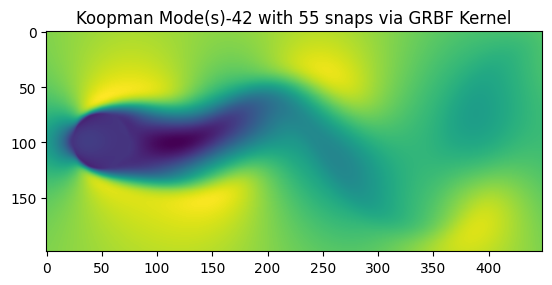}
    \includegraphics[scale=.27]{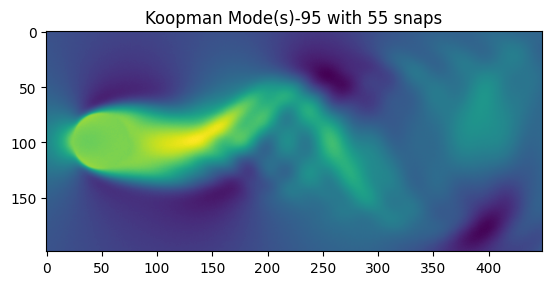}
    \includegraphics[scale=.27]{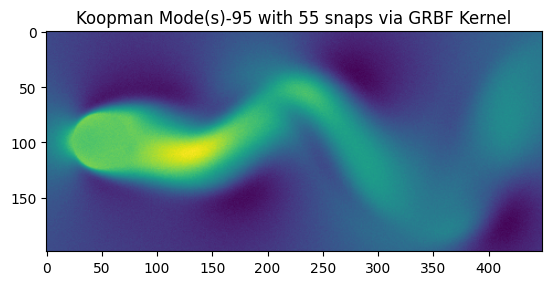}
    \includegraphics[scale=.27]{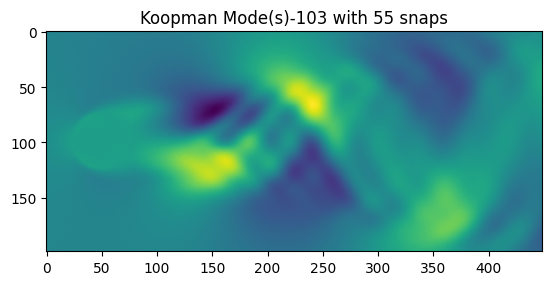}
    \includegraphics[scale=.27]{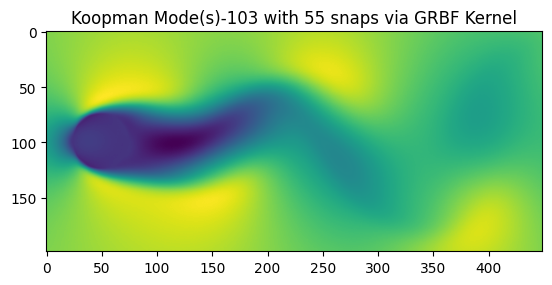}
    \includegraphics[scale=.27]{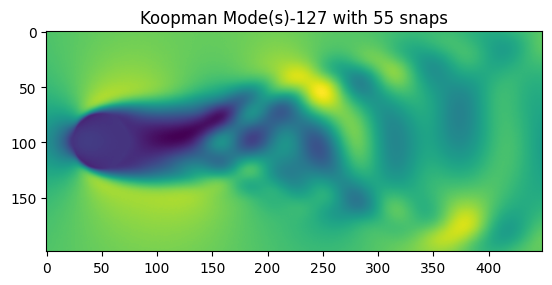}
    \includegraphics[scale=.27]{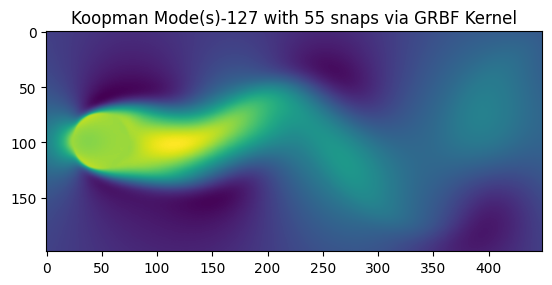}
    \caption{Dominant Koopman Modes when only $55-$snapshots [LDA]}
    \label{fig:realKoopman55snaps}
\end{figure}
\begin{figure}[H]
    \centering
    \includegraphics[scale=.27]{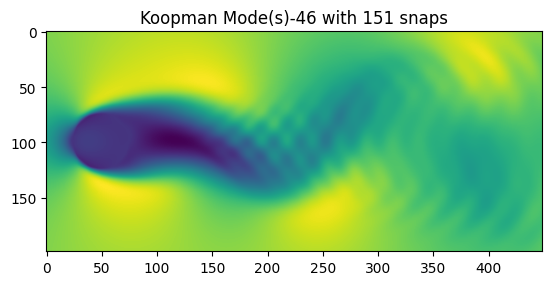}
    \includegraphics[scale=.27]{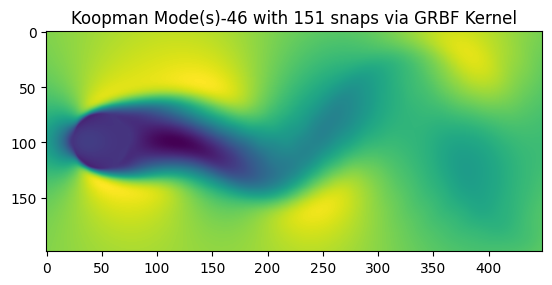}
    \includegraphics[scale=.27]{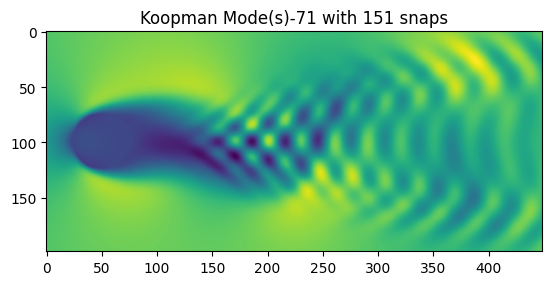}
    \includegraphics[scale=.27]{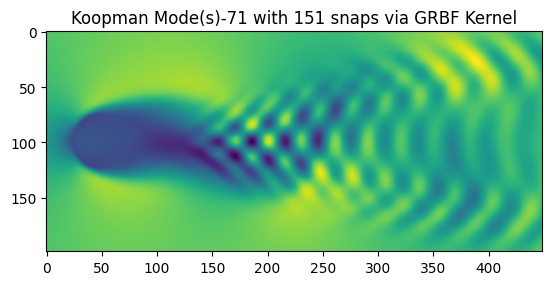}
    \includegraphics[scale=.27]{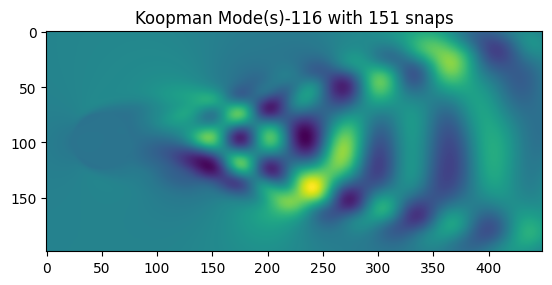}
    \includegraphics[scale=.27]{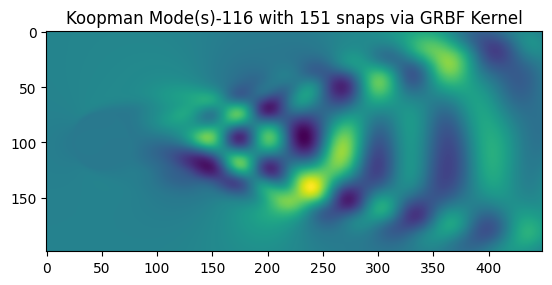}
    \includegraphics[scale=.27]{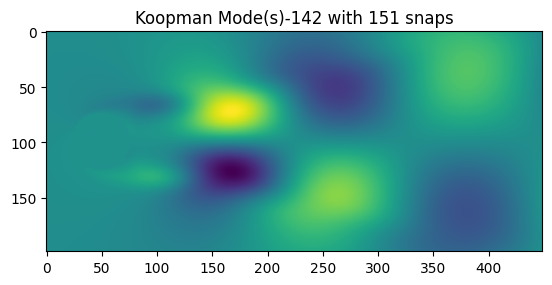}
    \includegraphics[scale=.27]{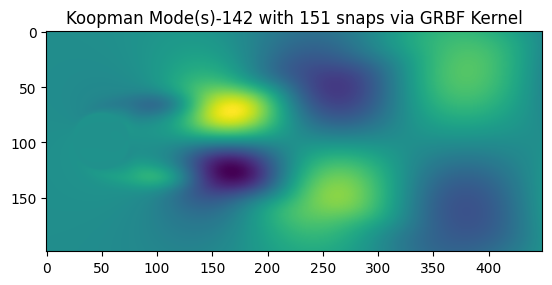}
    \caption{Dominant Koopman Modes with all $151-$snapshots [FDA]}
    \label{fig:realKoopman151snaps}
\end{figure}
\subsubsection{Imaginary dominant Koopman modes}
In the continuation, now, we will provide the imaginary part of the corresponding dominant Koopman modes. 
\begin{figure}[H]
    \centering
    \includegraphics[scale=.27]{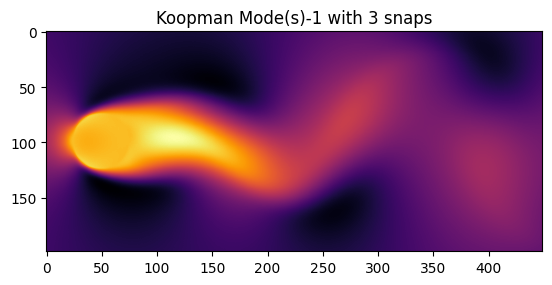}
    \includegraphics[scale=.27]{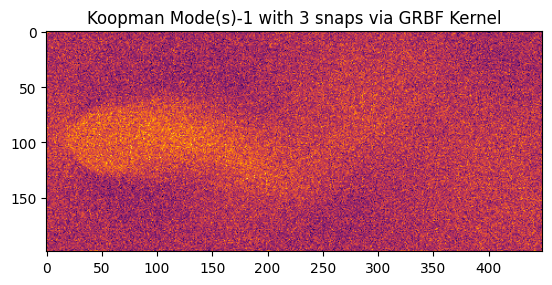}
    \includegraphics[scale=.27]{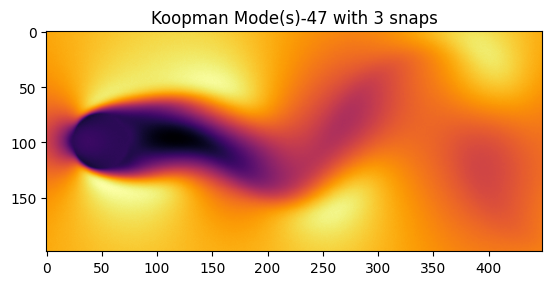}
    \includegraphics[scale=.27]{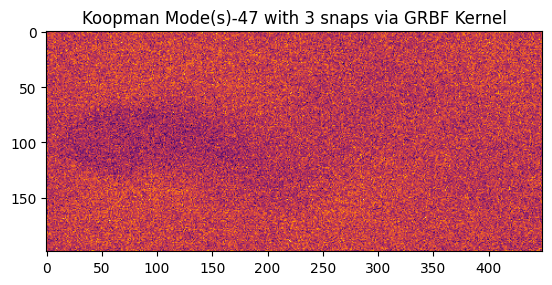}
    \includegraphics[scale=.27]{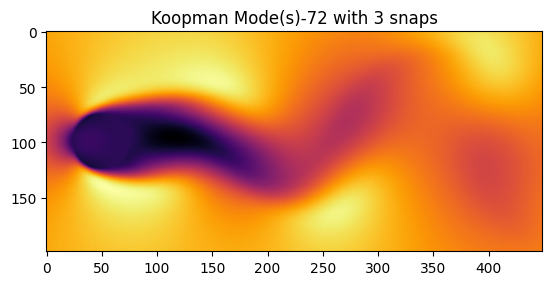}
    \includegraphics[scale=.27]{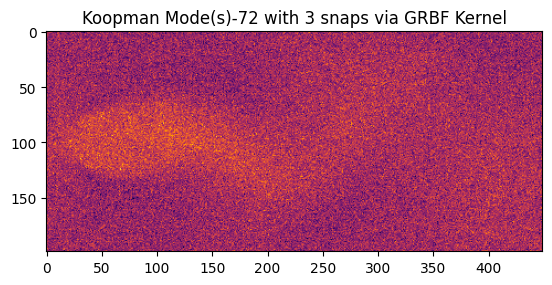}
    \includegraphics[scale=.27]{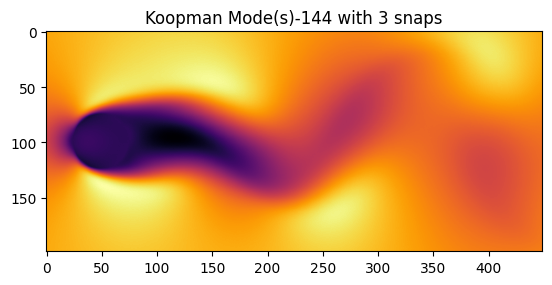}
    \includegraphics[scale=.27]{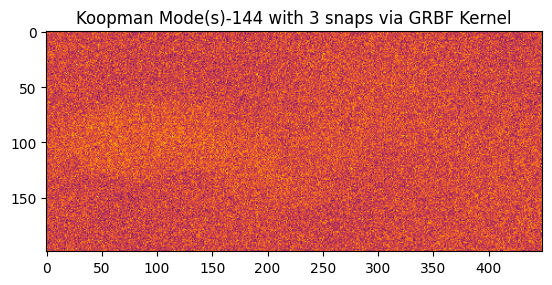}
    \caption{Dominant Koopman Modes with only $03-$snapshots [LDA]}
    \label{fig:imagKoopman03snaps}
\end{figure}
\begin{figure}[H]
    \centering
    \includegraphics[scale=.27]{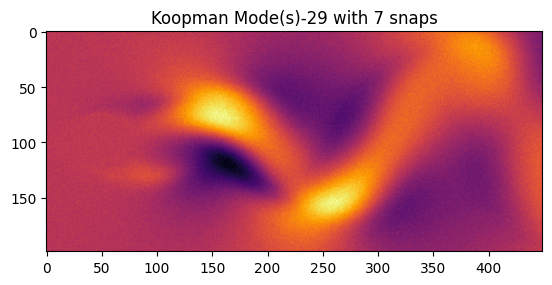}
    \includegraphics[scale=.27]{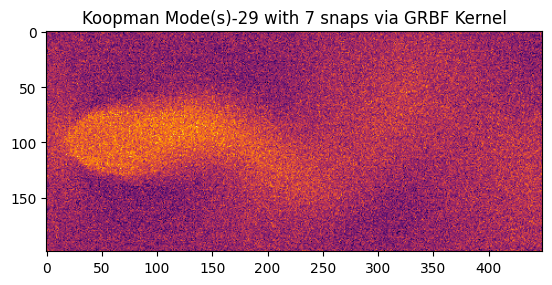}
    \includegraphics[scale=.27]{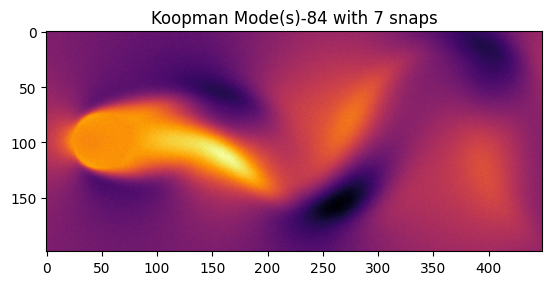}
    \includegraphics[scale=.27]{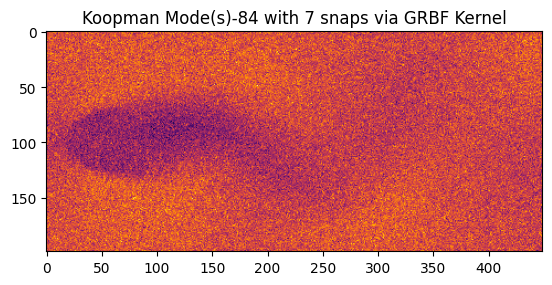}
    \includegraphics[scale=.27]{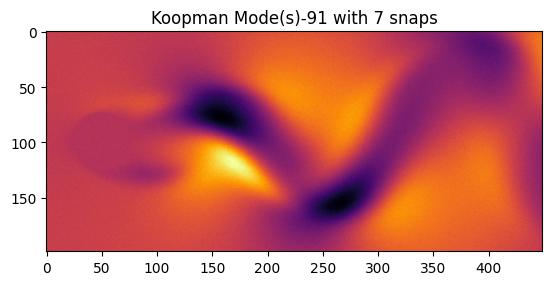}
    \includegraphics[scale=.27]{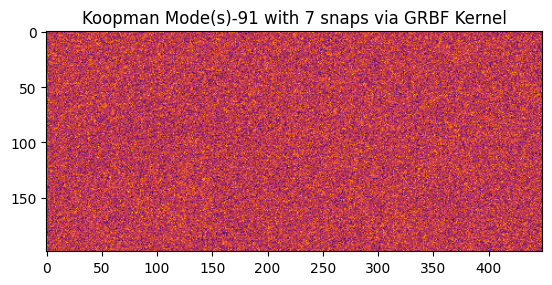}
    \includegraphics[scale=.27]{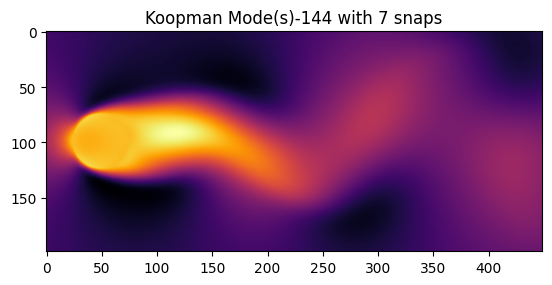}
    \includegraphics[scale=.27]{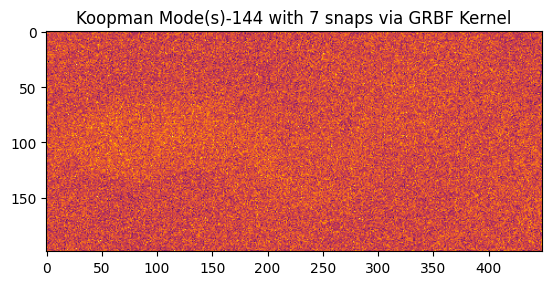}
    \caption{Dominant Koopman Modes with only $07-$snapshots [LDA]}
    \label{fig:imagKoopman07snaps}
\end{figure}
\begin{figure}[H]
    \centering
    \includegraphics[scale=.27]{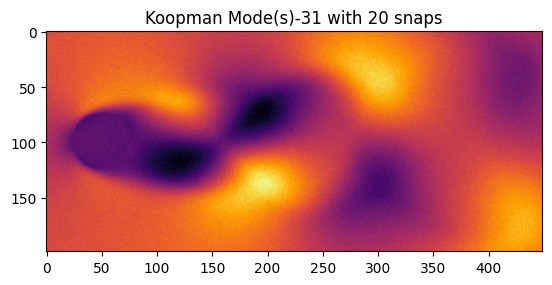}
    \includegraphics[scale=.27]{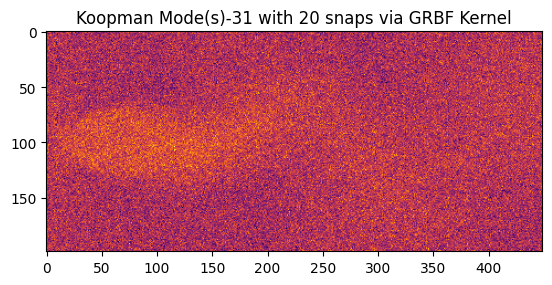}
    \includegraphics[scale=.27]{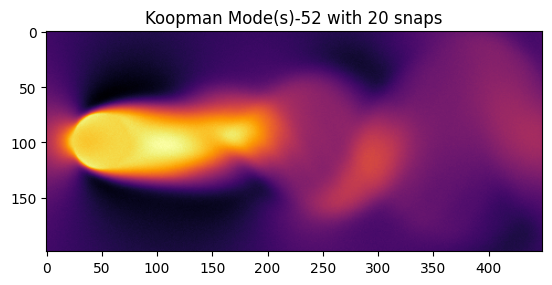}
    \includegraphics[scale=.27]{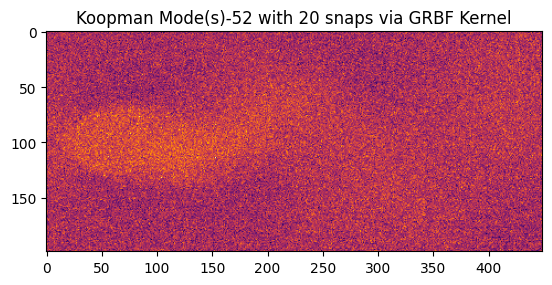}
    \includegraphics[scale=.27]{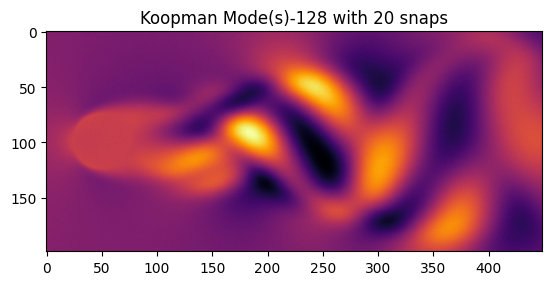}
    \includegraphics[scale=.27]{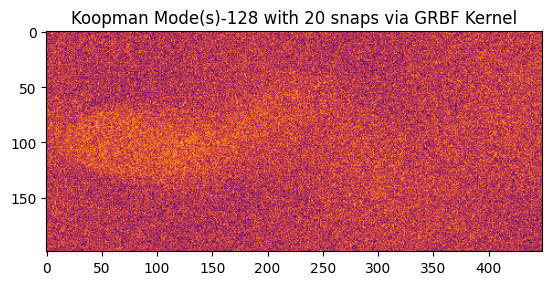}
    \includegraphics[scale=.27]{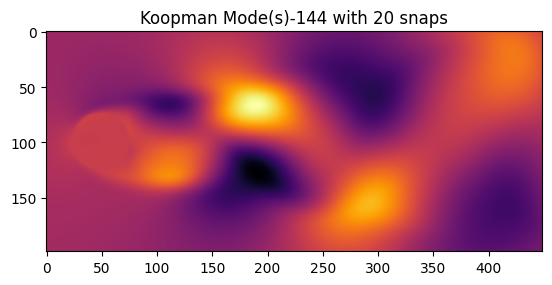}
    \includegraphics[scale=.27]{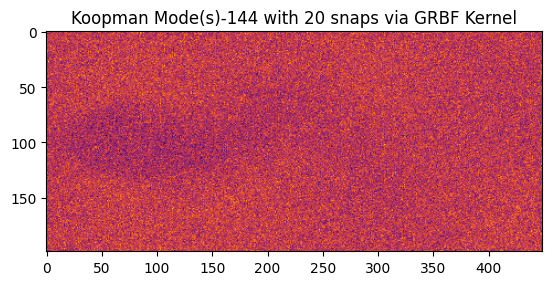}
    \caption{Dominant Koopman Modes with only $20-$snapshots [LDA]}
    \label{fig:imagKoopman20snaps}
\end{figure}
\begin{figure}[H]
    \centering
    \includegraphics[scale=.27]{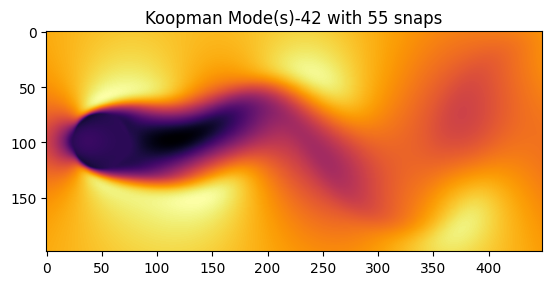}
    \includegraphics[scale=.27]{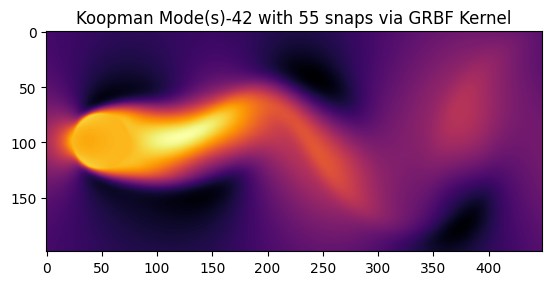}
    \includegraphics[scale=.27]{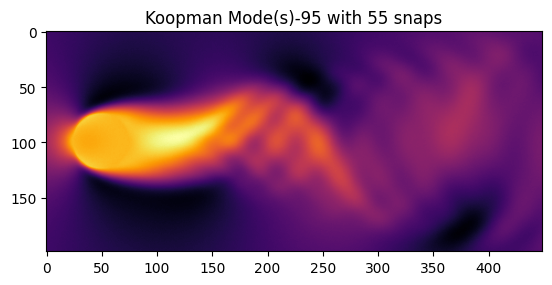}
    \includegraphics[scale=.27]{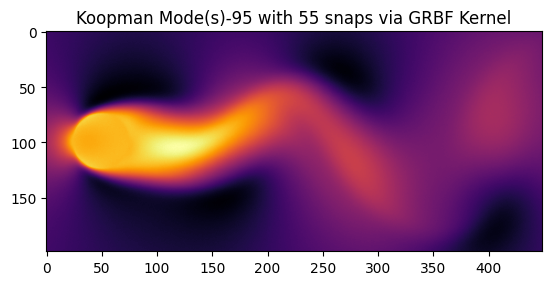}
    \includegraphics[scale=.27]{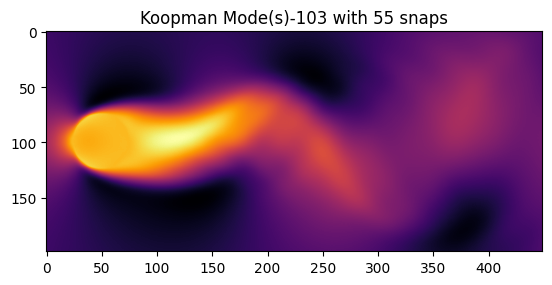}
    \includegraphics[scale=.27]{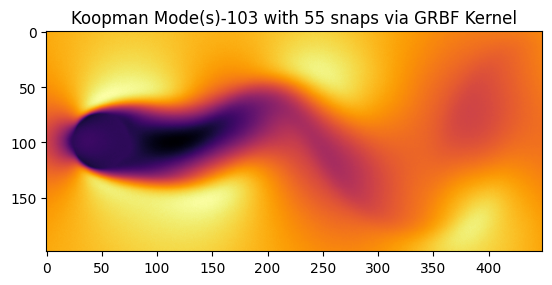}
    \includegraphics[scale=.27]{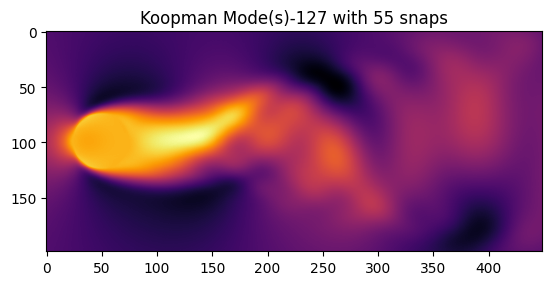}
    \includegraphics[scale=.27]{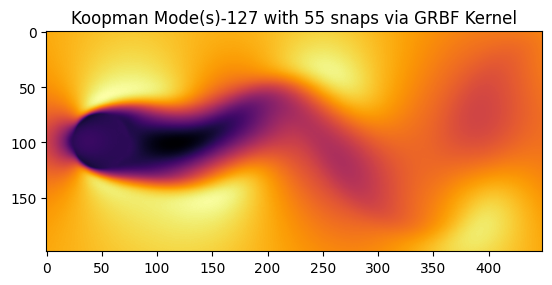}
    \caption{Dominant Koopman Modes with only $55-$snapshots [LDA]}
    \label{fig:imagKoopman55snaps}
\end{figure}
\begin{figure}[H]
    \centering
    \includegraphics[scale=.27]{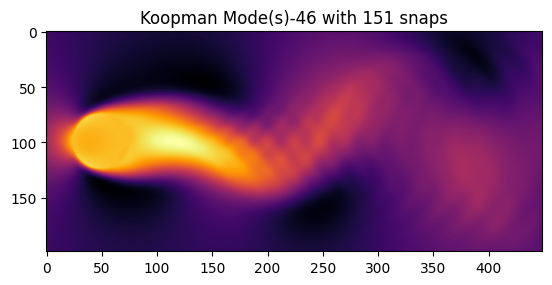}
    \includegraphics[scale=.27]{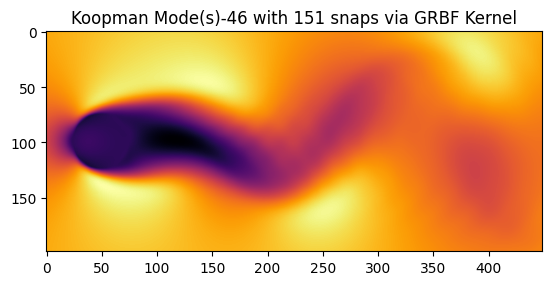}
    \includegraphics[scale=.27]{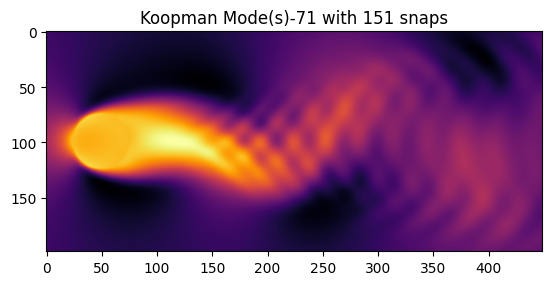}
    \includegraphics[scale=.27]{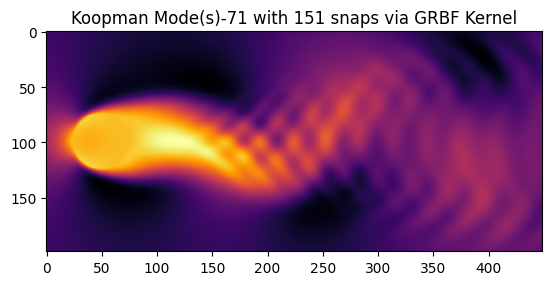}
    \includegraphics[scale=.27]{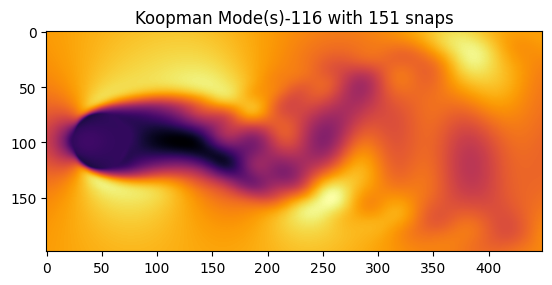}
    \includegraphics[scale=.27]{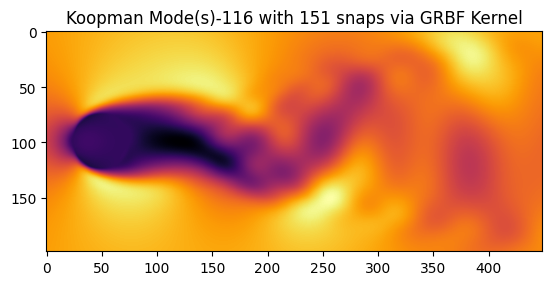}
    \includegraphics[scale=.27]{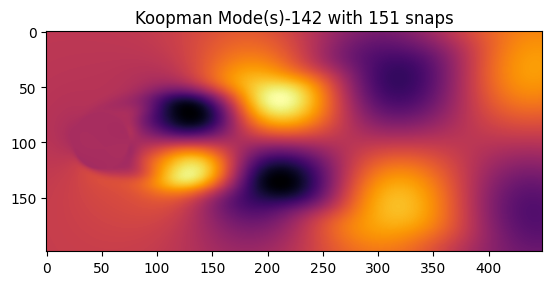}
    \includegraphics[scale=.27]{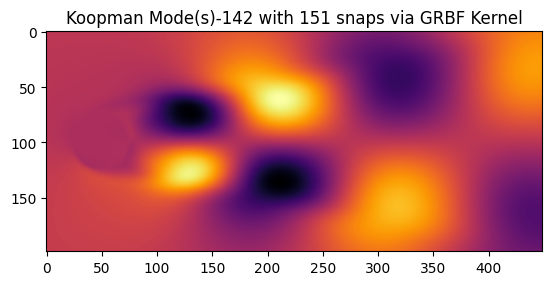}
    \caption{Dominant Koopman Modes with all $151-$snapshots [FDA]}
    \label{fig:imagKoopman151snaps}
\end{figure}
\subsection{Executive conclusion}
We presented both real and imaginary part of the dominant Koopman modes for fluid flow across cylinder experiment when limited number data snapshots were provided in the previous subsection. Following is a concise summary for the executive conclusion that we draw when we compare the results of both real and imaginary dominant Koopman modes that we get via employing respective both of the kernel functions. 
\subsubsection{Dominant Koopman modes constructed via Laplacian \& GRBF}
Of what follows, the dominant Koopman modes generated when we allowed the interaction of the Koopman operators on the Laplacian Kernel Function even when the as low as only \emph{3-}snapshots were available, we can successfully observe that in \autoref{fig:realKoopman3snaps} that in the result, we can clearly visualize the cylinder as well as the flow of the fluid as well. However, the same setup but with Gaussian RBF Kernel unfortunately fails to deliver the result which in fact also includes noise as well. Now, upon increasing the snapshots from \emph{3-}snapshots to \emph{7-}snapshots, we fail to see that the Gaussian RBF Kernel recovers the information in the form of Koopman modes. However, the respective dominant modes via the Laplacian Kernel Function is now achieving more maturity as it now can distinguishly highlights detailed eddies of fluid flow across cylinder. We do see the same result and conclusion when we move from \emph{7-} to \emph{20-}snapshots of this experiment. However, in the case when we avail \emph{55-}snapshots for the experiment study, we are successfully able to point out details that are provided by Gaussian RBF Kernel however, in the same situation, as one can observe that the results via the Laplacian Kernel Function has more depth and nuance in terms of providing vorticity profile for fluid flow across cylinder. Lastly, when all \emph{151-}snapshots are given, we see now, that Gaussian RBF Kernel can perfectly construct the Koopman modes. Additionally, it should also be noted here that the wiggly nature of fluid passing by cylinder appears quite at the early stage when we use the Laplacian Kernel Function.
\subsubsection{Gram-Matrix for data-set snapshots}After that we have analysed the Koopman modes generated by the \autoref{algo_lapkerneleDMD}, now we will provide the matrix structure with suitable color gradient scheme which on-the-spot informs on how many actual snapshots were provided to execute the aforementioned algorithm.
\begin{figure}[H]
    \centering
    \includegraphics[scale=.3]{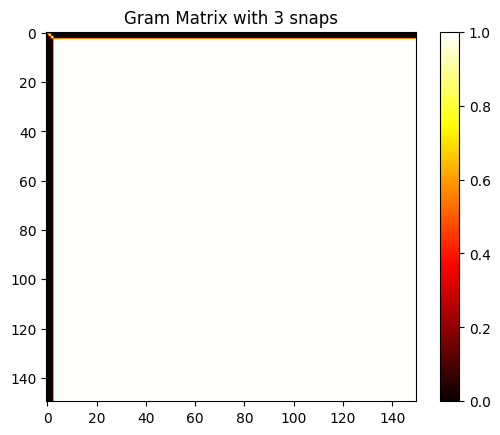}
    \includegraphics[scale=.3]{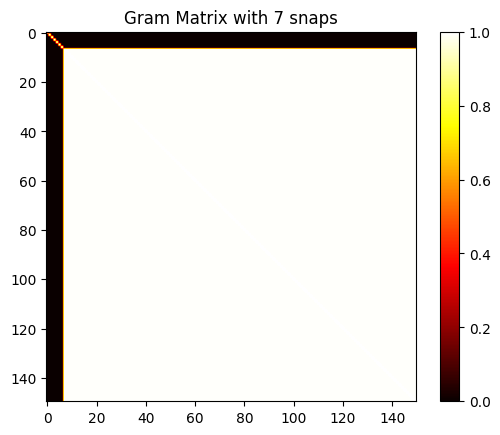}
    \includegraphics[scale=.3]{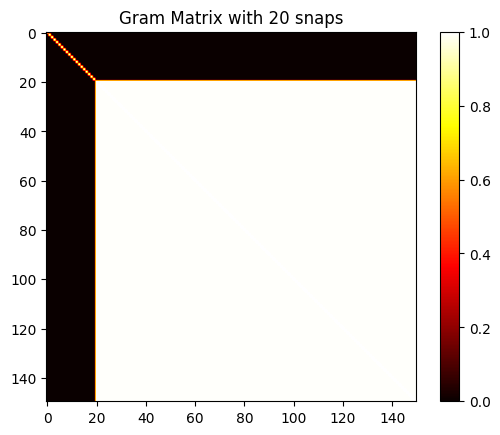}
    \includegraphics[scale=.3]{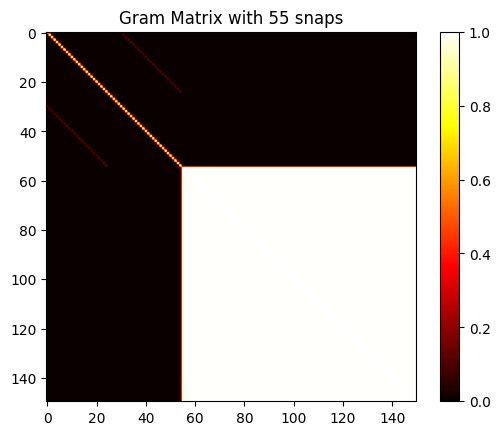}
    \includegraphics[scale=.3]{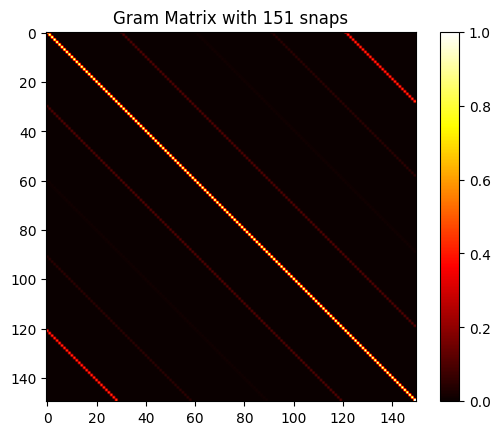}
    \caption{Various Gram-Matrix produced while executing Laplacian Kernel based eDMD given in \autoref{algo_lapkerneleDMD} coupled with Gaussian random matrix. }
    \label{fig:Gram-Matrix3snaps}
\end{figure}
\section{Why RKHS generated by the Laplacian measure is novel?}\label{section_novelLap}
From the previous section where we provide the empirical evidence where Laplacian Kernel easily outperforms the construction of the Koopman modes for fluid flow across cylinder experiment against the Gaussian RBF Kernel under the scope of limited data acquisition. This immediately warns us about the choice of using the reproducing kernel along with its RKHS into our data-science practices. Therefore, in the present section we will discuss the distinguishing property of the corresponding RKHS $H_{\sigma,1,\Cn}$ of the normalized Laplacian measure which eventually makes it novel and unique. For this, we review first the definition of the closable operator over general Hilbert space.
\subsection{Review of Closable operator in Hilbert space}
We recall when we mean an operator $T$ in a Hilbert space $\mathfrak{H}$ to be \emph{closable} or \emph{preclosed} as given in standard functional analysis references such as \cite[Chapter 13]{rudin1991functional}, \cite[Chapter X, Page 304]{conway2019course} \cite[Chapter 5, Page 193]{pedersen2012analysis}, or \cite[Chapter VIII, Page 250]{reed2012methods}. We define that particular notion systematically as follows:
\begin{dfn}[Graph of an operator]
    For an (unbounded) operator $T$ in Hilbert space $\mathfrak{H}$ with its domain $\mathcal{D}(T)$, we define the \emph{graph of $T$} in $\mathfrak{H}$ as follows:
    \begin{align}\label{eq_graphofT}
        \bm{\Gamma}(T)\coloneqq\left\{(x,Tx):x\in\mathcal{D}(T)\right\}.  
    \end{align}
\end{dfn}
\begin{dfn}[Extension of an operator in Hilbert space]
    Let $T_{\boxdot}$ and $T$ be operators over the Hilbert space $\mathfrak{H}$. Let $\bm{\Gamma}(T_{\boxdot})$ and $\bm{\Gamma}(T)$ be the respective graphs of $T_{\boxdot}$ and $T$ as defined in \eqref{eq_graphofT}. If $\bm{\Gamma}(T)\subset\bm{\Gamma}(T_{\boxdot})$, then $T_\boxdot$ is said to be an extension of $T$ and we write $T\subset T_\boxdot$ and equivalently if $T\subset T_\boxdot$ if and only if $\mathcal{D}(T)\subset\mathcal{D}(T_\boxdot)$ and $T_\boxdot\Lambda=T\Lambda$ for all $\Lambda\in\mathcal{D}(T)$.
\end{dfn}
\begin{dfn}[Closable operator in Hilbert space]
    An operator is \textbf{closable} if it has a closed extension.
\end{dfn}
As now we have given all the essential details regarding the notion of an unbounded operator to be closable over the Hilbert space, we now provide an easy characterization for the operator to be closable over the underlying Hilbert space which can be an easy functional analysis exercise.
\begin{lemma}\label{lemma_closable_characterization}
    The operator $T$ in Hilbert space $\mathfrak{H}$ is closable if and only if for each sequence $\left\{x_n\right\}_n\in\mathcal{D}(T)$ converging to $0$, the only accumulation point of $\left\{Tx_n\right\}_n$ is $0$.
\end{lemma}
The above lemma (cf. \cite[Chapter 5, Page 193]{pedersen2012analysis}) can be interpreted as follows: \emph{for a linear operator $T:\mathcal{D}(T)\to\mathfrak{H}$ is closable if and only if for any sequence $x_n$ such that $x_n\to0$ when $n\to\infty$ and $Tx_n\to y_n$, then $y_n=0$.} 
\autoref{lemma_closable_characterization} will be used to demonstrate that bounded Koopman operators $\koop$ which will be induced by holomorphic function $\varphi:\Cn\to\Cn$ is closable on the RKHS $H_{\sigma,1,\Cn}$ by constructing such a sequence of function whose behaviour follows the required condition presented in \autoref{lemma_closable_characterization}. However, the same cannot performed for the Gaussian RBF Kernel.
\subsection{Closability of Koopman operator over the RKHS of Laplacian measure}
We will now demonstrate how the bounded Koopman operator is closable over the RKHS $H_{\sigma,1,\Cn}$ with the help of its reproducing kernel function $K^{\sigma}(\cdot,\cdot)$. To this end, we shall keep $\sigma>0$ and consider the following operator $\mathfrak{I}_{-}$ defined over $\Cn$ as:
\begin{align}
    \mathfrak{I}_{-}\bmz=-I_{n}\bmz,
\end{align}
where $I_n$ is the identity matrix over $\Cn$. The operator define above is an injective linear operator and hence the null-space is explicitly the zero vector in $\Cn$. 
Now, we define the graph of the operator $\mathfrak{I}_{-}$ as follows:
\begin{align}\label{eq_78}
    _-\mathfrak{Z}_+\left(\mathfrak{I}_{-}\right)\coloneqq\left\{\left(\bmz,\mathfrak{I}_{-}\bmz\right)\in\Cn\times\Cn:\bmz\in\Cn\right\}\triangleq_-\mathfrak{Z}_+.
\end{align}Therefore, by the definition of the graph of $\mathfrak{I}_{-}$, essentially it is a coordinate system presented in the format of $\left\{\left(\bmz,-\bmz\right)\in\Cn\times\Cn:\bmz\in\Cn\right\}$. It is worth-while to mention that the graph of $\mathfrak{I}_{-}$ defined in \eqref{eq_78} is a closed subspace of $\Cn\times\Cn$. 

Interestingly, the formulation of the reproducing kernel function $K^{\sigma}(\cdot,\cdot)$ from \eqref{eq_21} over the set of coordinates present on $_-\mathfrak{Z}_+$ as defined in \eqref{eq_78}, we have following:
\begin{align*}
    K^{\sigma}\left(_-\mathfrak{Z}_+\right)=\frac{\sinh{\left(\sqrt{ \frac{\langle\bm{z},-\bm{z}\rangle_{\mathbb{C}^n}}{\sigma^2}}
        \right)}}{\sqrt{\frac{\langle\bm{z},-\bm{z}\rangle_{\mathbb{C}^n }}{\sigma^2}}
        }=\frac{\sinh{\left(\sqrt{-\frac{\langle\bmz,\bmz\rangle_{\Cn}}{\sigma^2}}\right)}}{\sqrt{-\frac{\langle\bmz,\bmz\rangle_{\Cn}}{\sigma^2}}}=\frac{\sinh{\left(\sqrt{-\frac{\|\bmz\|_{2}^2}{\sigma^2}}\right)}}{\sqrt{-\frac{\|\bmz\|_{2}^2}{\sigma^2}}}=&\frac{\sinh\left({i\frac{\|\bmz\|_2}{\sigma}}\right)}{i\frac{\|\bmz\|_2}{\sigma}}\\
        =&\frac{\sin\left({\frac{\|\bmz\|_2}{\sigma}}\right)}{{\frac{\|\bmz\|_2}{\sigma}}}.
\end{align*}
Now for the upcoming investigation, we will focus on the quantity ${\|\bmz\|_2}K^{\sigma}\left(_-\mathfrak{Z}_{+}\right)$. It is very important to show that quantity ${\|\bmz\|_2}K^{\sigma}\left(_-\mathfrak{Z}_{+}\right)$, which eventually is `${\sigma}\sin\left({\frac{\|\bmz\|_2}{\sigma}}\right)$' exist in the RKHS $H_{\sigma,1,\Cn}$ and therefore its norm is finite. In other words, we need to check whether the function ${\sigma}\sin\left({\frac{\|\bmz\|_2}{\sigma}}\right)$ is $L^2-$integrable with respect to the normalized Laplacian measure $d\mu_{\sigma,1,\Cn}(\bmz)$, and to that end, we have
\begin{align*}
    \int_{\Cn}\left|{\sigma}\sin\left({\frac{\|\bmz\|_2}{\sigma}}\right)\right|^2d\mu_{\sigma,1,\Cn}(\bmz)\leq\sigma^2\int_{\Cn}d\mu_{\sigma,1,\Cn}=\sigma^2<\infty,
\end{align*}
as the Laplacian measure in $d\mu_{\sigma,1,\Cn}$ is already given in normalized form. Since, we already have $\sigma<\infty$, therefore, the function defined by ${\|\bmz\|_2}K^{\sigma}(_-\mathfrak{Z}_{+})$ exists in the RKHS $H_{\sigma,1,\Cn}$ and their norm is bounded by $\sigma$. After this, we define a sequence of points on the subspace $_-\mathfrak{Z}_{+,N}$ of $_-\mathfrak{Z}_{+}$ as:
\begin{align}\label{eq_79}
    _-\mathfrak{Z}_{+,N}\coloneqq\left\{\left(\bmz_N,-\bmz_N\right)\in\Cn\times\Cn:\lim_{N\to\infty}\|\bmz_N\|_2=0~\text{where}~\bmz\in\Cn\right\}.
\end{align}
Apparently, it is easy to see that $_-\mathfrak{Z}_{+,N}\subset{{\!}_-}\mathfrak{Z}_{+}$ because $_-\mathfrak{Z}_{+,N}$ contains those sequence of coordinates points from $_-\mathfrak{Z}_{+}$ whose magnitude gets negligibly small as $N$ tends to $\infty.$ Observe that for coordinate points in $_-\mathfrak{Z}_{+,N}$, following relation holds:
\begin{align*}
    {\|\bmz_N\|_2}K^{\sigma}(_-\mathfrak{Z}_{+,N})=&{\sigma}\sin\left({\frac{\|\bmz_N\|_2}{\sigma}}\right),\\
    \implies\lim_{N\to\infty}{\|\bmz_N\|_2}K^{\sigma}(_-\mathfrak{Z}_{+,N})=&\lim_{N\to\infty}\left\{{\sigma}\sin\left({\frac{\|\bmz_N\|_2}{\sigma}}\right)\right\}\\
    \lim_{N\to\infty}{\|\bmz_N\|_2}K^{\sigma}(_-\mathfrak{Z}_{+,N})=&\sigma\left\{\lim_{N\to\infty}\sin\left({\frac{\|\bmz_N\|_2}{\sigma}}\right)\right\}\\
    \lim_{N\to\infty}{\|\bmz_N\|_2}K^{\sigma}(_-\mathfrak{Z}_{+,N})=&\sigma\cdot0\\
    \lim_{N\to\infty}{\|\bmz_N\|_2}K^{\sigma}(_-\mathfrak{Z}_{+,N})=&0.\numberthis\label{eq_80}
\end{align*}
Thus, the above manipulation makes us learn that we do have indeed a sequence of functions in terms of the reproducing kernel $K^\sigma(\cdot,\cdot)$ which approaches to $0$ as $N$ tends to infinity. If we denote this sequence by $\mathfrak{K}_N$ that is,  $\mathfrak{K}_N\coloneqq{\|\bmz_N\|_2}K^{\sigma}(_-\mathfrak{Z}_{+,N})$, then we have clearly
\begin{align}\label{eq_81r}
    \lim_{N\to\infty}\mathfrak{K}_N=\lim_{N\to\infty}{\|\bmz_N\|_2}K^{\sigma}(_-\mathfrak{Z}_{+,N})=0.
\end{align}

Let $\varphi:\Cn\to\Cn$ be a holomorphic function in which every coordinate function of $\varphi$ is holomorphic from $\Cn\to\mathbb{C}$. Considering the bounded Koopman operator $\koop$ induced by this $\varphi$ and hence, we see that $\varphi$ takes the affine structure as $\varphi(\bmz)=\mathcal{A}\bmz+B$ by \autoref{theorem_boundedKoopmanoverRKHS} where $\mathcal{A}\in\mathbb{C}^{n\times n}$ with $0<\|\mathcal{A}\|\leq1$ and $B$ is a $n-$dimensional complex vector. Suppose $B\equiv0$ in the aforementioned affine structure of $\varphi$ and hence we have $\varphi(\bmz)\equiv\mathcal{A}\bmz$ which is a pure linear structure over $\Cn$. Let this linear structure is denoted by $\varphi_{\mathcal{A}}(\bmz)\coloneqq\mathcal{A}\bmz$, then the corresponding Koopman operator $\mathcal{K}_{\varphi_\mathcal{A}}$ acting on sequence of function $\mathfrak{K}_N$ yields: 
\begin{align*}
    \mathcal{K}_{\varphi_\mathcal{A}}\mathfrak{K}_N=&{\sigma}\sin\left({\frac{\|\mathcal{A}\bmz_N\|_2}{\sigma}}\right)\\
    \implies\lim_{N\to\infty}\mathcal{K}_{\varphi_\mathcal{A}}\mathfrak{K}_N=&\lim_{N\to\infty}{\sigma}\sin\left({\frac{\|\mathcal{A}\bmz_N\|_2}{\sigma}}\right)\\
    \lim_{N\to\infty}\mathcal{K}_{\varphi_\mathcal{A}}\mathfrak{K}_N=&\sigma\left\{\lim_{N\to\infty}\sin\left({\frac{\|\mathcal{A}\bmz_N\|_2}{\sigma}}\right)\right\}\\
    \lim_{N\to\infty}\mathcal{K}_{\varphi_\mathcal{A}}\mathfrak{K}_N=&\sigma\cdot0\\
    \lim_{N\to\infty}\mathcal{K}_{\varphi_\mathcal{A}}\mathfrak{K}_N=&0.\numberthis\label{eq_82}
\end{align*}
Note that if $B\not\equiv0$ in our preceding assumption then we fail to achieve this convergence. So, in conclusion, we have determined a sequence of function inside the RKHS $H_{\sigma,1,\Cn}$ which is able to satisfy the conditions given in \autoref{lemma_closable_characterization} and based upon that, with additional assumptions on the boundedness of the Koopman operator $\koop$ acting over the RKHS $H_{\sigma,1,\Cn}$, we can prove that the Koopman operator is closable over the RKHS $H_{\sigma,1,\Cn}$. Explicitly, this particular sequence of function is given in \eqref{eq_81r} and the action of bounded Koopman operator $\mathcal{K}_{\varphi_{\mathcal{A}}}$ where $\varphi_\mathcal{A}(\bmz)=\mathcal{A}\bmz$ on this particular sequence of function is given in \eqref{eq_82}. 

With these deep function theoretic discussion, we are now ready to provide the complete closable characterization for the bounded Koopman operator when the underlying RKHS is generated by the Laplacian measure. The following theorem serves exactly this purpose.
\begin{theorem}
    Let $\sigma$ be positive and finite. Consider $\varphi:\Cn\to\Cn$ as a holomorphic function in which every coordinate function of $\varphi$ are holomorphic from $\Cn\to\mathbb{C}$, which induces the bounded Koopman operator $\koop:\mathcal{D}(\koop)\to H_{\sigma,1,\Cn}$ such that $\varphi(\bmz)=\mathcal{A}\bmz+B$ over the RKHS $H_{\sigma,1,\Cn}$ whose reproducing kernel is $K^{\sigma}(\cdot,\cdot)$. Then, there exists a sequence of function $\mathfrak{K}_N$ inside the RKHS $H_{\sigma,1,\Cn}$ as defined in \eqref{eq_81r} such that $\mathfrak{K}_N\to0$ as $N\to\infty$ and the Koopman operator $\koop$ is closable over the RKHS $H_{\sigma,1,\Cn}$ if $B\equiv0$, implying that $\varphi$ is purely linear in $\Cn$, that is $\varphi(\bmz)=\mathcal{A}\bmz$. Further, if we let $\varphi_{\mathcal{A}}(\bmz)\coloneqq\mathcal{A}\bmz$, then $\mathcal{K}_{\varphi_{\mathcal{A}}}\mathfrak{K}_N\to0$ as $N\to\infty$.
\end{theorem}
\subsection{Failure of closability of Koopman operator with GRBF}
Now, that we have successfully demonstrated the closability of the Koopman operators over the RKHS $H_{\sigma,1,\Cn}$, we will analyse now that how the Koopman operator fails to be closable with respect to the Gaussian Radial Basis Kernel Function whose norm for the function space is given in \eqref{eq_6GRBFfunctionspace}. The main reason why we fail to achieve the closability of the Koopman operators when it interacts with Gaussian Radial Basis Function Kernel is because the of the way we have defined the inner-product for the function space corresponding to the Gaussian Radial Basis Function Kernel (cf. \cite{steinwart2006explicit,steinwart2008support}). In particular, the measure present in the norm for the function space in \eqref{eq_6GRBFfunctionspace} is unable to make the $L^2-$integration finite. This can be easily understood as follows if we chose to follow to construct the same variant of sequence of functions that we constructed previously. To that recall the domain defined in \eqref{eq_78} and $K_{\text{exp}}^{2,\sigma}(\bm{x},\bm{z})$ as the Gaussian Radial Basis Function Kernel and then we have following:
\begin{align*}
    K_{\text{exp}}^{2,\sigma}(_-\mathfrak{Z}_+)=\exp\left(-\frac{\|\bmz+\bmz\|_2^2}{\sigma}\right)=\exp\left(-\frac{4\|\bmz\|_2^2}{\sigma}\right).
\end{align*}
With the subspace given in \eqref{eq_79}, we see that 
\begin{align*}
    \|\bmz_N\|_2K_{\text{exp}}^{2,\sigma}(_-\mathfrak{Z}_{+,N})=&\|\bmz_N\|_2\exp\left(-\frac{4\|\bmz\|_2^2}{\sigma}\right)\\
    \implies\lim_{N\to0}\|\bmz_N\|_2K_{\text{exp}}^{2,\sigma}(_-\mathfrak{Z}_{+,N})=&\lim_{N\to0}\left\{\|\bmz_N\|_2\exp\left(-\frac{4\|\bmz_N\|_2^2}{\sigma}\right)\right\}\\
    =&0\cdot1\\
    =&0.
\end{align*}
Obviously, the above sequence do indeed converges to $0$. However, the sequence of function $\|\bmz_N\|_2K_{\text{exp}}^{2,\sigma}(_-\mathfrak{Z}_{+,N})=\|\bmz_N\|_2\exp\left(-\frac{4\|\bmz\|_2^2}{\sigma}\right)$ fails to exist in the Hilbert function space defined in \eqref{eq_7GRBFHilbertspace}. This can be learned easily from the following verification:
\begin{align*}
    \left\|\|\bmz_N\|_2\exp\left(-\frac{4\|\bmz\|_2^2}{\sigma}\right)\right\|_{\sigma}^2=&\frac{2^n\sigma^{2n}}{\pi^n}\|\bmz_N\|_2^2\int_{\Cn}\exp\left(-\frac{4\|\bmz\|_2^2}{\sigma}\right)e^{\sigma^2\sum_{i=1}^n{\left(z_i-\overline{z}_i\right)}^2}dV(\bmz)\\
    \leq&\frac{2^n\sigma^{2n}}{\pi^n}\|\bmz_N\|_2^2\int_{\Cn}e^{\sigma^2\sum_{i=1}^n{\left(z_i-\overline{z}_i\right)}^2}dV(\bmz)\\
    \not<&\infty.
\end{align*}
Since, the above sequence of function fails to exist in the RKHS $H_\sigma$ as given in \eqref{eq_7GRBFHilbertspace}, therefore we fail to construct a sequence by which we can show the closability of the Koopman operators over the RKHS $H_\sigma$. 
\section{Conclusion}
In this paper, we explicitly refer `$d\mu_{\sigma,1,\Cn}$' as the normalized Laplacian measure defined over the entire $\Cn$. The pursuit of having this unprecedented way of study for the normalized Laplacian measure is duly motivated by taking the gauge transformation of the Laplacian Kernel Function. \emph{As already mentioned, the underlying concern for this study is we consider only limited number of snapshots or data vector for fluid flow across cylinder experiment to generate the desired Koopman modes via the method of the Kernelized Extended Dynamic Mode Decomposition.} For the data driven discovery for this particular experiment, both Laplacian Kernel Function and Gaussian Radial Basis Function Kernel were employed. However, based upon the empirical evidence presented, we learned that it was \emph{only} the Laplacian Kernel Function which was able to discover the desired dominant Koopman modes from the limited availability of snapshots and this manuscripts serves the purpose of reporting these exciting experimental insights. The other great deal of insights present in this manuscript is the amalgamation of the random matrix theory topics to execute the Kernelized Extended Dynamic Mode Decomposition.

Lastly, apart from the routine operator theoretic characterizations of the Koopman operators or the composition operators along with other of its variants such as the \emph{weighted composition operators}, in this paper we consider the \emph{closable} nuance of the Koopman operators as well. Once, we have derived the regular and basic operator theoretic properties for the corresponding function space of the RKHS $H_{\sigma,1,\Cn}$ such as \emph{boundedness, essential norm estimates} and \emph{compactness}, we further indeed, proved that the Koopman operators are closable over this RKHS. It should be noted that Koopman operators have been studied in great details by various authors and mathematicians in \cite{ikeda2022boundedness,ikeda2022koopman,singh1993composition,shapiro1987essential,hai2016boundedness,hai2018complex,hai2021weighted,zhao2015invertible} in various settings of function spaces including those which corresponds to function space containing entire functions of exponential type \cite{chacon2007composition}. In particular, as observed in \cite{chacon2007composition}, Koopman operators fails to be compact in such spaces. However, we fail to encounter relevant discussion on the closability for the Koopman operators in the previously cited manuscripts, but this manuscript makes a successful attempts to enlighten us into these directions over the newly constructed Hilbert space along with an extra edge of reproducing kernel theory.
\section{Acknowledgement}
This manuscript is directed towards the scope of performing scientific machine learning and data acquisition by borrowing relevant and contemporary mathematical strategies that exists. Hence, the present research work is scripted in the {\tt{julia}} programming language coded over the {\tt{Google Collab}} online platform. The validity of this contribution is immediately justified when we observe that is little to nothing research contribution in the light of this coding language. However, we do see almost every research contribution of the present research agenda either on {\tt{Python}} or {\tt{Matlab}}. The author acknowledges the {\tt{Matlab}} coding support from his PhD advisor \textsc{Dr. Joel A. Rosenfeld} which help him making the corresponding {\tt{julia}} codes for fluid flow across cylinder experiment. The Koopman operator theory analysis over the RKHS presented in this paper was a smooth research journey and the author would like to thanks his past various successful mathematician collaborators with whom he learnt all these operator theory and functional analysis topics in great details.


\begin{thebibliography}{100}

\bibitem{alexander2020operator}
{\sc Alexander, R., and Giannakis, D.}
\newblock {Operator-theoretic framework for forecasting nonlinear time series with kernel analog techniques}.
\newblock {\em Physica D: Nonlinear Phenomena 409\/} (2020), 132520.

\bibitem{aronszajn1950theory}
{\sc Aronszajn, N.}
\newblock {Theory of reproducing kernels}.
\newblock {\em {Transactions of the American Mathematical Society} 68}, 3 (1950), 337--404.

\bibitem{baddoo2022kernel}
{\sc Baddoo, P.~J., Herrmann, B., McKeon, B.~J., and Brunton, S.~L.}
\newblock {Kernel learning for robust Dynamic Mode Decomposition: Linear and Nonlinear disambiguation optimization}.
\newblock {\em Proceedings of the Royal Society A 478}, 2260 (2022), 20210830.

\bibitem{bagheri2013koopman}
{\sc Bagheri, S.}
\newblock {Koopman-mode decomposition of the cylinder wake}.
\newblock {\em Journal of Fluid Mechanics 726\/} (2013), 596--623.

\bibitem{bayart2010parabolic}
{\sc Bayart, F.}
\newblock {Parabolic composition operators on the ball}.
\newblock {\em Advances in Mathematics 223}, 5 (2010), 1666--1705.

\bibitem{bayart2011composition}
{\sc Bayart, F.}
\newblock {Composition operators on the polydisk induced by affine maps}.
\newblock {\em Journal of Functional Analysis 260}, 7 (2011), 1969--2003.

\bibitem{belkin2018understand}
{\sc Belkin, M., Ma, S., and Mandal, S.}
\newblock {To understand deep learning we need to understand kernel learning}.
\newblock In {\em International Conference on Machine Learning\/} (2018), PMLR, pp.~541--549.

\bibitem{berezin1975general}
{\sc Berezin, F.~A.}
\newblock {General Concept of Quantization}.
\newblock {\em Communications in Mathematical Physics 40\/} (1975), 153--174.

\bibitem{boyd2004convex}
{\sc Boyd, S.~P., and Vandenberghe, L.}
\newblock {\em Convex optimization}.
\newblock Cambridge University Press, 2004.

\bibitem{brunton2022data}
{\sc Brunton, S.~L., and Kutz, J.~N.}
\newblock {\em {Data-driven science and engineering: Machine learning, dynamical systems, and control}}.
\newblock Cambridge University Press, 2022.

\bibitem{burov2021kernel}
{\sc Burov, D., Giannakis, D., Manohar, K., and Stuart, A.}
\newblock {Kernel analog forecasting: Multiscale test problems}.
\newblock {\em Multiscale Modeling \& Simulation 19}, 2 (2021), 1011--1040.

\bibitem{carswell2003composition}
{\sc Carswell, B., MacCluer, B.~D., and Schuster, A.}
\newblock {Composition operators on the Fock space}.
\newblock {\em Acta Sci. Math.(Szeged) 69}, 3-4 (2003), 871--887.

\bibitem{caselle2004random}
{\sc Caselle, M., and Magnea, U.}
\newblock {Random matrix theory and symmetric spaces}.
\newblock {\em Physics reports 394}, 2-3 (2004), 41--156.

\bibitem{chacon2007composition}
{\sc Chac{\'o}n, G., and Gim{\'e}nez, J.}
\newblock {Composition operators on spaces of Entire functions}.
\newblock {\em Proceedings of the American Mathematical Society 135}, 7 (2007), 2205--2218.

\bibitem{chen2020deep}
{\sc Chen, L., and Xu, S.}
\newblock {Deep Neural Tangent Kernel and Laplace Kernel have the same RKHS}.
\newblock {\em arXiv preprint arXiv:2009.10683\/} (2020).

\bibitem{colbrook2023multiverse}
{\sc Colbrook, M.~J.}
\newblock {The Multiverse of Dynamic Mode Decomposition Algorithms}.
\newblock {\em arXiv preprint arXiv:2312.00137\/} (2023).

\bibitem{colbrook2024rigorous}
{\sc Colbrook, M.~J., and Townsend, A.}
\newblock {Rigorous data-driven computation of spectral properties of Koopman operators for dynamical systems}.
\newblock {\em Communications on Pure and Applied Mathematics 77}, 1 (2024), 221--283.

\bibitem{conway2019course}
{\sc Conway, J.~B.}
\newblock {\em A Course in Functional Analysis}, vol.~96.
\newblock Springer, 2019.

\bibitem{cook2011mean}
{\sc Cook, R.~D., and Forzani, L.}
\newblock {On the mean and variance of the generalized inverse of a singular Wishart matrix}.
\newblock {\em Electronic Journal of Statistics 5\/} (2011).

\bibitem{cowen1983composition}
{\sc Cowen, C.~C.}
\newblock {Composition operators on $H^2$}.
\newblock {\em Journal of Operator Theory\/} (1983), 77--106.

\bibitem{cowen2019composition}
{\sc Cowen~Jr, C.~C.}
\newblock {\em {Composition operators on spaces of analytic functions}}.
\newblock Routledge, 2019.

\bibitem{das2020koopman}
{\sc Das, S., and Giannakis, D.}
\newblock {Koopman spectra in reproducing kernel Hilbert spaces}.
\newblock {\em Applied and Computational Harmonic Analysis 49}, 2 (2020), 573--607.

\bibitem{das2021reproducing}
{\sc Das, S., Giannakis, D., and Slawinska, J.}
\newblock {Reproducing kernel Hilbert space compactification of unitary evolution groups}.
\newblock {\em Applied and Computational Harmonic Analysis 54\/} (2021), 75--136.

\bibitem{diaz2006distribution}
{\sc D{\'\i}az-Garc{\'\i}a, J.~A., and Guti{\'e}rrez-J{\'a}imez, R.}
\newblock {Distribution of the generalised inverse of a random matrix and its applications}.
\newblock {\em {Journal of Statistical Planning and Inference} 136}, 1 (2006), 183--192.

\bibitem{doan2017composition}
{\sc Doan, M.~L., Khoi, L.~H., and Le, T.}
\newblock {Composition operators on Hilbert spaces of entire functions of several variables}.
\newblock {\em Integral Equations and Operator Theory 88\/} (2017), 301--330.

\bibitem{edelman2005random}
{\sc Edelman, A., and Rao, N.~R.}
\newblock Random matrix theory.
\newblock {\em Acta numerica 14\/} (2005), 233--297.

\bibitem{fasshauer2007meshfree}
{\sc Fasshauer, G.~E.}
\newblock {\em {Meshfree approximation methods with MATLAB}}, vol.~6.
\newblock World Scientific, 2007.

\bibitem{fujii2019dynamic}
{\sc Fujii, K., and Kawahara, Y.}
\newblock {Dynamic Mode Decomposition in vector-valued reproducing kernel Hilbert spaces for extracting dynamical structure among observables}.
\newblock {\em Neural Networks 117\/} (2019), 94--103.

\bibitem{garnett2007bounded}
{\sc Garnett, J.}
\newblock {\em Bounded {A}nalytic {F}unctions}, vol.~236.
\newblock Springer Science \& Business Media, 2007.

\bibitem{geifman2020similarity}
{\sc Geifman, A., Yadav, A., Kasten, Y., Galun, M., Jacobs, D., and Ronen, B.}
\newblock {On the similarity between the Laplace and Neural Tangent Kernels}.
\newblock {\em Advances in Neural Information Processing Systems 33\/} (2020), 1451--1461.

\bibitem{NEURIPS2020_1006ff12}
{\sc Geifman, A., Yadav, A., Kasten, Y., Galun, M., Jacobs, D., and Ronen, B.}
\newblock {On the Similarity between the Laplace and Neural Tangent Kernels}.
\newblock In {\em Advances in Neural Information Processing Systems\/} (2020), H.~Larochelle, M.~Ranzato, R.~Hadsell, M.~Balcan, and H.~Lin, Eds., vol.~33, Curran Associates, Inc., pp.~1451--1461.

\bibitem{giannakis2020extraction}
{\sc Giannakis, D., and Das, S.}
\newblock {Extraction and prediction of coherent patterns in incompressible flows through space--time Koopman analysis}.
\newblock {\em Physica D: Nonlinear Phenomena 402\/} (2020), 132211.

\bibitem{giraud2005positive}
{\sc Giraud, B., and Peschanski, R.}
\newblock {On positive functions with positive Fourier transforms}.
\newblock {\em Acta Physica Polonica B 37\/} (2006), 331.

\bibitem{gonzalez2023modeling}
{\sc Gonzalez, E., Avazpour, L., Kamalapurkar, R., and Rosenfeld, J.~A.}
\newblock {Modeling Partially Unknown Dynamics with Continuous Time DMD}.
\newblock In {\em 2023 American Control Conference (ACC)\/} (2023), IEEE, pp.~2913--2918.

\bibitem{hai2016boundedness}
{\sc Hai, P.~V., et~al.}
\newblock {Boundedness and Compactness of Weighted Composition Operators on Fock Spaces ${F}^p(\mathbb{C})$}.
\newblock {\em Acta Mathematica Vietnamica 41}, 3 (2016), 531--537.

\bibitem{hai2018complex}
{\sc Hai, P.~V., and Khoi, L.~H.}
\newblock Complex symmetric weighted composition operators on the {F}ock space in several variables.
\newblock {\em Complex Variables and Elliptic Equations 63}, 3 (2018), 391--405.

\bibitem{hai2021weighted}
{\sc Hai, P.~V., and Rosenfeld, J.~A.}
\newblock Weighted {C}omposition {O}perators on the {M}ittag-{L}effler {S}paces of {E}ntire {F}unctions.
\newblock {\em Complex Analysis and Operator Theory 15}, 1 (2021), 1--26.

\bibitem{hall2013quantum}
{\sc Hall, B.~C.}
\newblock {\em {Quantum theory for Mathematicians}}.
\newblock Springer, 2013.

\bibitem{halmos2012hilbert}
{\sc Halmos, P.~R.}
\newblock {\em {A Hilbert space problem book}}, vol.~19.
\newblock Springer Science \& Business Media, 2012.

\bibitem{hitczenko1994rademacher}
{\sc Hitczenko, P., and Kwapien, S.}
\newblock {On the Rademacher series, Probability in Banach Spaces, Nine, Sandbjerg, Denmark, J. Hoffmann--J{\o}rgensen, J. Kuelbs, MB Marcus, Eds}.
\newblock {\em Birkhauser, Boston 31\/} (1994), 36.

\bibitem{hui2018kernel}
{\sc Hui, L., Ma, S., and Belkin, M.}
\newblock Kernel machines beat deep neural networks on mask-based single-channel speech enhancement.
\newblock {\em Proc. Interspeech 2019\/} (2019), 2748--–2752.

\bibitem{ikeda2022boundedness}
{\sc Ikeda, M., Ishikawa, I., and Sawano, Y.}
\newblock {Boundedness of composition operators on reproducing kernel Hilbert spaces with analytic positive definite functions}.
\newblock {\em Journal of Mathematical Analysis and Applications 511}, 1 (2022), 126048.

\bibitem{ikeda2022koopman}
{\sc Ikeda, M., Ishikawa, I., and Schlosser, C.}
\newblock {Koopman and Perron--Frobenius operators on reproducing kernel Banach spaces}.
\newblock {\em Chaos: An Interdisciplinary Journal of Nonlinear Science 32}, 12 (2022).

\bibitem{ivanov2002random}
{\sc Ivanov, D.~A.}
\newblock {Random-matrix ensembles in p-wave vortices}.
\newblock In {\em Vortices in unconventional superconductors and superfluids}. Springer, 2002, pp.~253--265.

\bibitem{jacot2018neural}
{\sc Jacot, A., Gabriel, F., and Hongler, C.}
\newblock {Neural tangent kernel: Convergence and generalization in neural networks}.
\newblock {\em {Advances in Neural Information Processing Systems} 31\/} (2018).

\bibitem{janson1987hankel}
{\sc Janson, S., Peetre, J., and Rochberg, R.}
\newblock {Hankel forms and the Fock space}.
\newblock {\em Revista Matematica Iberoamericana 3}, 1 (1987), 61--138.

\bibitem{kamalapurkar2021occupation}
{\sc Kamalapurkar, R., and Rosenfeld, J.~A.}
\newblock {An Occupation Kernel approach to Optimal Control}.
\newblock {\em arXiv preprint arXiv:2106.00663\/} (2021).

\bibitem{kevrekidis2016kernel}
{\sc Kevrekidis, I., Rowley, C.~W., and Williams, M.}
\newblock {A kernel-based method for data-driven Koopman spectral analysis}.
\newblock {\em Journal of Computational Dynamics 2}, 2 (2016), 247--265.

\bibitem{khosravi2023representer}
{\sc Khosravi, M.}
\newblock {Representer theorem for learning Koopman operators}.
\newblock {\em IEEE Transactions on Automatic Control\/} (2023).

\bibitem{klus2018kernel}
{\sc Klus, S., Bittracher, A., Schuster, I., and Sch{\"u}tte, C.}
\newblock {A kernel-based approach to molecular conformation analysis}.
\newblock {\em The Journal of Chemical Physics 149}, 24 (2018).

\bibitem{klus2020kernel}
{\sc Klus, S., N{\"u}ske, F., and Hamzi, B.}
\newblock {Kernel-based approximation of the Koopman generator and Schr{\"o}dinger operator}.
\newblock {\em Entropy 22}, 7 (2020), 722.

\bibitem{koopman1931hamiltonian}
{\sc Koopman, B.~O.}
\newblock {Hamiltonian systems and transformation in Hilbert space}.
\newblock {\em Proceedings of the National Academy of Sciences 17}, 5 (1931), 315--318.

\bibitem{korda2018convergence}
{\sc Korda, M., and Mezi{\'c}, I.}
\newblock {On convergence of Extended Dynamic Mode Decomposition to the Koopman operator}.
\newblock {\em Journal of Nonlinear Science 28\/} (2018), 687--710.

\bibitem{kutz2016dynamic}
{\sc Kutz, J.~N., Brunton, S.~L., Brunton, B.~W., and Proctor, J.~L.}
\newblock {\em {Dynamic Mode Decomposition: Data-Driven Modeling of Complex Dystems}}.
\newblock SIAM, 2016.

\bibitem{kutz2016multiresolution}
{\sc Kutz, J.~N., Fu, X., and Brunton, S.~L.}
\newblock {Multiresolution Dynamic Mode Decomposition}.
\newblock {\em SIAM Journal on Applied Dynamical Systems 15}, 2 (2016), 713--735.

\bibitem{le2014normal}
{\sc Le, T.}
\newblock {Normal and isometric weighted composition operators on the Fock space}.
\newblock {\em Bulletin of the London Mathematical Society 46}, 4 (2014), 847--856.

\bibitem{le2017composition}
{\sc Le, T.}
\newblock {Composition operators between Segal--Bargmann spaces}.
\newblock {\em Journal of Operator Theory 78}, 1 (2017), 135--158.

\bibitem{lorenz1963deterministic}
{\sc {Lorenz, Edward N}}.
\newblock {Deterministic nonperiodic flow}.
\newblock {\em Journal of atmospheric sciences 20}, 2 (1963), 130--141.

\bibitem{mehta2004random}
{\sc Mehta, M.~L.}
\newblock {\em Random matrices}.
\newblock Elsevier, 2004.

\bibitem{mezic2005spectral}
{\sc Mezi{\'c}, I.}
\newblock {Spectral properties of dynamical systems, model reduction and decompositions}.
\newblock {\em Nonlinear Dynamics 41\/} (2005), 309--325.

\bibitem{mezic2021koopman}
{\sc Mezi{\'c}, I.}
\newblock {Koopman operator, geometry, and learning of dynamical systems}.
\newblock {\em Not. Am. Math. Soc. 68}, 7 (2021), 1087--1105.

\bibitem{mezic2004comparison}
{\sc Mezi{\'c}, I., and Banaszuk, A.}
\newblock {Comparison of systems with complex behavior}.
\newblock {\em Physica D: Nonlinear Phenomena 197}, 1-2 (2004), 101--133.

\bibitem{morrison2023dynamic}
{\sc Morrison, Z., Abudia, M., Rosenfeld, J., and Kamalapurar, R.}
\newblock {Dynamic Mode Decomposition of Control-Affine Nonlinear Systems using Discrete Control Liouville Operators}.
\newblock {\em arXiv preprint arXiv:2309.09817\/} (2023).

\bibitem{pedersen2012analysis}
{\sc Pedersen, G.~K.}
\newblock {\em {Analysis now}}, vol.~118.
\newblock Springer Science \& Business Media, 2012.

\bibitem{peetre1990berezin}
{\sc Peetre, J.}
\newblock {The Berezin transform and Ha-plitz operators}.
\newblock {\em Journal of Operator Theory\/} (1990), 165--186.

\bibitem{philipp2023error}
{\sc Philipp, F., Schaller, M., Worthmann, K., Peitz, S., and N{\"u}ske, F.}
\newblock {Error bounds for kernel-based approximations of the Koopman operator}.
\newblock {\em arXiv preprint arXiv:2301.08637\/} (2023).

\bibitem{proctor2016dynamic}
{\sc Proctor, J.~L., Brunton, S.~L., and Kutz, J.~N.}
\newblock {Dynamic Mode Decomposition with Control}.
\newblock {\em SIAM Journal on Applied Dynamical Systems 15}, 1 (2016), 142--161.

\bibitem{rasmussen2006gaussian}
{\sc {Rasmussen, Carl Edward and Williams, Christopher KI and others}}.
\newblock {\em {Gaussian Processes for Machine Learning}}, vol.~1.
\newblock Springer, 2006.

\bibitem{reed2012methods}
{\sc Reed, M.}
\newblock {\em {\textsc{Methods of modern mathematical physics}: Functional Analysis}}.
\newblock Elsevier, 2012.

\bibitem{reedmethods}
{\sc Reed, M., and Simon, B.}
\newblock {\em {\textsc{Methods of modern mathematical physics}. 1: Functional Analysis}}.
\newblock Academic Press, April 1980.

\bibitem{rosenfeld2023singular}
{\sc Rosenfeld, J.~A., and Kamalapurkar, R.}
\newblock {Singular Dynamic Mode Decomposition}.
\newblock {\em SIAM Journal on Applied Dynamical Systems 22}, 3 (2023), 2357--2381.

\bibitem{rosenfeld2021occupationACC}
{\sc Rosenfeld, J.~A., Kamalapurkar, R., Gruss, L.~F., and Johnson, T.~T.}
\newblock {On Occupation Kernels, Liouville operators, and Dynamic Mode Decomposition}.
\newblock In {\em 2021 American Control Conference (ACC)\/} (2021), IEEE, pp.~3957--3962.

\bibitem{rosenfeld2022dynamic}
{\sc Rosenfeld, J.~A., Kamalapurkar, R., Gruss, L.~F., and Johnson, T.~T.}
\newblock {Dynamic Mode Decomposition for Continuous Time Systems with the Liouville operator}.
\newblock {\em Journal of Nonlinear Science 32\/} (2022), 1--30.

\bibitem{rosenfeld2019occupation}
{\sc Rosenfeld, J.~A., Kamalapurkar, R., Russo, B., and Johnson, T.~T.}
\newblock {Occupation kernels and densely defined Liouville operators for system identification}.
\newblock In {\em {2019 IEEE 58th Conference on Decision and Control (CDC)}\/} (2019), IEEE, pp.~6455--6460.

\bibitem{rosenfeld2018mittag}
{\sc Rosenfeld, J.~A., Russo, B., and Dixon, W.~E.}
\newblock {The Mittag Leffler reproducing kernel Hilbert spaces of entire and analytic functions}.
\newblock {\em Journal of Mathematical Analysis and Applications 463}, 2 (2018), 576--592.

\bibitem{rosenfeld2021theoretical}
{\sc Rosenfeld, J.~A., Russo, B.~P., and Kamalapurkar, R.}
\newblock {Theoretical Foundations for the Dynamic Mode Decomposition of High Order Dynamical Systems}.
\newblock {\em arXiv preprint arXiv:2101.02646\/} (2021).

\bibitem{rowley2009spectral}
{\sc Rowley, C.~W., Mezi{\'c}, I., Bagheri, S., Schlatter, P., and Henningson, D.~S.}
\newblock {Spectral analysis of nonlinear flows}.
\newblock {\em Journal of fluid mechanics 641\/} (2009), 115--127.

\bibitem{rudin1987real}
{\sc Rudin, W.}
\newblock Real and complex analysis.
\newblock {\em (Mcgraw-Hill International Editions: Mathematics series)\/} (1987).

\bibitem{rudin1991functional}
{\sc Rudin, W.}
\newblock {Functional Analysis 2nd ed}.
\newblock {\em International Series in Pure and Applied Mathematics. McGraw-Hill, Inc., New York\/} (1991).

\bibitem{russo2022liouville}
{\sc Russo, B.~P., and Rosenfeld, J.~A.}
\newblock {Liouville operators over the Hardy space}.
\newblock {\em Journal of Mathematical Analysis and Applications 508}, 2 (2022), 125854.

\bibitem{saitoh1983hilbert}
{\sc Saitoh, S.}
\newblock {Hilbert spaces induced by Hilbert space valued functions}.
\newblock {\em Proceedings of the American Mathematical Society 89}, 1 (1983), 74--78.

\bibitem{saitoh1997integral}
{\sc Saitoh, S.}
\newblock {\em {Integral transforms, reproducing kernels and their applications}}, vol.~369.
\newblock CRC Press, 1997.

\bibitem{sarker2021deep}
{\sc Sarker, I.~H.}
\newblock {Deep learning: a comprehensive overview on techniques, taxonomy, applications and research directions}.
\newblock {\em SN Computer Science 2}, 6 (2021), 420.

\bibitem{scheffe1947useful}
{\sc Scheff{\'e}, H.}
\newblock {A useful convergence theorem for probability distributions}.
\newblock {\em The Annals of Mathematical Statistics 18}, 3 (1947), 434--438.

\bibitem{schmid2010dynamic}
{\sc Schmid, P.~J.}
\newblock {Dynamic Mode Decomposition of Numerical and Experimental Data}.
\newblock {\em Journal of fluid mechanics 656\/} (2010), 5--28.

\bibitem{schmid2011application}
{\sc Schmid, P.~J.}
\newblock {Application of the Dynamic Mode Decomposition to experimental data}.
\newblock {\em Experiments in fluids 50\/} (2011), 1123--1130.

\bibitem{schmid2022dynamic}
{\sc Schmid, P.~J.}
\newblock {Dynamic Mode Decomposition and its variants}.
\newblock {\em Annual Review of Fluid Mechanics 54\/} (2022), 225--254.

\bibitem{scholkopf2000kernel}
{\sc Sch{\"o}lkopf, B.}
\newblock {The kernel trick for distances}.
\newblock {\em Advances in Neural Information Processing Systems 13\/} (2000).

\bibitem{schroder1870ueber}
{\sc Schr{\"o}der, E.}
\newblock {Ueber iterirte functionen}.
\newblock {\em Mathematische Annalen 3}, 2 (1870), 296--322.

\bibitem{shapiro1987essential}
{\sc Shapiro, J.~H.}
\newblock The essential norm of a composition operator.
\newblock {\em Annals of mathematics\/} (1987), 375--404.

\bibitem{shapiro2012composition}
{\sc Shapiro, J.~H.}
\newblock {\em {Composition operators: And Classical Function Theory}}.
\newblock Springer Science \& Business Media, 2012.

\bibitem{singh2023new}
{\sc Singh, H.}
\newblock {A new kernel function for better AI methods}.
\newblock In {\em 2023 AMS Spring Eastern Sectional Meeting\/} (2023), no.~{68-Computer Science, 68T-Artificial Intelligence and 68T07-Artificial Neural Networks and Deep Learning} in \href{https://meetings.ams.org/math/spring2023e/meetingapp.cgi/Paper/23517}{April 1, 2023}, AMS.

\bibitem{singh2023appointment}
{\sc Singh, H.}
\newblock \href{https://arxiv.org/abs/2312.10693}{An appointment with Reproducing Kernel Hilbert Space generated by Generalized Gaussian RBF as $L^2-$measure}, 2023.

\bibitem{singh1993composition}
{\sc Singh, R.~K., and Manhas, J.~S.}
\newblock {\em {Composition operators on function spaces}}.
\newblock Elsevier, 1993.

\bibitem{stein2010complex}
{\sc Stein, E.~M., and Shakarchi, R.}
\newblock {\em {Complex analysis}}, vol.~2.
\newblock Princeton University Press, 2010.

\bibitem{steinwart2008support}
{\sc Steinwart, I., and Christmann, A.}
\newblock {\em {Support vector machines}}.
\newblock Springer Science \& Business Media, 2008.

\bibitem{steinwart2006explicit}
{\sc Steinwart, I., Hush, D., and Scovel, C.}
\newblock {An explicit description of the reproducing kernel Hilbert spaces of Gaussian RBF kernels}.
\newblock {\em IEEE Transactions on Information Theory 52}, 10 (2006), 4635--4643.

\bibitem{tao2023topics}
{\sc Tao, T.}
\newblock {\em Topics in random matrix theory}, vol.~132.
\newblock American Mathematical Society, 2023.

\bibitem{tong1990fundamental}
{\sc Tong, Y.~L., and Tong, Y.}
\newblock {\em Fundamental properties and sampling distributions of the multivariate normal distribution}.
\newblock Springer, 1990.

\bibitem{villani1985another}
{\sc Villani, A.}
\newblock {Another note on the inclusion $L^p(\mu)\subset L^q(\mu)$}.
\newblock {\em The American Mathematical Monthly 92}, 7 (1985), 485--C76.

\bibitem{williams1991probability}
{\sc Williams, D.}
\newblock {\em {Probability with Martingales}}.
\newblock Cambridge University Press, 1991.

\bibitem{williams2015data}
{\sc Williams, M.~O., Kevrekidis, I.~G., and Rowley, C.~W.}
\newblock {A data--driven approximation of the Koopman operator: Extending Dynamic Mode Decomposition}.
\newblock {\em Journal of Nonlinear Science 25\/} (2015), 1307--1346.

\bibitem{MatthewO.Williams2015JournalofComputationalDynamics}
{\sc Williams, M.~O., Rowley, C.~W., and Kevrekidis, I.~G.}
\newblock {A kernel-based method for data-driven Koopman spectral analysis}.
\newblock {\em {Journal of Computational Dynamics} 2}, 2 (2015), 247--265.

\bibitem{zhao2015invertible}
{\sc Zhao, L.}
\newblock {Invertible Weighted Composition Operators on the Fock Space of} $\mathbb{C}^{N}$.
\newblock {\em Journal of Function Spaces 2015\/} (2015).

\bibitem{zhao2016analog}
{\sc Zhao, Z., and Giannakis, D.}
\newblock {Analog forecasting with dynamics-adapted kernels}.
\newblock {\em Nonlinearity 29}, 9 (2016), 2888.

\bibitem{zhu2005spaces}
{\sc Zhu, K.}
\newblock {\em Spaces of Holomorphic Functions in the Unit Ball}, vol.~226.
\newblock Springer, 2005.

\bibitem{zhu2012analysis}
{\sc Zhu, K.}
\newblock {\em {Analysis on Fock Spaces}}, vol.~263.
\newblock Springer Science \& Business Media, 2012.

\bibitem{zirnbauer1996riemannian}
{\sc Zirnbauer, M.~R.}
\newblock {Riemannian symmetric superspaces and their origin in random-matrix theory}.
\newblock {\em Journal of Mathematical Physics 37}, 10 (1996), 4986--5018.

\end{thebibliography}
\end{document}